\documentclass[reqno]{amsart}
\usepackage{amsmath, amssymb, amsthm, mathtools, amsfonts, bbm}
\usepackage[utf8]{inputenc}
\usepackage[all]{xy}
\usepackage{graphicx, epsfig, multirow}
\usepackage{enumerate, xcolor, color, enumitem}
\usepackage[colorlinks,allcolors=blue]{hyperref} 
\usepackage{cleveref}
\usepackage{appendix,pdfsync}
\usepackage{float}
\usepackage{setspace}
\usepackage{xpatch}
\usepackage{subcaption}
\usepackage{pst-node}
\usepackage{tikz-cd}
\usetikzlibrary{positioning}
\usetikzlibrary{shapes,arrows}
\usepackage{tikz}

\newtheorem{theorem}{Theorem}[section]
\newtheorem{proposition}[theorem]{Proposition}

\newtheorem{corollary}[theorem]{Corollary}
\newtheorem{lemma}[theorem]{Lemma}
\newtheorem{claim}[theorem]{Claim}
\newtheorem{question}[theorem]{Question}
\newtheorem{remark}[theorem]{Remark}

\theoremstyle{definition}
\newtheorem{definition}[theorem]{Definition}

\AddToHook{env/proposition/begin}{\crefalias{theorem}{proposition}}
\AddToHook{env/proposition/end}{\crefalias{proposition}{theorem}}
\AddToHook{env/conjecture/begin}{\crefalias{theorem}{conjecture}}
\AddToHook{env/conjecture/end}{\crefalias{conjecture}{theorem}}
\AddToHook{env/corollary/begin}{\crefalias{theorem}{corollary}}
\AddToHook{env/corollary/end}{\crefalias{corollary}{theorem}}
\AddToHook{env/lemma/begin}{\crefalias{theorem}{lemma}}
\AddToHook{env/lemma/end}{\crefalias{lemma}{theorem}}
\AddToHook{env/claim/begin}{\crefalias{theorem}{claim}}
\AddToHook{env/claim/end}{\crefalias{claim}{theorem}}
\AddToHook{env/question/begin}{\crefalias{theorem}{question}}
\AddToHook{env/question/end}{\crefalias{question}{theorem}}
\AddToHook{env/remark/begin}{\crefalias{theorem}{remark}}
\AddToHook{env/remark/end}{\crefalias{remark}{theorem}}
\AddToHook{env/definition/begin}{\crefalias{theorem}{definition}}
\AddToHook{env/definition/end}{\crefalias{definition}{theorem}}
\AddToHook{env/example/begin}{\crefalias{theorem}{example}}
\AddToHook{env/example/end}{\crefalias{example}{theorem}}

\def\os{<_{\Omega}} 
\def\s{\sigma}
\def\om{\omega}
\def\al{\alpha}
\def\bt{\beta}
\def\gm{\gamma}
\def\Sg{\Sigma}
\newcommand{\cn}[1]{C_n^p[#1^c]}
\newcommand{\Frs}[2]{{F_s({#1},{#2})}}
\newcommand{\Fuv}[2]{{F({#1},{#2})}}
\newcommand{\md}{M(\Delta_3(C_n^p))}
\newcommand{\mcl}[1]{\mathcal{#1}}
\newcommand{\mb}[1]{\mathbbm{#1}}
\newcommand{\X}[2]{$(\mcl{X}_{#1}^{#2})$}
\newcommand{\Cl}{\mathsf{Cl}}
\newcommand{\bK}{\mathbb{K}}

\begin{document}
	
	\title{Shellability of $3$-cut complexes of powers of cycle graphs}
	
	\author{Pratiksha Chauhan}
	\address{School of Mathematical and Statistical Sciences, IIT Mandi, India}
	\email{d22037@students.iitmandi.ac.in}
	
	\author{Samir Shukla}
	\address{School of Mathematical and Statistical Sciences, IIT Mandi, India}
	\email{samir@iitmandi.ac.in}
	
	
	\begin{abstract}
		In connection with commutative algebra, Bayer et al.\ introduced \textit{cut complexes} in [Topology of cut complexes of graphs, SIAM J.\ Discrete Math., 38(2):1630-1675, 2024].
		For a positive integer $k$, the $k$-cut complex of a graph $G$, denoted as $\Delta_k(G)$, is the simplicial complex whose facets are the $(|V(G)|-k)$-subsets $\sigma$ of the vertex set $V(G)$ of $G$ such that the induced subgraph $G[V(G) \setminus \sigma]$ is disconnected. Let $C_n^p$ denote the $p$-th power graph of the cycle graph $C_n$ on $n$ vertices.
		In this article, we show that $\Delta_3(C_n^p)$ is shellable for $n \geq 6p-3$, and therefore these complexes are homotopy equivalent to a wedge of spheres of dimension $n-4$. We provide an explicit shelling order on the facets of $\Delta_3(C_n^p)$.  We also characterize and count the number of spanning facets in this shelling order, and determine the number of spheres appearing in the wedge in the homotopy type of $\Delta_3(C_n^p)$. 
	\end{abstract}
	
	\maketitle
	
	\section{Introduction}
	
	In this article, all graphs are assumed to be finite and simple (that is, without loops or multiple edges). The vertex set and edge set of a graph $G$ are denoted by $V(G)$ and $E(G)$, respectively.
	
	A graph complex is a simplicial complex associated to a graph, where the simplices are determined by using certain combinatorial properties of the graph. In recent years, investigating the topological properties of graph complexes has become an increasingly active area of research. 
	Numerous graph complexes have been introduced and extensively studied, including the neighborhood complex \cite{Bjorner2003,Lovasz1978,Samir2019}, independence complex \cite{Barmak2013,SSA2021, Meshulam_domination}, clique complex \cite{Aharoni_Berger_Meshulam, Clique_line_2022,Kahle_Flag}, and matching complex \cite{Bjorner1994,Wachs2001,Wachs2007}. 
	These complexes establish a connection between topology and graph theory. In particular, the combinatorial properties of the underlying graphs can be investigated through topological invariants of these complexes, such as homotopy type, Betti numbers, homology groups, topological connectivity, etc. For further details on graph complexes, we refer the reader to \cite{JonssonBook,Kozlov2008}. 
	
	Recently, Bayer et al.\ \cite{Bayer2024Cutcomplex}  introduced a new family of graph complexes, called cut complexes. For $k \ge 1$, the {\it $k$-cut complex} of a graph $G$, denoted as $\Delta_k(G)$, is the simplicial complex whose facets (maximal simplices) are $\s \subseteq V(G)$ such that $|\s| = |V(G)|-k$ (where $|\cdot |$ denotes cardinality) and the induced subgraph $G[V(G) \setminus \s]$ is disconnected.   
	One of the principal motivations for studying cut complexes arises from a celebrated theorem of  Ralf Fr{\"o}berg connecting commutative algebra and graph theory through topology (see \Cref{thm:Fr}). For more  on connections  of graph complexes  with commutative algebra, see \cite{DochtermannEngstrom2009, ReinerRoberts2000, Woodroofe2014}.

	Let $\Delta$ be a simplicial complex on the vertex set $V(\Delta)= \{v_1, v_2, \ldots, v_n \}$ and let $\bK$ be a field. The Stanley-Reisner ideal  $I_{\Delta}$ of $\Delta$ is the ideal of the polynomial ring $\bK[x_1, x_2,\ldots, x_n]$ generated by the monomials corresponding to  minimal  subsets of $V(\Delta)$, which are not simplices  of $\Delta$, \textit{i.e.}, $I_{\Delta} = \langle x_{i_1}x_{i_2}\ldots x_{i_k} : \{v_{i_1},v_{i_2}, \ldots, v_{i_k}\} \notin \Delta  \rangle$. The Stanley-Reisner ring  $\bK[\Delta]$ is the quotient ring  $\bK[x_1, \ldots, x_n]/I_{\Delta}$. For more details, we refer the reader to  \cite{Eagon1998}.
	
	The Alexander dual $\Delta^{\vee}$ of the simplicial complex $\Delta$ is the simplicial complex on the vertex set $V(\Delta)$, whose simplices are the subsets of $V(\Delta)$ such that their complements are not simplices of $\Delta$, \textit{i.e.}, 
	$$
	\Delta^{\vee} = \{\sigma \subset V(\Delta): V(\Delta) \setminus \sigma \notin \Delta\}.
	$$
	
	The clique complex  $\mathsf{Cl}(G)$ of a graph $G$ is the simplicial complex whose simplices are $\sigma \subseteq V(G)$ such that the induced subgraph $G[\sigma]$ is a complete graph. It is easy to check that the  Stanley-Reisner ideal  $I_{\Cl(G)}$ of $\Cl(G)$,  is generated by quadratic square-free monomials.

	\begin{theorem} [{\cite[p.\ 274]{Eagon1998}, \cite[Theorem 1]{Froberg1990}}] \label{thm:Fr}
		A Stanley–Reisner ideal $I_{\Delta}$ generated by quadratic square-free monomials has a $2$-linear resolution if and only if $\Delta$ is the clique complex $\Cl(G)$ of a chordal graph $G$.
	\end{theorem}
	
	The following can be inferred from \cite[Proposition 8]{Eagon1998}.
	
	\begin{theorem}[\cite{Eagon1998}] \label{theorem:equivalence} 
		$G$ is chordal if and only if   $\Cl(G)^{\vee}$  is shellable. 
	\end{theorem}
	For the definition of a shellable complex, see \Cref{definition:shellability}.
	Observe that  $\Cl(G)^{\vee} = \Delta_2(G)$ for any graph $G$.  Therefore,  \Cref{theorem:equivalence} implies the following.
	\begin{equation*}\label{equation:equivalence}
		G  \ \text{is chordal} \Longleftrightarrow  \Delta_2(G) \  \text{is shellable}.
	\end{equation*}
	
	The above equivalence shows that 
	shellability of the $2$-cut complex characterizes chordal graphs. 
	This connection strongly motivates the systematic study 
	of higher cut complexes $\Delta_k(G)$ as natural generalizations of 
	$\Delta_2(G)$.
	
	Shellability itself is a  well-known  concept in topological combinatorics that provides deep insights into the combinatorial structure, topological properties, and algebraic invariants of simplicial complexes. One important consequence is that shellable complexes are homotopy equivalent to a wedge of spheres.
	Shellability  has proven useful in many areas, including polytope theory \cite{Bruggesser_Mani_1972,McMullen_1970},  poset theory \cite{Bruhat_order_1982,Bjorner_lexicographically_1983}, combinatorial topology \cite{Gerrit_chain_2015,Zhu_Turkish_2016}, topological combinatorics \cite{Higher_Independence_2025,Ziegler_chessboard_1994}, algebraic combinatorics \cite{Bjorner_Coxeter_1984,Meulen_2017}, commutative algebra \cite{Bjorner_Cohen_1980,Tuyl_2008}, etc. Determining whether a given simplicial complex is shellable is an important and active research direction in topological combinatorics.
	
	The shellability of cut complexes has recently attracted considerable attention.
	Bayer et al.\ \cite{Bayer2024Cutcomplex} determined shellability criteria for cut complexes of complete bipartite graphs and multipartite graphs, and showed that, for a chordal graph $G$, the  3-cut complex $\Delta_3(G)$ is shellable.  
	Subsequent works established shellability results for various structured graph families, including squared path graphs \cite{bayer2025topology}, $2 \times n$ grid graphs \cite{chandrakar2024topology} and hexagonal grids \cite{chandrakar2026hexagonal}.

	In this article, we consider the shellability of $3$-cut complexes of $C_n^p$, the $p$-th power of the cycle graph $C_n$ on $n$ vertices (see \Cref{defn:graph_power} for power of graphs).
	
	\begin{definition}\label{defn:graph_power}
		For a graph $G$ and a positive integer $p$, the $p$-th power graph of $G$ is a graph $G^p$ with $V(G^p)=V(G)$ and $\{u, v\}\in E(G^p)$ if and only if there exists a path between $u$ and $v$ of length at most $p$ in $G$ (here, the length of a path is the number of edges in the path). Clearly, $G^1 = G$.
	\end{definition}
	
	The graphs $C_n^p$ are also a family of circulant graphs, the Cayley graph of the cyclic group of order $n$ (see \Cref{definintion:Cayley} for the definitions of circulant graphs and Cayley graphs). 
	
	Powers of cycles have already been  studied in the context of  graph complexes. Adamaszek described the homotopy types of independence  complexes of $C_n^p$ \cite{Adamaszek_power_of_cyles} and also investigated their clique complexes \cite{Adamaszek_2013}.
	In \cite{Shukla_Chauhan_Vinayak_2025}, the authors determined the homotopy type of $2$-cut complexes of $C_n^p$ for $n=2p+2$ and $n\ge 3p+1$. Further, Bravo \cite{Andres_total_cut_dual} extended the study of $2$-cut complexes to all powers of cycle graphs and proved that $\Delta_2(C_n^p)$ is homotopy equivalent to a wedge of spheres. However, the topology of $k$-cut complexes of graph powers remains largely unexplored for $k \geq 3$.
	
	The study of shellability of cut complexes of powers of cycle graphs was initiated by Bayer et al.\ in \cite{Bayer2024Cutcomplex} and \cite{Bayer2024TotalCutcomplex}. It was shown that the $2$-cut complexes of the cycle graphs $C_n$ (\cite[Theorem 4.15]{Bayer2024TotalCutcomplex}) and the squared cycle graphs $C_n^2$ (\cite[Proposition 4.19]{Bayer2024TotalCutcomplex}) are not shellable. 
	However,  $\Delta_{k}(C_n)$ is shellable for $k \geq 3$ (\cite[Proposition 7.11]{Bayer2024Cutcomplex}). Motivated by computational evidence, Bayer et al.\ conjectured that $\Delta_k(C_n^2)$ is shellable for $k \geq 3$ and  $n \geq k+6$  (\cite[Conjecture 7.25]{Bayer2024Cutcomplex}). This conjecture was verified for $k=3$ in  \cite{Shellability2025}, where the authors  proved that   $\Delta_{3}(C_n^2)$ is shellable for $n\geq  9$. 
	
	In this article, we extend these results to all powers of cycle graphs for $3$-cut complexes. We prove the following.
	
	\begin{theorem}\label{theorem:main}
		Let $p\ge2$ and $n\ge 6p-3$. Then $\Delta_3(C_n^p)$ is shellable. Moreover, 
		$$\Delta_3(C_n^p) \simeq \bigvee\limits_{\binom{n-2p}{2}-(2p^2+p-1)} \mathbb{S}^{n-4}.$$
	\end{theorem}
	
	To prove \Cref{theorem:main}, we define an order on the facets of $\Delta_3(C_n^p)$ and show that this order is a shelling order. Then, we characterize and count the number of spanning facets in this shelling order to determine the number of spheres appearing in the wedge in the homotopy type of $\Delta_3(C_n^p)$. 
	
	This paper is organized as follows: \Cref{section:prel} introduces preliminaries from graph theory and simplicial complexes. \Cref{section:mainproof} presents the proof of \Cref{theorem:main}; we first define an order on the facets of $\Delta_3(C_n^p)$ and then proceed in two subsections. In \Cref{subsection:shelling order}, we show that this order is a shelling order. In \Cref{subsection:spanning facets}, we determine the  spanning facets for the shelling order. Finally, \Cref{section:future_directions} discusses the conclusions and future directions.
	
	\section{Preliminaries} \label{section:prel}
	
	In this section, we recall some basic definitions and results used in this article. 
	
	\subsection{Graph}\label{subsection:graph}
	A \textit{graph} $G$ is a pair $(V(G), E(G))$, where $V(G)$ is the set of vertices of $G$ and $E(G) \subseteq \binom{V(G)}{2}$
	denotes the set of edges. For any $u,v\in V(G)$, we say that $u$ and $v$ are adjacent if $\{u,v\} \in E(G)$. We write $u\sim v$ for adjacency and $u\nsim v$ for non-adjacency. 
	A {\it subgraph} $H$ of $G$ is a graph with $V(H) \subseteq V(G)$ and $E(H) \subseteq E(G)$. For a subset $U \subseteq V(G)$, the \textit{induced subgraph} $G[U]$ is the subgraph with $V(G[U]) = U$ and $E(G[U]) = \{\{a, b\} \in E(G) \ | \ a, b \in U\}$.
	
	For $u,v\in V(G)$, a {\it path} from $u$ to $v$ is a sequence of distinct vertices $u=v_0,v_1, \ldots, v_n=v$ such that $v_i \sim v_{i+1}$ for all $0 \le i\le n-1$. The {\it length} of a path is the number of edges in the path. 
	A graph is \textit{connected} if there exists a path between each pair of its vertices; otherwise, it is \textit{disconnected}. The graph with the empty set $\emptyset$ as its set of vertices is considered connected.
	
	\begin{definition}\label{definintion:Cayley}
		Let $\Gamma$ be a group and let $S\subset\Gamma$ be a subset not containing the identity element.
		The \textit{Cayley graph}  of $\Gamma$ with respect to $S$ is the graph $G(\Gamma,S)$ having vertex set $\Gamma$, and $\{u, v\}\in E(G(\Gamma,S))$ if and only if $uv^{-1} \in S \cup S^{-1}$, where $S^{-1} = \{x^{-1} \ | \ x \in S\}$.
		Let  $\mathbb{Z}_n$ denotes  the cyclic group of order $n$. For $n\ge 2$ and $S \subset\mathbb{Z}_n$ such that $0 \notin S$,  the   Cayley graph $G(\mathbb{Z}_n,S)$ is called the \textit{circulant graph}  on the generating set $S$.  
	\end{definition}
	
	For $n\ge 3$, let $C_n$ denote the cycle graph on $n$ vertices $\{0,1,\ldots,n-1\}$. Observe that for $p \ge 1$, the $p$-th power $C_n^p$ of the cycle graph is a graph with $V(C_n^p) = V(C_n)$ and $E(C_n^p) =\{ \{i,i+j\ (\text{mod $n$})\}\,|\,0\le i\le n-1 $ and $1\le j\le p$\}. 
	It is easy to check that for $n\ge 2p+1$, $C_n^p$ is a circulant graph on the generating set $\{1, 2,\ldots, p\}$.
	
	We refer the reader to \cite{bondy1976graph} and \cite{west} for more details about the graphs.
	
	\subsection{Simplicial complex}
	
	A {\it finite abstract simplicial complex} $\Delta$ is a collection of finite sets such that if $\tau \in\Delta$ and $\s \subset \tau$, then $\s \in\Delta$. 
	The elements of $\Delta$ are called {\it simplices} of $\Delta$. If $\s \subset \tau$, we say that $\s$ is a {\it face} of $\tau$. The \textit{dimension of a simplex} $\s$ is equal to $|\s|-1$. The \textit{dimension of an abstract simplicial complex} is the maximum of the dimensions of its simplices. 
	If a simplex has dimension $d$, it is said to be $d$-{\it dimensional}. The $0$-dimensional simplices are called \textit{vertices} of $\Delta$. An abstract simplicial complex which is an empty collection of sets is called the \textit{void} abstract simplicial complex, and is denoted by $\emptyset$. The {\it boundary} of a $d$-dimensional simplex $\sigma $ is the simplicial complex, consisting of  all  faces of $\sigma$ of dimension $\le d-1$ and it is denoted by $Bd(\sigma).$
	
	A simplex that is not a face of any other simplex is called a {\it maximal simplex} or \textit{facet}. The set of maximal simplices of $\Delta$ is denoted by $M(\Delta)$. 
	A simplicial complex is called {\it pure $d$-dimensional}, if all of its maximal simplices are of dimension $d$. A \textit{subcomplex} $\Delta'$ of $\Delta$ is a simplicial complex such that $\sigma\in\Delta'$ implies $\sigma\in\Delta$. 
	
	In this article, we consider any simplicial complex as a topological space, namely, its geometric realization (see \cite{Kozlov2008} for details). For terminologies related to algebraic topology, we refer to \cite{hatcher2005algebraic}.

	\begin{definition}[{\cite[Definition 12.1]{Kozlov2008}}]\label{definition:shellability} 
		A simplicial complex $\Delta$ is called \textit{shellable} if its facets can be arranged in a linear order $F_1, F_2, \ldots, F_j$ in such a way that the subcomplex $(\bigcup_{r=1}^{s-1} F_r) \cap F_s$ is pure and of dimension $(\dim F_s-1)$ for all $2\le s\le j$. Such an ordering of facets is called a \textit{shelling order}.
	\end{definition} 
	
	In other words, a simplicial complex $\Delta$ has a \textit{shelling order} $F_1, F_2, \ldots, F_j$ of its facets if and only if for any $r, s$ satisfying $1\le r<s\le j$, there exists $1\le t<s$ such that $F_t \cap F_s=F_s\setminus\{u\}$ for some $u\in F_s\setminus F_r$.
	
	A facet $F_i$ $(1<i \le j)$  is called {\it spanning} with respect to the given shelling order if $Bd(F_i) \subseteq \bigcup\limits_{l=1}^{i-1} F_l$.
	It is easy to check that $F_i$ is a spanning facet if for each $u\in F_i$, there exists $1\le t<i$ such that $F_t \cap F_i=F_i \setminus\{u\}$.
	
	The following theorem can be inferred from \cite[Theorem 12.3]{Kozlov2008}.
	
	\begin{theorem}\label{theorem:wedge}
		Let $\Delta$ be a pure shellable simplicial complex of dimension $d$. Then $\Delta$ has the homotopy type of a wedge of $\bt$ spheres of dimension  $d$, where $\bt$ is the number of total spanning facets in a given shelling. Hence $\Delta \simeq \bigvee\limits_{\bt} \mathbb{S}^{d}.$
	\end{theorem}
	
	\section{Proof of \Cref{theorem:main}}\label{section:mainproof}
	
	To prove \Cref{theorem:main}, we construct an order on the facets of $\Delta_3(C_n^p)$ and show that it is a shelling order. We then identify and count all spanning facets in this shelling order to determine the exact number of spheres in the wedge in the homotopy type of $\Delta_3(C_n^p)$. A particularly notable aspect of this order is that, for the $p$-th power of the cycle graph, the set of all facets, $\md$, is partitioned into exactly $p$ ordered subsets. 
	These subsets are denoted by $\mcl{M}_{\al}$ with $\al\in\{0,1,\ldots,p-1\}$, which will be described in detail later. The shelling order is then constructed by arranging the sets $\mcl{M}_{\al}$ such that $\mcl{M}_{\al-1}$ precedes $\mcl{M}_{\al}$ for all $\al\in\{1,2,\ldots,p-1\}$.
	
	Throughout the section, we fix $p\ge2$ and $n\ge 6p-3$. Before proceeding to the main proof, we first introduce some notation and establish important results that will be used later. Define		
	$$\mb{c}:=\begin{cases}
		\frac{n+1}{2} & \text{ if $n$ is odd},\\
		\phantom{1}\frac{n}{2} & \text{ if $n$ is even}.
	\end{cases} $$
	
	\begin{proposition}\label{proposition: c and p relation} We have the following.
		\begin{enumerate}[label=(\roman*)]
			\item \label{range of c} $p<3p-1\le\mb{c}\le n-3p+2<n-p$.
			
			\item \label{c-p/2<=n-2p-1} $\mb{c}-\frac{p}{2}\le n-2p-1$.
			
			\item \label{c-(bt+1)/2<=n-2p-1} Let $\bt\ge 1$. Then $\mb{c}-\frac{\bt+1}{2}\le n-2p-1$. Moreover, if $\mb{c}-\frac{\bt+1}{2}=n-2p-1$, then $p=2$. For any integer $z$, if $z\ge\mb{c}+\frac{\bt}{2}$, then $2\mb{c}-z\le n-2p-1$.
			
			\item \label{<=n-p} $\mb{c}+\frac{3p-2}{2}\le n-p$. 
			
			\item \label{>=p} $\mb{c}-\frac{3p}{2}\ge p$.
		\end{enumerate}
	\end{proposition}
	
	\begin{proof}
		\begin{enumerate}[label=(\roman*)]
			\item By the definition of $\mb{c}$, we get $\frac{n}{2}\le\mb{c}\le\frac{n+1}{2}$. We have $n\ge 6p-3$. If $n=6p-3$ (odd), then $\mb{c}=3p-1$; and if $n\ge 6p-2$, then $\mb{c}\ge\frac{n}{2}$ implies that $\mb{c}\ge 3p-1$. Hence $\mb{c}\ge3p-1$. 
			It follows that $3p-1\le2\mb{c}-\mb{c}\le n+1-\mb{c} $, and thus $\mb{c}\le n-3p+2$. Therefore $3p-1\le\mb{c}\le n-3p+2$. Moreover, since $p\ge2$, we get $p<3p-1$ and $n-3p+2<n-p$. Thus $p<3p-1\le\mb{c}\le n-3p+2<n-p$. 
			
			\item By \ref{range of c}, $\mb{c}\le n-3p+2$. This implies $\mb{c}-\frac{p}{2}\le n-\frac{7p}{2}+2$. Since $p\ge2$, we get $\mb{c}-\frac{p}{2}\le n-2p-1$.
			
			\item We have $\bt\ge 1$. Since $\mb{c}\le n-3p+2$, it follows that $\mb{c}-\frac{\bt+1}{2}\le n-3p+1.$ Using $p\ge2$, we get $\mb{c}-\frac{\bt+1}{2}\le n-2p-1$.
			
			Suppose $\mb{c}-\frac{\bt+1}{2}=n-2p-1$. Then $\mb{c}\le\frac{n+1}{2}$ and $\bt\ge 1$ imply $n\le 4p+1$. Since $n\ge 6p-3$, we get $p\le2$. Hence $p=2$ (as $p\ge2$).
			
			Now, suppose $z\ge\mb{c}+\frac{\bt}{2}$. Then $2\mb{c}-z\le\mb{c}-\frac{\bt}{2}\le n-2p-\frac{1}{2}$. Since $n$, $p$, $\mb{c}$ and $z$ are all integers, we obtain $2\mb{c}-z\le n-2p-1$.
			
			\item Since $\mb{c}\le n-3p+2$, we get $ \mb{c}+\frac{3p-2}{2}\le n-\frac{3p-2}{2}=n-p-\frac{p-2}{2}$. Then $p\ge2$ implies that $ \mb{c}+\frac{3p-2}{2}\le n-p$.
			
			\item Since $\mb{c}\ge 3p-1$ and $p\ge2$, it follows that $\mb{c}-\frac{3p}{2}\ge\frac{3p}{2}-1=p+\frac{p}{2}-1\ge p$. 
		\end{enumerate}
		\vspace{-0.2cm}
	\end{proof}
	
	We now define an ordered set $\Omega:=(\om_1,\om_2,\ldots,\om_n)$ by arranging the elements of $V(C_n^p)=\{0,1,2,\ldots,n-1\}$ such that, for all $ i \in\{1,2,\ldots,n\}$, the $i$-th element is given by $$\om_i=\mb{c}+(-1)^{i-1} \lfloor i/2 \rfloor \ ( \text{mod} \ n),$$ where $\lfloor i/2 \rfloor$ denotes the greatest integer less than or equal to $i/2$. 
	Observe that $$\Omega=\begin{cases}
		(\mb{c},\mb{c}-1, \mb{c}+1,\mb{c}-2, \mb{c}+2, \dots,n-2,2,n-1,1, 0), &\text{ if $n$ is odd},\\
		(\mb{c},\mb{c}-1, \mb{c}+1,\mb{c}-2, \mb{c}+2,\dots,2,n-2,1,n-1, 0), & \text{ if $n$ is even}.
	\end{cases} $$
	
	Using the ordered set $\Omega$, we define an order $\os$ on the elements of $V(C_n^p)$ as follows: for any $u,v\in V(C_n^p)$, we say that $u\os v$ if and only if $u=\om_i$ and $v=\om_j$ for some $i,j\in\{1,2,\ldots,n\}$ with $i<j$.
	
	\begin{remark}[{\cite[Remark 3.1]{Shellability2025}}]\label{remark:order in S}
		Let $u,v\in V(C_n^p)$. 
		\begin{enumerate}[label=(\roman*)]
			\item \label{part:i} If $u<\mb{c}$, then $u\os v$ if and only if either $v<u$ or $v\ge2\mb{c}-u$.
			
			\item \label{part:ii} If $u\ge\mb{c} $, then $u\os v$ if and only if either $v<2\mb{c}-u$ or $v>u$. 
			
			\item \label{part:iii} If $v<\mb{c}$, then $u\os v$ if and only if $v<u<2\mb{c}-v$. 
			
			\item \label{part:iv} If $v>\mb{c}$, then $u\os v$ if and only if $2\mb{c}-v\le u<v$. 
		\end{enumerate}			
	\end{remark} 
	
	The complement of a set $X \subseteq V(C_n^p)$ is denoted by $X^c$ throughout the proof. For a positive integer $m$, let $[m]=\{1,2,\ldots,m\}$ and $[0,m]=[m]\cup\{0\}$.
	
	Let $\mcl{A}=\{A \subseteq V(C_n^p) : |A|=n-3 \}$. For $i\in[n-2]$, define the subsets $\mcl{A}_i$ of $\mcl{A}$ as follows:
	\begin{equation}\label{equation:Ai}
		\mcl{A}_i:=\begin{cases}
			\{A \in\mcl{A}: \om_1\in A^c \}, &\text{if $i=1$,}\\
			\{A \in\mcl{A}: \om_i\in A^c \ \text{and } \om_1,\om_2, \dots,\om_{i-1}\notin A^c \}, &\text{if $i>1$.}
		\end{cases}
	\end{equation}
	Observe that the sets $\mcl{A}_i$ are pairwise disjoint and form a partition of $\mcl{A}$. So $\mcl{A}=\bigsqcup_{i=1}^{n-2}\mcl{A}_i$.
	
	Let $A\in\mcl{A}$ such that $A\in\mcl{A}_i$ for some $i\in[n-2]$. By the definition of $\mcl{A}_i$, we have $\om_i\in A^c$. It follows that $A^c=\{\om_i,i_1, i_2\}$ for some $i_1,i_2\in V(C_n^p)$. 
	Throughout this section, we assume that if $A \in\mcl{A}_i$ and $A^c=\{\om_i\}\sqcup\{i_1,i_2\}$, then $i_1<i_2$. 
	
	\begin{remark}[{\cite[Remark 3.2]{Shellability2025}}]\label{remark:Aj}
		Let $A\in\mcl{A}$ such that $A\in\mcl{A}_i$ and $A^c=\{\om_i\}\sqcup\{i_1,i_2\}$. Then $\om_i\os i_1, i_2$ by the definition of $\mcl{A}_i$.
	\end{remark} 
	
	We know that each $A\in\mcl{A}$ is uniquely associated with subset $\mcl{A}_i$ for some $i$. Using this partitioning of $\mcl{A}$ into subsets $\mcl{A}_i$, we define an order $\ll$ on the elements of $\mcl{A}$. The following definition precisely formulates this order.
	\begin{definition}\label{definition:order_ll}
		Define an order $\ll$ on the elements of $\mcl{A}$ as follows: for $ A, A' \in\mcl{A}$, let $A\in\mcl{A}_i$ and $ A' \in\mcl{A}_j$ such that $A^c=\{\om_i\}\sqcup\{i_1,i_2\}$ and $ A'^c=\{\om_j\} \sqcup \{j_1, j_2\}$. Then $ A \ll A'$ if and only if any one of the following conditions is true: 
		\begin{enumerate}[label=(\roman*)]
			\item \label{om_i=om_j} $i=j$, \textit{i.e.}, $\om_i=\om_j$, and either $i_1<j_1$ or, $i_1=j_1$ and $i_2<j_2$.
			
			\item \label{om_i<om_j} $i<j$, \textit{i.e.}, $\om_i\os\om_j$.
		\end{enumerate} 
	\end{definition}
	It is easy to see that \Cref{definition:order_ll} defines a total order on $\mcl{A}$.
	
	The set of all the facets of $\Delta_3(C_n^p)$ is denoted by $\md$. By the definition of the $3$-cut complex, for any $F \in\md$, we have $|F|=n-3$ and the induced subgraph $\cn{F}$ is disconnected. Clearly, $\md \subset \mcl{A}=\bigsqcup_{i=1}^{n-2}\mcl{A}_i$. 
	Therefore, the order $\ll$ defined on $\mcl{A}$ can be naturally restricted to $\md$, resulting in a total order on the elements of $\md$. 
	
	To obtain the desired shelling order, we modify the order $\ll$. This is done by partitioning $\md$ into $p$ disjoint sets $\mcl{M}_{0}, \mcl{M}_1, \ldots, \mcl{M}_{p-1}$. 
	For $\al\in[p-1]$, define the subsets $\mcl{M}_{\al}$ (the subset $\mcl{M}_0$ is defined later) of $\md$ as follows:
	a facet $F \in\mcl{M}_{\al}$ if and only if $F\in\mcl{A}_i$ for some $i\in[n-2]$ such that $F^c=\{\om_i\}\sqcup\{i_1,i_2\}$, and $F$ satisfies one of the following conditions:		
	\begin{enumerate}[label=$(\mcl{X}_{\al}^{\arabic*})$]
		
		\item $i_1<i_2<\om_i$ with $\om_i<\mb{c}+\frac{p}{2}$ and $i_1=2\mb{c}-\om_i-p+\al-1$.
		
		\item $i_1<i_2<\om_i$ with $\om_i\ge\mb{c}+\frac{p}{2}$ and $i_1=\om_i-2p+\al-1$.
		
		\item $i_1<\om_i<i_2$ with $\om_i<\mb{c}-\frac{p}{2}$ and $i_2=i_1+2p-\al+1$.
		
		\item $i_1<\om_i<i_2$ with $\om_i\ge\mb{c}-\frac{p}{2}$, $\om_i-p+\al\le i_1\le2\mb{c}-\om_i-p-1$ and $i_2=i_1+2p-\al+1$.
		
		\item $i_1<\om_i<i_2$ with $\om_i\le\mb{c}-\frac{\al+1}{2}$, $i_2=2\mb{c}-\om_i+p-\al$ and $i_1\ge i_2-2p+\al$ ($\implies$ $i_1\ge2\mb{c}-\om_i-p$).
		
		\item $i_1<\om_i<i_2$ with $\om_i\ge\mb{c}+\frac{\al}{2}$, $i_1=2\mb{c}-\om_i-p+\al-1$ and $i_2\le i_1+2p-\al$ ($\implies$ $i_2\le2\mb{c}-\om_i+p-1$).
		
		\item $i_1<\om_i<i_2$ with $\om_i<\mb{c}+\frac{p}{2}$, $i_1=i_2-2p+\al-1$ and $2\mb{c}-\om_i+p\le i_2\le\om_i+p-\al$.
		
		\item $i_1<\om_i<i_2$ with $\om_i\ge\mb{c}+\frac{p}{2}$ and $i_1=i_2-2p+\al-1$.
		
		\item $\om_i<i_1<i_2$ with $\om_i\le n-2p-2$, $i_1\ge\om_i+p+1$ and $i_2=\om_i+2p-\al+1$.
		
		\item $\om_i<i_1<i_2$ with $\om_i>n-2p-2$, $i_1\ge\om_i+p+1$ and $i_2=n-\al-1$.
	\end{enumerate}
	
	We now show that the sets $\mcl{M}_{\al}$, where $\al\in[p-1]$, are pairwise disjoint. To do so, we first derive a few preliminary results. 
	Recall that for any $u, v \in V(C_n^p)$, the notation $u\sim v$ denotes that the vertices $u$ and $v$ are adjacent in $C_n^p$. This adjacency holds if and only if 
	$v=u \pm j \ (\text{mod $n$}) \text{ for some } j \in [p].$ 
	If $u$ and $v$ are not adjacent, then we write $u \nsim v$.
	
	\begin{proposition}\label{proposition:al_i>c}
		Let $F\in\md$ such that $F\in\mcl{A}_i$ and $F^c=\{\om_i\}\sqcup\{i_1,i_2\}$. If $i_1<i_2<\om_i$ and $i_1\ge2\mb{c}-\om_i-p$, then $i_2\le\om_i-p-1$.
	\end{proposition}
	
	\begin{proof}
		Since $F\in\md$, $\cn{F}$ is disconnected. Suppose $\om_i\le\mb{c} $. Then $2\mb{c}-\om_i\ge\om_i$, which implies $i_1\ge\om_i-p$. Hence $\om_i-p\le i_1<i_2<\om_i$. This means that $\cn{F}$ is connected, a contradiction. Therefore $\om_i>\mb{c}$.
		By \Cref{remark:Aj}, $\om_i\os i_2$. Since $i_2<w_i$, \Cref{remark:order in S}\ref{part:ii} gives $i_2<2\mb{c}-\om_i$. Then $2\mb{c}-\om_i-p\le i_1<i_2$ implies $i_2-p<i_1<i_2$. Hence $i_1\sim i_2$. Using the fact that $\cn{F}$ is disconnected, it follows that $i_2\nsim\om_i$, and therefore $i_2\le\om_i-p-1$.
	\end{proof}
	
	The following result provides additional information about a facet $F\in\mcl{M}_{\al}$ for some $\al\in[p-1]$.
	
	\begin{proposition} \label{proposition:M_al conditions}
		Let $F\in\md$ such that $F\in\mcl{A}_i$ and $F^c=\{\om_i\}\sqcup\{i_1,i_2\}$. Assume $F\in\mcl{M}_{\al}$ for some $\al\in[p-1]$. 
		\begin{enumerate}[label=(\roman*)]
			\item \label{ob_x1} If $F$ satisfies \X{\al}{1}, then $\mb{c}+\frac{\al+1}{2}\le\om_i<n-p$.
			
			\item \label{ob_x2} If $F$ satisfies \X{\al}{2}, then $\om_i<n-p$.
			
			\item \label{ob_x3} If $F$ satisfies \X{\al}{3}, then $i_1\ge2\mb{c}-\om_i-2p+\al-1$, $i_1\ge\om_i-p+\al$, $i_1>p$, $i_2<n-p$ and $i_1\sim\om_i$.
			
			\item \label{ob_x4} If $F$ satisfies \X{\al}{4}, then $\om_i\le\mb{c}-\frac{\al+1}{2}$, $i_1>p$, $i_2\le2\mb{c}-\om_i+p-\al$, $i_2\le n-p$ and $i_1\sim\om_i$.
			
			\item \label{ob_x5} If $F$ satisfies \X{\al}{5}, then $\om_i\ge\mb{c}-\frac{p-1}{2}$, $i_1>p$, $i_2<n-p$ and $i_1\sim\om_i$.
			
			\item \label{ob_x6} If $F$ satisfies \X{\al}{6}, then $\om_i\le\mb{c}+\frac{p-2}{2}$, $i_1>p$, $i_2\le\mb{c}-\frac{\al}{2}+p-1\le\om_i+p-\al-1$, $i_2<n-p$ and $i_2\sim\om_i$.
			
			\item \label{ob_x7} If $F$ satisfies \X{\al}{7}, then $\om_i\ge\mb{c}+\frac{\al}{2}$, $i_1\ge2\mb{c}-\om_i-p+\al-1$, $i_1>p$, $i_2<n-p$ and $i_2\sim\om_i$.
			
			\item \label{ob_x8} If $F$ satisfies \X{\al}{8}, then $i_1>p$, $i_2\le\om_i+p-\al<\om_i+p$, $i_2\le\mb{c}+\frac{3p}{2}-1$, $i_2\le n-p$ and $i_2\sim\om_i$.
			
			\item \label{ob_x9} If $F$ satisfies \X{\al}{9}, then $\om_i>2p$ and $i_1\sim i_2$.
			
			\item \label{ob_x10} If $F$ satisfies \X{\al}{10}, then $\om_i\ge\mb{c}$, $\om_i\ge2p$ and $i_1\sim i_2$.
		\end{enumerate}
	\end{proposition} 
	
	\begin{proof} 
		By \Cref{remark:Aj}, we have $\om_i\os i_1,i_2$.
		\begin{enumerate}[label=(\roman*)]
			\item Suppose $F$ satisfies \X{\al}{1}. Then $i_1<i_2<\om_i<\mb{c}+\frac{p}{2}$ and $i_1=2\mb{c}-\om_i-p+\al-1$. By \Cref{proposition:al_i>c}, $i_2\le\om_i-p-1$. 
			It follows that $2\mb{c}-\om_i-p+\al-1<\om_i-p-1$, which implies $\om_i\ge\mb{c}+\frac{\al+1}{2}$. Since $p\ge2$, we have $\mb{c}+\frac{p}{2}<\mb{c}+\frac{3p-2}{2}$. Moreover, $\mb{c}+\frac{3p-2}{2}\le n-p$ by \Cref{proposition: c and p relation}\ref{<=n-p}. Thus, $\om_i<\mb{c}+\frac{p}{2}$ implies $\om_i<n-p$.
			
			\item Suppose $F$ satisfies \X{\al}{2}. Then $i_1<i_2<\om_i$, $\om_i\ge\mb{c}+\frac{p}{2}$ and $i_1=\om_i-2p+\al-1$. Note that $\om_i>\mb{c}$. 
			Since $\om_i\os i_2$ and $i_2<\om_i$, we have $i_2<2\mb{c}-\om_i$ by \Cref{remark:order in S}\ref{part:ii}. This gives $\om_i-2p+\al-1=i_1<i_2<2\mb{c}-\om_i$, which implies $\om_i\le\mb{c}+\frac{2p-\al-1}{2}$. 
			Since $\al\ge 1$ and $p\ge2$, it follows that $\om_i\le\mb{c}+\frac{2p-2}{2}<\mb{c}+\frac{3p-2}{2}$. Hence $\om_i<n-p$ by \Cref{proposition: c and p relation}\ref{<=n-p}.
			
			\item Suppose $F$ satisfies \X{\al}{3}. Then $i_1<\om_i<i_2$, $\om_i<\mb{c}-\frac{p}{2}$ and $i_2=i_1+2p-\al+1$. Clearly, $\om_i<\mb{c}$. By \Cref{remark:order in S}\ref{part:i}, $\om_i\os i_2$ and $i_2>\om_i$ yield $i_2\ge2\mb{c}-\om_i$.
			Since $i_2=i_1+2p-\al+1$, this implies $i_1\ge2\mb{c}-\om_i-2p+\al-1$. Using $\om_i<\mb{c}-\frac{p}{2}$, we get $i_1>\mb{c}-\frac{3p}{2}+\al-1>\om_i-p+\al-1\ge\om_i-p$ (as $\al\ge 1$). This means that $\om_i-p<i_1<\om_i$, and hence $i_1\sim\om_i$.
			Moreover, $i_1>\mb{c}-\frac{3p}{2}+\al-1\ge\mb{c}-\frac{3p}{2}$. Since $\mb{c}-\frac{3p}{2}\ge p$ by \Cref{proposition: c and p relation}\ref{>=p}, we obtain $i_1>p$. 
			We have $i_2=i_1+2p-\al+1\le i_1+2p$ and $i_1<\om_i<\mb{c}-\frac{p}{2}$. It follows that $i_2<\mb{c}+\frac{3p-2}{2}$. Thus, $i_2<n-p$ by \Cref{proposition: c and p relation}\ref{<=n-p}.
			
			\item Suppose $F$ satisfies \X{\al}{4}. Then $i_1<\om_i<i_2$, $\om_i\ge\mb{c}-\frac{p}{2}$, $\om_i-p+\al\le i_1\le2\mb{c}-\om_i-p-1$ and $i_2=i_1+2p-\al+1$.
			Using $\om_i-p+\al\le i_1\le2\mb{c}-\om_i-p-1$, we obtain $\om_i\le\mb{c}-\frac{\al+1}{2}$. Note that $\om_i\ge\mb{c}-\frac{p}{2}$ and $i_1\ge\om_i-p+\al$ imply $i_1\ge\mb{c}-\frac{3p}{2}+\al>\mb{c}-\frac{3p}{2}$. 
			Therefore $i_1>p$ by \Cref{proposition: c and p relation}\ref{>=p}. Since $i_1\le2\mb{c}-\om_i-p-1$ and $i_2=i_1+2p-\al+1$, we get $i_2\le2\mb{c}-\om_i+p-\al\le2\mb{c}-\om_i+p-1$. Further, $\om_i\ge\mb{c}-\frac{p}{2}$ implies $i_2\le\mb{c}+\frac{3p-2}{2}$. 
			Hence $i_2\le n-p$ by \Cref{proposition: c and p relation}\ref{<=n-p}. Finally, $\om_i-p<\om_i-p+\al\le i_1<\om_i$ implies $i_1\sim\om_i$.
			
			\item Suppose $F$ satisfies \X{\al}{5}. Then $i_1<\om_i<i_2$, $\om_i\le\mb{c}-\frac{\al+1}{2}$, $i_2=2\mb{c}-\om_i+p-\al$ and $i_1\ge2\mb{c}-\om_i-p$. Since $\om_i<\mb{c}$, we have $\om_i<2\mb{c}-\om_i$.
			From $2\mb{c}-\om_i-p\le i_1<\om_i$, we obtain $\om_i\ge\mb{c}-\frac{p-1}{2}$ and $\om_i-p<i_1<\om_i$. It follows that $i_1\sim\om_i$ and $i_1>\mb{c}-\frac{3p-1}{2}$. By \Cref{proposition: c and p relation}\ref{>=p}, $i_1>p$. 
			Moreover, $i_2=2\mb{c}-\om_i+p-\al$ and $\om_i\ge\mb{c}-\frac{p-1}{2}$ imply $i_2\le\mb{c}+\frac{3p-1}{2}-\al\le\mb{c}+\frac{3p-3}{2}$. Thus $i_2<n-p$ by \Cref{proposition: c and p relation}\ref{<=n-p}. 
			
			\item Suppose $F$ satisfies \X{\al}{6}. Then $i_1<\om_i<i_2$, $\om_i\ge\mb{c}+\frac{\al}{2}$, $i_1=2\mb{c}-\om_i-p+\al-1$ and $i_2\le2\mb{c}-\om_i+p-1$. 
			Using $\om_i\ge\mb{c}+\frac{\al}{2}$ and $i_2\le2\mb{c}-\om_i+p-1$, we obtain $i_2\le\mb{c}-\frac{\al}{2}+p-1\le\om_i+p-\al-1$.
			Note that $\om_i>\mb{c}$ implies $2\mb{c}-\om_i<\om_i$. Since $\om_i<i_2\le2\mb{c}-\om_i+p-1$, we have $\om_i\le\mb{c}+\frac{p-2}{2}$ and $\om_i<i_2<\om_i+p$. This gives $i_2\sim\om_i$ and $i_2\le\mb{c}+\frac{3p-4}{2}$. Hence $i_2<n-p$ by \Cref{proposition: c and p relation}\ref{<=n-p}. 
			Now, since $i_1=2\mb{c}-\om_i-p+\al-1$ and $\om_i\le\mb{c}+\frac{p-2}{2}$, we get $i_1\ge\mb{c}-\frac{3p-2}{2}+\al-1\ge\mb{c}-\frac{3p-2}{2}$. By \Cref{proposition: c and p relation}\ref{>=p}, $i_1>p$.
			
			\item Suppose $F$ satisfies \X{\al}{7}. Then $i_1<\om_i<i_2$, $\om_i<\mb{c}+\frac{p}{2}$, $i_1=i_2-2p+\al-1$ and $2\mb{c}-\om_i+p\le i_2\le\om_i+p-\al$. From the bounds on $i_2$, we obtain $\om_i\ge\mb{c}+\frac{\al}{2}$. 
			Moreover, $\om_i<\mb{c}+\frac{p}{2}$ and $i_2\le\om_i+p-\al$ imply $i_2<\mb{c}+\frac{3p}{2}-\al\le\mb{c}+\frac{3p-2}{2}$. Therefore $i_2<n-p$ by \Cref{proposition: c and p relation}\ref{<=n-p}.
			Since $i_2\ge2\mb{c}-\om_i+p$ and $i_1=i_2-2p+\al-1$, we get $i_1\ge2\mb{c}-\om_i-p+\al-1\ge2\mb{c}-\om_i-p$. Further, $\om_i<\mb{c}+\frac{p}{2}$ implies $i_1>\mb{c}-\frac{3p}{2}$. By \Cref{proposition: c and p relation}\ref{>=p}, $i_1>p$. Now, $\om_i<i_2\le\om_i+p-\al<\om_i+p$ implies $i_2\sim\om_i$.
			
			\item Suppose $F$ satisfies \X{\al}{8}. Then $i_1<\om_i<i_2$, $\om_i\ge\mb{c}+\frac{p}{2}$ and $i_1=i_2-2p+\al-1$. Clearly, $\om_i>\mb{c}$. Since $\om_i\os i_1$ and $i_1<\om_i$, \Cref{remark:order in S}\ref{part:ii} implies $i_1<2\mb{c}-\om_i$. Using $i_1=i_2-2p+\al-1$, we obtain $i_2\le2\mb{c}-\om_i+2p-\al$. 
			Further, $\om_i\ge\mb{c}+\frac{p}{2}$ gives $i_2\le\mb{c}+\frac{3p}{2}-\al\le\om_i+p-\al<\om_i+p$. Since $\om_i<i_2$, we get $i_2\sim\om_i$. 
			Moreover, $i_2\le\mb{c}+\frac{3p}{2}-\al\le\mb{c}+\frac{3p-2}{2}$. By \Cref{proposition: c and p relation}\ref{<=n-p}, $i_2\le n-p$.
			Now, $i_1=i_2-2p+\al-1\ge i_2-2p$ and $\mb{c}+\frac{p}{2}\le\om_i<i_2$ imply $i_1\ge\mb{c}-\frac{3p}{2}+1$. Hence $i_1>p$ by \Cref{proposition: c and p relation}\ref{>=p}.
			
			\item Suppose $F$ satisfies \X{\al}{9}. Then $\om_i<i_1<i_2$, $\om_i\le n-2p-2$, $i_1\ge\om_i+p+1$ and $i_2=\om_i+2p-\al+1$. Since $i_1<i_2=\om_i+2p-\al+1\le\om_i+2p<i_1+p$, we have $i_1\sim i_2$. By \Cref{proposition: c and p relation}\ref{range of c}, $\mb{c}\ge 3p-1$. If $\om_i\ge\mb{c} $, then $\om_i>2p$ (as $p\ge2$). Otherwise, if $\om_i<\mb{c}$, then $\om_i\os i_1$ and $i_1>\om_i$ imply $i_1\ge2\mb{c}-\om_i$ by \Cref{remark:order in S}\ref{part:i}. Since $i_1<i_2=\om_i+2p-\al+1$, it follows that $\om_i\ge\mb{c}-\frac{2p-\al}{2}\ge 3p-1-\frac{2p-1}{2}>2p$. Hence, in either case, $\om_i>2p$.
			
			\item Suppose $F$ satisfies \X{\al}{10}. Then $\om_i<i_1<i_2$, $\om_i>n-2p-2$, $i_1\ge\om_i+p+1$ and $i_2=n-\al-1$. Since $n\ge 6p-3$ and $p\ge2$, it follows that $\om_i\ge n-2p-1\ge 4p-4\ge2p$. Note that $i_1<i_2=n-\al-1<\om_i+2p-\al+1\le\om_i+2p<i_1+p$. Hence $i_1\sim i_2$.
			Moreover, $n-p\le\om_i+p+1\le i_1<i_2=n-\al-1\le n-2$. This implies $p\ge 3$. Now, suppose $\om_i<\mb{c}$. Then $n-2p-1\le\om_i<\mb{c}\le\frac{n+1}{2}$, which yields $n<4p+3$. Since $n\ge 6p-3$, we get $p<3$, a contradiction. Therefore $\om_i\ge\mb{c}$.
		\end{enumerate} 
		\vspace{-0.2 cm}
	\end{proof} 
	
	\begin{remark}\label{remark:x4_x5 and x6_x7}
		Let $F\in\md$ such that $F\in\mcl{A}_i$ and $F^c=\{\om_i\}\sqcup\{i_1,i_2\}$. Suppose $F\in\mcl{M}_{\bt}$ for some $\bt\in[p-1]$. 
		\begin{enumerate}[label=(\roman*)]
			\item \label{x4_x5} If $F$ satisfies \X{\bt}{4} or \X{\bt}{5}, then $i_2\le2\mb{c}-\om_i+p-\bt$.
			\item \label{x6_x7} If $F$ satisfies \X{\bt}{6} or \X{\bt}{7}, then $i_1\ge2\mb{c}-\om_i-p+\bt-1$. 
		\end{enumerate}
	\end{remark}
	
	\begin{proposition}\label{proposition:om_i>c x6 x7 x8}
		Let $F\in\md$ such that $F\in\mcl{A}_i$ and $F^c=\{\om_i\}\sqcup\{i_1,i_2\}$ with $i_1<\om_i<i_2$ and $\om_i\ge\mb{c}$. 
		Suppose $F\in\mcl{M}_{\bt}$ for some $\bt\in[p-1]$. Then $F$ satisfies \X{\bt}{6} or \X{\bt}{7} or \X{\bt}{8}. Moreover, if $ \om_i<\mb{c}+\frac{p}{2}$, then $F$ satisfies \X{\bt}{6} or \X{\bt}{7}.
	\end{proposition}
	
	\begin{proof} 
		Since $F\in\mcl{M}_{\bt}$ for some $\bt\in[p-1]$ and $F^c=\{\om_i\}\sqcup\{i_1,i_2\}$ with $i_1<\om_i<i_2$, $F$ satisfies one of the conditions from \X{\bt}{3} to \X{\bt}{8}. If $F$ satisfies \X{\bt}{3} or \X{\bt}{5}, then $\om_i<\mb{c}$; and if $F$ satisfies \X{\bt}{4}, then $\om_i<\mb{c}$ by \Cref{proposition:M_al conditions}\ref{ob_x4}. Therefore, $\om_i\ge\mb{c}$ implies that $F$ satisfies \X{\bt}{6} or \X{\bt}{7} or \X{\bt}{8}.
		Moreover, if $\om_i<\mb{c}+\frac{p}{2}$, then $F$ does not satisfy \X{\bt}{8}, so $F$ satisfies \X{\bt}{6} or \X{\bt}{7}.
	\end{proof}
	
	\begin{proposition}\label{proposition:M_al M_be disjoint}
		For any distinct $\al, \bt\in[p-1]$, $\mcl{M}_{\al}\cap\mcl{M}_{\bt}=\emptyset$.
	\end{proposition}
	
	\begin{proof}
		On the contrary, suppose that $\mcl{M}_{\al}\cap\mcl{M}_{\bt}\ne\emptyset$. Then there exists a facet $F\in\md$ such that $F\in\mcl{M}_{\al}\cap\mcl{M}_{\bt}$. Since $F\in\mcl{M}_{\al}$, it follows that $F\in\mcl{A}_i$ such that $F^c=\{\om_i\}\sqcup\{i_1,i_2\}$, and $F$ satisfies one of the conditions from \X{\al}{1} to \X{\al}{10}. 
		
		Suppose $F$ satisfies \X{\al}{1}. Then $i_1<i_2<\om_i$, $\om_i<\mb{c}+\frac{p}{2}$ and $i_1=2\mb{c}-\om_i-p+\al-1$.
		Therefore, $F\in\mcl{M}_{\bt}$ implies $F$ satisfies \X{\bt}{1}. It follows that $i_1=2\mb{c}-\om_i-p+\bt-1$, thereby implying that $\al=\bt$, a contradiction. 
		Now, if $F$ satisfies \X{\al}{2} or \X{\al}{9} or \X{\al}{10}, then a similar argument leads to a contradiction to $\al\ne\bt$.
		
		Suppose $F$ satisfies \X{\al}{3}. Then $i_1<\om_i<i_2$, $\om_i<\mb{c}-\frac{p}{2}$ and $i_2=i_1+2p-\al+1$. 
		Therefore, $F\in\mcl{M}_{\bt}$ implies $F$ satisfies \X{\bt}{3} or \X{\bt}{5} or \X{\bt}{7}. If $F$ satisfies \X{\bt}{3}, then $i_2=i_1+2p-\bt+1$, which implies $\al=\bt$, a contradiction. Now, if $F$ satisfies \X{\bt}{5}, then $\om_i\ge\mb{c}-\frac{p-1}{2}$ by \Cref{proposition:M_al conditions}\ref{ob_x5}; and if $F$ satisfies \X{\bt}{7}, then $\om_i\ge\mb{c}+\frac{\bt}{2}$ by \Cref{proposition:M_al conditions}\ref{ob_x7}. Both contradict $\om_i<\mb{c}-\frac{p}{2}$. 
		
		Suppose $F$ satisfies any of the conditions from \X{\al}{4} to \X{\al}{8}. Then, using \Cref{proposition:M_al conditions} and a similar argument as above, we again get a contradiction to $\al\ne\bt$. 
		
		Hence our assumption that $\mcl{M}_{\al}\cap\mcl{M}_{\bt}\ne\emptyset$ is false.
	\end{proof}
	
	Define the subset $\mcl{M}_0$ of $\md$ as $$\mcl{M}_0:=\md\setminus(\sqcup_{\al=1}^{p-1}\mcl{M}_{\al}).$$ 
	Using \Cref{proposition:M_al M_be disjoint}, we conclude that $\md=\bigsqcup_{\al=0}^{p-1} \mcl{M}_{\al}$. 
	
	We now define an order $\prec$ on the elements of $\md$, which we prove to be a shelling order for the facets of $\Delta_3(C_n^p)$. For this, we modify the order $\ll$ by repositioning the elements of the sets $\mcl{M}_{\al}$ for $\al\in[p-1]$ in the poset $(\md, \ll)$, leading to a new poset $(\md, \prec)$.
	
	\begin{definition}\label{def:prec} 
		The order $\prec$ on the elements of $\md$ is defined as follows. Let $F,F'\in\md$. Then $F \prec F'$ if and only if any one of the following conditions is true:
		\begin{enumerate}[label=(\roman*)]
			\item \label{def 1} $F \ll F'$ and $F, F' \in\mcl{M}_{\al}$ for some $\al\in[0,p-1]$.
			\item \label{def 2} $F\in\mcl{M}_{\al}$ and $F' \in\mcl{M}_{\bt}$ for some $\al,\bt\in[0,p-1]$ and $\al<\bt$. 
		\end{enumerate}		
	\end{definition}
	
	Observe that the order $\prec$ on the elements of $\md$ is defined such that $\mcl{M}_{\al-1}$ precedes $\mcl{M}_{\al}$ for all $\al\in[p-1]$. Moreover, within each $\mcl{M}_{\al}$, where $\al\in[0,p-1]$, all facets are ordered using the order $\ll$. Hence, $\prec$ is a total order. 
	To show that $\prec$ provides a shelling order for the facets of $\Delta_3(C_n^p)$, we first prove several key results.
	
	\begin{proposition}\label{proposition:i1_i2 nsim al_i}
		Let $F\in\md$ such that $F\in\mcl{A}_i$ and $F^c=\{\om_i\}\sqcup\{i_1,i_2\}$.
		\begin{enumerate}[label=(\roman*)]
			\item \label{sim i1} If $\om_i-p\le i_1<\om_i$, then $i_2\nsim\om_i$ and $i_2\ge\om_i+p+1>\om_i$. 
			\item \label{sim i2} If $\om_i<i_2\le\om_i+p$, then $i_1\nsim\om_i$ and $i_1\le\om_i-p-1<\om_i$.
			\item \label{nsim i1} If $i_1<i_2<\om_i$, then $i_1\nsim\om_i$. 
			\item \label{nsim i2} If $\om_i<i_1<i_2$, then $i_2\nsim\om_i$.
		\end{enumerate} 
	\end{proposition}
	
	\begin{proof} 
		Since $F\in\md$, $\cn{F}$ is disconnected. Recall that $F\in\mcl{A}_i$ implies $i_1<i_2$ by our assumption.
		\begin{enumerate}[label=(\roman*)]
			\item Since $\om_i-p\le i_1<\om_i$, we have $i_1\sim\om_i$. The fact that $\cn{F}$ is disconnected implies $i_2\nsim\om_i$. Using $i_1<i_2$, we obtain $i_2\ge\om_i+p+1>\om_i$.
			
			\item Since $\om_i<i_2\le\om_i+p$, we get $i_2\sim\om_i$. Then, $\cn{F}$ is disconnected implies $i_1\nsim\om_i$. Moreover, $i_1<i_2$ gives $i_1\le\om_i-p-1<\om_i$.
			
			\item We have $0\le i_1<i_2<\om_i\le n-1$. Suppose $i_1 \sim \om_i$. Since $\cn{F}$ is disconnected, we get $i_1\nsim i_2$, and hence $i_1\le\om_i+p\ (\text{mod } n)$. This means that $\om_i\ge n-p$. Since $n-p>\mb{c}$ by \Cref{proposition: c and p relation}\ref{range of c}, we get $\om_i>\mb{c}$. 
			By \Cref{remark:Aj}, $\om_i\os i_2$. Therefore $i_2<2\mb{c}-\om_i$ by \Cref{remark:order in S}\ref{part:ii}. Using $\om_i\ge n-p$, we obtain $i_2<2\mb{c}-n+p\le p+1$ (as $2\mb{c}\in\{n,n+1\}$). It follows that $0\le i_1<i_2\le p$, which implies $i_1\sim i_2$, a contradiction. Hence $i_1\nsim\om_i$. 
			
			\item We have $0\le\om_i<i_1<i_2\le n-1$. Suppose $i_2 \sim \om_i$. Using the fact that $\cn{F}$ is disconnected, we get $i_1\nsim i_2$, and therefore $\om_i\le i_2+p\ (\text{mod } n)$. This means that $\om_i\le p-1$. Since $p<\mb{c}$ by \Cref{proposition: c and p relation}\ref{range of c}, $\om_i<\mb{c}$. 
			By \Cref{remark:Aj}, $\om_i\os i_1$, and hence by \Cref{remark:order in S}\ref{part:i}, $i_1\ge2\mb{c}-\om_i$. Now, $\om_i\le p-1$ implies $i_1\ge2\mb{c}-p+1\ge n-p+1$. This gives $n-p+1\le i_1<i_2\le n-1$, and thus $i_1\sim i_2$, a contradiction. Hence $i_2\nsim\om_i$.
		\end{enumerate}
		\vspace{-0.2 cm}
	\end{proof}
	
	\begin{proposition}\label{proposition:in M_bt+gm} 
		Let $F,F'\in\md$ and let $\bt\in[p-1]$. Assume that $F\in\mcl{A}_i$ such that $F^c=\{\om_i\}\sqcup\{i_1,i_2\}$ and $F'\in\mcl{M}_{\bt}$. 
		\begin{enumerate}[label=(\arabic*)]
			\item \label{satisfies x1} Suppose $i_1<i_2<\om_i$ and $\om_i<\mb{c}+\frac{p}{2}$. 
			\begin{enumerate}[label=(\roman*)]
				\item \label{satisfies x1 i} If $i_1=2\mb{c}-\om_i-p+\bt-1$, then $F\notin\mcl{M}_0$. Moreover, if $F'\ll F$, then $F'\prec F$.
				
				\item \label{satisfies x1 ii} If $i_1\ge2\mb{c}-\om_i-p+\bt$, then $F\notin\mcl{M}_0$ and $F'\prec F$.
			\end{enumerate}
			
			\item \label{satisfies x2} Suppose $i_1<i_2<\om_i$ and $\om_i\ge\mb{c}+\frac{p}{2}$. 
			\begin{enumerate}[label=(\roman*)]
				\item \label{satisfies x2 i} If $i_1=\om_i-2p+\bt-1$, then $F\notin\mcl{M}_0$. Moreover, if $F'\ll F$, then $F'\prec F$.
				
				\item \label{satisfies x2 ii} If $i_1\ge\om_i-2p+\bt$, then $F\notin\mcl{M}_0$ and $F'\prec F$.
			\end{enumerate}
			
			\item \label{satisfies x3} Suppose $i_1<\om_i<i_2$ and $\om_i<\mb{c}-\frac{p}{2}$. 
			\begin{enumerate}[label=(\roman*)]
				\item \label{satisfies x3 i} If $i_2=i_1+2p-\bt+1$, then $F\notin\mcl{M}_0$. Moreover, if $F'\ll F$, then $F'\prec F$.
				
				\item \label{satisfies x3 ii} If $i_2\le i_1+2p-\bt$, then $F\notin\mcl{M}_0$ and $F'\prec F$.
			\end{enumerate}
			
			\item \label{satisfies x4} Suppose $i_1<\om_i$, $\om_i\ge\mb{c}-\frac{p}{2}$ and $\om_i-p\le i_1\le2\mb{c}-\om_i-p-1$. 
			\begin{enumerate}[label=(\roman*)]
				\item \label{satisfies x4 i} If $i_2=i_1+2p-\bt+1$, then $F\notin\mcl{M}_0$. Moreover, if $F'\ll F$, then $F'\prec F$.
				
				\item \label{satisfies x4 ii} If $i_2\le i_1+2p-\bt$, then $F\notin\mcl{M}_0$ and $F'\prec F$.
			\end{enumerate}
			
			\item \label{satisfies x5} Suppose $i_1<\om_i$, $\om_i\ge\mb{c}-\frac{p}{2}$, $i_1\ge\om_i-p$ and $i_1\ge2\mb{c}-\om_i-p$. 
			\begin{enumerate}[label=(\roman*)]
				\item \label{satisfies x5 i} If $i_2=2\mb{c}-\om_i+p-\bt$, then $F\notin\mcl{M}_0$. Moreover, if $F'\ll F$, then $F'\prec F$.
				
				\item \label{satisfies x5 ii} If $i_2\le2\mb{c}-\om_i+p-\bt-1$, then $F\notin\mcl{M}_0$ and $F'\prec F$.
			\end{enumerate}
			
			\item \label{satisfies x6} Suppose $\om_i<i_2$, $\om_i<\mb{c}+\frac{p}{2}$, $i_2\le\om_i+p$ and $i_2\le2\mb{c}-\om_i+p-1$. 
			\begin{enumerate}[label=(\roman*)]
				\item \label{satisfies x6 i} If $i_1=2\mb{c}-\om_i-p+\bt-1$, then $F\notin\mcl{M}_0$. Moreover, if $F'\ll F$, then $F'\prec F$.
				
				\item \label{satisfies x6 ii} If $i_1\ge2\mb{c}-\om_i-p+\bt$, then $F\notin\mcl{M}_0$ and $F'\prec F$.
			\end{enumerate}
			
			\item \label{satisfies x7} Suppose $\om_i<i_2$, $\om_i<\mb{c}+\frac{p}{2}$ and $2\mb{c}-\om_i+p\le i_2\le\om_i+p$. 
			\begin{enumerate}[label=(\roman*)]
				\item \label{satisfies x7 i} If $i_1=i_2-2p+\bt-1$, then $F\notin\mcl{M}_0$. Moreover, if $F'\ll F$, then $F'\prec F$.
				
				\item \label{satisfies x7 ii} If $i_1\ge i_2-2p+\bt$, then $F\notin\mcl{M}_0$ and $F'\prec F$.
			\end{enumerate}
			
			\item \label{satisfies x8} Suppose $i_1<\om_i<i_2$ and $\om_i\ge\mb{c}+\frac{p}{2}$. 
			\begin{enumerate}[label=(\roman*)]
				\item \label{satisfies x8 i} If $i_1=i_2-2p+\bt-1$, then $F\notin\mcl{M}_0$. Moreover, if $F'\ll F$, then $F'\prec F$.
				
				\item \label{satisfies x8 ii} If $i_1\ge i_2-2p+\bt$, then $F\notin\mcl{M}_0$ and $F'\prec F$. 
			\end{enumerate}
			
			\item \label{satisfies x9} Suppose $\om_i<i_1<i_2$ and $\om_i\le n-2p-2$. Further, assume that $i_1\ge\om_i+p+1$ or $i_1\ge i_2-p$.
			\begin{enumerate}[label=(\roman*)]
				\item \label{satisfies x9 i} If $i_2=\om_i+2p-\bt+1$, then $F\notin\mcl{M}_0$. Moreover, if $F'\ll F$, then $F'\prec F$.
				
				\item \label{satisfies x9 ii} If $i_2\le\om_i+2p-\bt$, then $F\notin\mcl{M}_0$ and $F'\prec F$. 
			\end{enumerate}
			
			\item \label{satisfies x10} Suppose $\om_i<i_1<i_2$ and $\om_i>n-2p-2$. Further, assume that $i_1\ge\om_i+p+1$ or $i_1\ge i_2-p$. 
			\begin{enumerate}[label=(\roman*)]
				\item \label{satisfies x10 i} If $i_2=n-\bt-1$, then $F\notin\mcl{M}_0$. Moreover, if $F'\ll F$, then $F'\prec F$.
				
				\item \label{satisfies x10 ii} If $i_2\le n-\bt-2$, then $F\notin\mcl{M}_0$ and $F'\prec F$. 
			\end{enumerate}
		\end{enumerate}
	\end{proposition}
	
	\begin{proof}
		\begin{enumerate}[label=(\arabic*)]
			\item If $i_1=2\mb{c}-\om_i-p+\bt-1$, then $F$ satisfies \X{\bt}{1}, and therefore $F\in\mcl{M}_{\bt}$. Hence $F\notin\mcl{M}_0$. Moreover, if $F'\ll F$, then $F'\in\mcl{M}_{\bt}$ implies $F'\prec F$ by \Cref{def:prec}\ref{def 1}.
			
			Now, let $i_1\ge2\mb{c}-\om_i-p+\bt$. Then $i_1=2\mb{c}-\om_i-p+\bt+\gm-1$ for some $\gm\ge 1$. By \Cref{proposition:al_i>c}, $i_2\le\om_i-p-1$. Using $i_1<i_2$, we get $\om_i>\mb{c}+\frac{\bt+\gm}{2}$.
			Since $\om_i<\mb{c}+\frac{p}{2}$, $\bt+\gm<p-1$. Moreover, $\bt+\gm>\bt\ge 1$. This gives $\bt+\gm\in[p-1]$. Therefore $F$ satisfies \X{\bt+\gm}{1}, and thus $F\in\mcl{M}_{\bt+\gm}$.
			Hence $F\notin\mcl{M}_0$, and $F'\in\mcl{M}_{\bt}$ implies $F'\prec F$ by \Cref{def:prec}\ref{def 2}.
			
			\item If $i_1=\om_i-2p+\bt-1$, then $F$ satisfies \X{\bt}{2}, and thus $F\in\mcl{M}_{\bt}$. Hence $F\notin\mcl{M}_0$. Moreover, if $F'\ll F$, then $F'\in\mcl{M}_{\bt}$ implies $F'\prec F$ by \Cref{def:prec}\ref{def 1}.
			
			Now, let $i_1\ge\om_i-2p+\bt$. Then $i_1=\om_i-2p+\bt+\gm-1$ for some $\gm\ge 1$. By \Cref{remark:Aj}, $\om_i\os i_2$. Since $\om_i>\mb{c}$ and $i_2<\om_i$, \Cref{remark:order in S}\ref{part:ii} gives $i_2<2\mb{c}-\om_i$. Using $i_1<i_2$, we obtain $i_1\le2\mb{c}-\om_i-2$, which implies $\om_i\le\mb{c}+\frac{2p-\bt-\gm-1}{2}$. Since $\om_i\ge\mb{c}+\frac{p}{2}$, we have $\bt+\gm\le p-1$. 
			Therefore $\bt+\gm>\bt\ge 1$ implies $\bt+\gm\in[p-1]$. Note that $F$ satisfies \X{\bt+\gm}{2}, and thus $F\in\mcl{M}_{\bt+\gm}$.
			Hence $F\notin\mcl{M}_0$, and $F'\in\mcl{M}_{\bt}$ implies $F'\prec F$ by \Cref{def:prec}\ref{def 2}.
			
			\item If $i_2=i_1+2p-\bt+1$, then $F$ satisfies \X{\bt}{3}, and therefore $F\in\mcl{M}_{\bt}$. Hence $F\notin\mcl{M}_0$. Moreover, if $F'\ll F$, then $F'\in\mcl{M}_{\bt}$ implies $F'\prec F$ by \Cref{def:prec}\ref{def 1}.
			
			Now, let $i_2\le i_1+2p-\bt$. Then $i_2=i_1+2p-\bt-\gm+1$ for some $\gm\ge 1$. By \Cref{remark:Aj}, $\om_i\os i_2$. Since $\om_i<\mb{c}$ and $i_2>\om_i$, we have $i_2\ge2\mb{c}-\om_i$ \Cref{remark:order in S}\ref{part:i}. This implies $\om_i\ge2\mb{c}-i_2=2\mb{c}-i_1-2p+\bt+\gm-1$. Using $i_1<\om_i$, we obtain $\om_i\ge2\mb{c}-\om_i-2p+\bt+\gm$, and hence $\om_i\ge\mb{c}-\frac{2p-\bt-\gm}{2}$. 
			Further, $\om_i<\mb{c}-\frac{p}{2}$ implies $\bt+\gm\le p-1$. Since $\bt+\gm>\bt\ge 1$, we get $\bt+\gm\in[p-1]$. Therefore $F$ satisfies \X{\bt+\gm}{3}, and thus $F\in\mcl{M}_{\bt+\gm}$. Hence $F\notin\mcl{M}_0$, and $F'\in\mcl{M}_{\bt}$ implies $F'\prec F$ by \Cref{def:prec}\ref{def 2}.
			
			\item Since $\om_i-p\le i_1<\om_i$, we get $i_2\ge\om_i+p+1>\om_i$ by \Cref{proposition:i1_i2 nsim al_i}\ref{sim i1}. Hence $i_1<\om_i<i_2$ and $i_2\ge i_1+p+2$.
			
			If $i_2=i_1+2p-\bt+1$, then $i_2\ge\om_i+p+1$ implies $i_1\ge\om_i-p+\bt$. This means that $F$ satisfies \X{\bt}{4}, and therefore $F\in\mcl{M}_{\bt}$. Hence $F\notin\mcl{M}_0$. Moreover, if $F'\ll F$, then $F'\in\mcl{M}_{\bt}$ implies $F'\prec F$ by \Cref{def:prec}\ref{def 1}.
			
			Now, let $i_2\le i_1+2p-\bt$. Then $i_2\ge i_1+p+2$ implies $i_2=i_1+2p-\bt-\gm+1$ for some $1\le\gm\le p-\bt-1$. Note that $\bt+\gm\le p-1$. Since $\bt+\gm>\bt\ge 1$, we have $\bt+\gm\in[p-1]$. Moreover, $i_2\ge\om_i+p+1$ implies $i_1\ge\om_i-p+\bt+\gm$. 
			Therefore $F$ satisfies \X{\bt+\gm}{4}, and thus $F\in\mcl{M}_{\bt+\gm}$. Hence $F\notin\mcl{M}_0$, and $F'\in\mcl{M}_{\bt}$ implies $F'\prec F$ by \Cref{def:prec}\ref{def 2}.
			
			\item Since $\om_i-p\le i_1<\om_i$, \Cref{proposition:i1_i2 nsim al_i}\ref{sim i1} gives $i_2\ge\om_i+p+1>\om_i$. Hence $i_1<\om_i<i_2$.
			
			If $i_2=2\mb{c}-\om_i+p-\bt$, then $2\mb{c}-\om_i+p-\bt\ge\om_i+p+1$, which implies $\om_i\le\mb{c}-\frac{\bt+1}{2}$.
			Observe that $F$ satisfies \X{\bt}{5}, and hence $F\in\mcl{M}_{\bt}$. Therefore $F\notin\mcl{M}_0$. Moreover, if $F'\ll F$, then $F'\in\mcl{M}_{\bt}$ implies $F'\prec F$ by \Cref{def:prec}\ref{def 1}.
			
			Now, let $i_2\le2\mb{c}-\om_i+p-\bt-1$. Then $i_2=2\mb{c}-\om_i+p-\bt-\gm$ for some $\gm\ge 1$. Therefore, $i_2\ge\om_i+p+1$ implies $\om_i\le\mb{c}-\frac{\bt+\gm+1}{2}$. Since $\om_i\ge\mb{c}-\frac{p}{2}$, we get $\bt+\gm\le p-1$. Moreover, $\bt+\gm>\bt\ge 1$. It follows that $\bt+\gm\in[p-1]$.
			Thus $F$ satisfies \X{\bt+\gm}{5}, which implies $F\in\mcl{M}_{\bt+\gm}$. Hence $F\notin\mcl{M}_0$, and $F'\in\mcl{M}_{\bt}$ implies $F'\prec F$ by \Cref{def:prec}\ref{def 2}.
			
			\item By \Cref{proposition:i1_i2 nsim al_i}\ref{sim i2}, $\om_i<i_2\le\om_i+p$ implies $i_1\le\om_i-p-1<\om_i$. Hence $i_1<\om_i<i_2$.
			
			If $i_1=2\mb{c}-\om_i-p+\bt-1$, then $2\mb{c}-\om_i-p+\bt-1\le\om_i-p-1$, which implies $\om_i\ge\mb{c}+\frac{\bt}{2}$.
			Therefore $F$ satisfies \X{\bt}{6}, and hence $F\in\mcl{M}_{\bt}$. Thus $F\notin\mcl{M}_0$. Moreover, if $F'\ll F$, then $F'\in\mcl{M}_{\bt}$ implies $F'\prec F$ by \Cref{def:prec}\ref{def 1}.
			
			Now, let $i_1\ge2\mb{c}-\om_i-p+\bt$. Then $i_1=2\mb{c}-\om_i-p+\bt+\gm-1$ for some $\gm\ge 1$. Using $i_1\le\om_i-p-1$, we obtain $\om_i\ge\mb{c}+\frac{\bt+\gm}{2}$. Further, $\om_i<\mb{c}+\frac{p}{2}$ implies $\bt+\gm\le p-1$. Since $\bt+\gm>\bt\ge 1$, it follows that $\bt+\gm\in[p-1]$. 
			Therefore $F$ satisfies \X{\bt+\gm}{6}, and thus $F\in\mcl{M}_{\bt+\gm}$. Hence $F\notin\mcl{M}_0$, and $F'\in\mcl{M}_{\bt}$ implies $F'\prec F$ by \Cref{def:prec}\ref{def 2}.
			
			\item Since $\om_i<i_2\le\om_i+p$, \Cref{proposition:i1_i2 nsim al_i}\ref{sim i2} gives $i_1\le\om_i-p-1<\om_i$. Hence $i_1<\om_i<i_2$ and $i_1\le i_2-p-2$.
			
			If $i_1=i_2-2p+\bt-1$, then $i_1\le\om_i-p-1$ implies $i_2\le\om_i+p-\bt$. Therefore $F$ satisfies \X{\bt}{7}, and hence $F\in\mcl{M}_{\bt}$. Thus $F\notin\mcl{M}_0$. Moreover, if $F'\ll F$, then $F'\in\mcl{M}_{\bt}$ implies $F'\prec F$ by \Cref{def:prec}\ref{def 1}.
			
			Now, let $i_1\ge i_2-2p+\bt$. Then $i_1\le i_2-p-2$ implies $i_1=i_2-2p+\bt+\gm-1$ for some $1\le\gm\le p-\bt-1$. Clearly, $\bt+\gm\le p-1$. Since $\bt+\gm>\bt\ge 1$, we have $\bt+\gm\in[p-1]$. Moreover, $i_1\le\om_i-p-1$ implies $i_2\le\om_i+p-\bt-\gm$. 
			Therefore $F$ satisfies \X{\bt+\gm}{7}, and thus $F\in\mcl{M}_{\bt+\gm}$. Hence $F\notin\mcl{M}_0$, and $F'\in\mcl{M}_{\bt}$ implies $F'\prec F$ by \Cref{def:prec}\ref{def 2}.
			
			\item If $i_1=i_2-2p+\bt-1$, then $F$ satisfies \X{\bt}{8} and therefore $F\in\mcl{M}_{\bt}$. Hence $F\notin\mcl{M}_0$. Moreover, if $F'\ll F$, then $F'\in\mcl{M}_{\bt}$ implies $F'\prec F$ by \Cref{def:prec}\ref{def 1}.
			
			Now, let $i_1\ge i_2-2p+\bt$. Then $i_1=i_2-2p+\bt+\gm-1$ for some $\gm\ge 1$. By \Cref{remark:Aj}, $\om_i\os i_1$. Therefore, $\om_i>\mb{c}$ and $i_1<\om_i$ imply $i_1<2\mb{c}-\om_i$ by \Cref{remark:order in S}\ref{part:ii}. It follows that $\om_i\le2\mb{c}-i_1-1=2\mb{c}-i_2+2p-\bt-\gm$. 
			Since $i_2>\om_i$, we get $\om_i\le2\mb{c}-\om_i+2p-\bt-\gm-1$, and thus $\om_i\le\mb{c}+\frac{2p-\bt-\gm-1}{2}$. Using $\om_i\ge\mb{c}+\frac{p}{2}$, we obtain $\bt+\gm\le p-1$. Moreover, $\bt+\gm>\bt\ge 1$, which implies $\bt+\gm\in[p-1]$. 
			Observe that $F$ satisfies \X{\bt+\gm}{8}, and thus $F\in\mcl{M}_{\bt+\gm}$. Hence $F\notin\mcl{M}_0$, and $F'\in\mcl{M}_{\bt}$ implies $F'\prec F$ by \Cref{def:prec}\ref{def 2}.
			
			\item We have $\om_i<i_1<i_2$ and $\om_i\le n-2p-2$. Furthermore, $i_1\ge\om_i+p+1$ or $i_1\ge i_2-p$.
			\begin{enumerate}[label=(\alph*)]
				\item Let $i_1\ge\om_i+p+1$. If $i_2=\om_i+2p-\bt+1$, then $F$ satisfies \X{\bt}{9} and thus $F\in\mcl{M}_{\bt}$. Hence $F\notin\mcl{M}_0$. Moreover, if $F'\ll F$, then $F'\prec F$ by \Cref{def:prec}\ref{def 1} (as $F'\in\mcl{M}_{\bt}$).
				
				Now, suppose $i_2\le\om_i+2p-\bt$. Since $i_2>i_1\ge\om_i+p+1$, $i_2=\om_i+2p-\bt-\gm+1$ for some $1\le\gm\le p-\bt-1$. Note that $\bt+\gm\in[p-1]$. 
				Therefore $F$ satisfies \X{\bt+\gm}{9}, and hence $F\in\mcl{M}_{\bt+\gm}$. Thus $F\notin\mcl{M}_0$, and $F'\in\mcl{M}_{\bt}$ implies $F'\prec F$ by \Cref{def:prec}\ref{def 2}.
				
				\item Let $i_1\ge i_2-p$. Then $i_1<i_2$ implies $i_1\sim i_2$. Since $F\in\md$, $\cn{F}$ is disconnected. Hence $i_1\nsim\om_i$, and thus $i_1\ge\om_i+p+1$ (as $i_1>\om_i$). Therefore, the result follows from case (a).
			\end{enumerate}
			
			\item We have $\om_i<i_1<i_2$ and $\om_i>n-2p-2$. Furthermore, $i_1\ge\om_i+p+1$ or $i_1\ge i_2-p$.
			\begin{enumerate}[label=(\alph*)]
				\item Let $i_1\ge\om_i+p+1$. If $i_2=n-\bt-1$, then $F$ satisfies \X{\bt}{10} and therefore $F\in\mcl{M}_{\bt}$. Hence $F\notin\mcl{M}_0$. Moreover, if $F'\ll F$, then $F'\prec F$ by \Cref{def:prec}\ref{def 1} (as $F'\in\mcl{M}_{\bt}$).
				
				Now, assume $i_2\le n-\bt-2$. Since $\om_i>n-2p-2$ and $i_2>i_1\ge\om_i+p+1$, we get $i_2>n-p$. This means that $i_2=n-\bt-\gm-1$ for some $1\le\gm<p-\bt-1$.
				Clearly, $\bt+\gm\in[p-1]$. Therefore $F$ satisfies \X{\bt+\gm}{10}, and hence $F\in\mcl{M}_{\bt+\gm}$. Thus $F\notin\mcl{M}_0$, and $F'\in\mcl{M}_{\bt}$ implies $F'\prec F$ by \Cref{def:prec}\ref{def 2}. 
				
				\item Let $i_1\ge i_2-p$. Since $i_1<i_2$, we have $i_1\sim i_2$. Then, $F\in\md$ implies $i_1\nsim\om_i$. Hence $i_1\ge\om_i+p+1$ (as $i_1>\om_i$). Therefore, the result follows from case (a).
			\end{enumerate}
		\end{enumerate} 
		\vspace{-0.2 cm}
	\end{proof} 
	
	\begin{corollary}\label{corollary:notin M_0}
		Let $F,F'\in\md$, and let $\bt\in[p-1]$. Assume $F\in\mcl{A}_i$ with $F^c=\{\om_i\}\sqcup\{i_1,i_2\}$, and $F'\in\mcl{M}_{\bt}$. 
		\begin{enumerate}[label=(\roman*)]
			\item \label{satisfies 3_4_5} If $\om_i-p\le i_1<\om_i$, $i_2\le i_1+2p-\bt+1$ and $i_2\le2\mb{c}-\om_i+p-\bt$, then $F\notin\mcl{M}_0$. Moreover, if $F'\ll F$, then $F'\prec F$.
			
			\item \label{satisfies 6_7_8} If $\om_i<i_2\le\om_i+p$, $i_1\ge i_2-2p+\bt-1$ and $i_1\ge2\mb{c}-\om_i-p+\bt-1$, then $F\notin\mcl{M}_0$. Moreover, if $F'\ll F$, then $F'\prec F$.
		\end{enumerate}
	\end{corollary} 
	
	\begin{proof}
		\begin{enumerate}[label=(\roman*)] 
			\item By \Cref{proposition:i1_i2 nsim al_i}\ref{sim i1}, we have $i_2>\om_i$. Observe that if $\om_i<\mb{c}-\frac{p}{2}$, then $F\notin\mcl{M}_0$ by \Cref{proposition:in M_bt+gm}\ref{satisfies x3}. 
			Now, assume that $\om_i\ge\mb{c}-\frac{p}{2}$. If $i_1\le2\mb{c}-\om_i-p-1$, then $F\notin\mcl{M}_0$ by \Cref{proposition:in M_bt+gm}\ref{satisfies x4}, and if $i_1\ge2\mb{c}-\om_i-p$, then $F\notin\mcl{M}_0$ by \Cref{proposition:in M_bt+gm}\ref{satisfies x5}. 
			Similarly, using \Cref{proposition:in M_bt+gm}\ref{satisfies x3}, \ref{satisfies x4} and \ref{satisfies x5}, if $F'\ll F$, then $F'\prec F$.
			
			\item By \Cref{proposition:i1_i2 nsim al_i}\ref{sim i2}, we have $i_1<\om_i$. The proof follows a similar argument as in \ref{satisfies 3_4_5}, using \Cref{proposition:in M_bt+gm}\ref{satisfies x6}, \ref{satisfies x7} and \ref{satisfies x8}. 
		\end{enumerate}
		\vspace{-0.2 cm}
	\end{proof} 
	
	Let $F \in\md$. For $u\in F$ and $v\in F^c$, we introduce the notation $\Fuv{u}{v}$ to represent the set obtained by replacing $u$ in $F$ with $v$, formally defined as $$\Fuv{u}{v}:=(F\setminus\{u\})\sqcup\{v\}.$$
	It is easy to see that $\Fuv{u}{v} \cap F=F\setminus\{u\}$ and $(\Fuv{u}{v})^c=(F^c\setminus\{v\})\sqcup\{u\}$. These relations will be used later in our arguments.
	
	\subsection{Shelling Order}\label{subsection:shelling order}
	
	This section focuses on proving that the order $\prec$, defined on the facets of  $\Delta_3(C_n^p)$ is a shelling order. We begin by introducing several results that will be used in the proof.
	
	\begin{proposition}\label{proposition: om_i<=2p-1}
		Let $F, F' \in\md$ such that $F, F' \in\mcl{A}_i$. Suppose $F^c=\{\om_i\}\sqcup\{i_1,i_2\}$ and $F'^c=\{\om_i\} \sqcup \{j_1, j_2\}$. 
		If $i_1\in F'$, $ i_1<\om_i<j_2$ with $j_2\nsim\om_i$, and $F'(i_1,j_1)\notin\md$, then $\om_i\le2p-1$, $\om_i<i_2<j_2$, $i_2\nsim\om_i$ and $i_1\nsim i_2$.
	\end{proposition}
	
	\begin{proof}
		We have $(F'(i_1,j_1))^c=(F'^c\setminus\{j_1\}) \sqcup \{ i_1\}=\{\om_i,j_2,i_1\}$. Since $F'(i_1,j_1)\notin\md$, we get $\cn{(F'(i_1,j_1))}$ is connected. Therefore, $i_1<\om_i<j_2$ with $j_2\nsim\om_i$ implies $j_2\sim i_1$ and $ i_1\sim\om_i$. 
		Hence $\om_i\le j_2+2p\ (\text{mod } n)\le2p-1$ (as $j_2\le n-1$). Moreover, $F\in\md$ implies $\cn{F}$ is disconnected. Since $i_1\sim\om_i$, we have $i_2\nsim\om_i$ and $i_1\nsim i_2$. 
		Using $i_1<\om_i<j_2$, $j_2\nsim\om_i$, $j_2\sim i_1$, $i_1\sim\om_i$ and $i_1\nsim i_2$, we obtain $\om_i<i_2<j_2$.
	\end{proof}
	
	\begin{proposition}\label{proposition:in M_0}
		Let $F\in\md$ such that $F\in\mcl{A}_i$ and $F^c=\{\om_i\}\sqcup\{i_1,i_2\}$. If $i_1<\om_i<i_2$ and $\om_i\le2p-1$, then $F\in\mcl{M}_0$.
	\end{proposition}
	
	\begin{proof}
		Suppose that $F\in\mcl{M}_{\al}$ for some $\al\in[p-1]$. Since $i_1<\om_i<i_2$, it follows that $F$ satisfies one of the conditions from \X{\al}{3} to \X{\al}{8}. By \Cref{proposition: c and p relation}\ref{range of c}, $\mb{c}\ge 3p-1$. Hence, $\om_i\le2p-1$ implies $\om_i\le\mb{c}-p<\mb{c}-\frac{p}{2}$.
		This means that $F$ satisfies \X{\al}{3} or \X{\al}{5} or \X{\al}{7}. If $F$ satisfies \X{\al}{5} or \X{\al}{7}, then \Cref{proposition:M_al conditions}\ref{ob_x5} and \ref{ob_x7} contradict $\om_i<\mb{c}-\frac{p}{2}$.
		Therefore $F$ satisfies \X{\al}{3}. By \Cref{proposition:M_al conditions}\ref{ob_x3}, $i_1\ge2\mb{c}-\om_i-2p+\al-1$. Since $\om_i\le\mb{c}-p$ and $\al\ge 1$, it follows that $i_1\ge\mb{c}-p\ge\om_i$, a contradiction. Hence $F\in\mcl{M}_0$.
	\end{proof}
	
	\begin{proposition}\label{proposition:2c-i_2<=al_i<i_2}
		Let $F,F'\in\md$ and let $\bt\in[p-1]$. Assume $F\in\mcl{A}_i$ with $F^c=\{\om_i\}\sqcup\{i_1,i_2\}$, and $F'\in\mcl{M}_{\bt}$. If $i_1<\om_i<i_2$, $2\mb{c}-i_2\le\om_i<i_2$, $i_1=2\mb{c}-i_2-p+\bt-1$, $i_2<\mb{c}+\frac{p}{2}$ and $F\ll F'$, then $F\prec F'$.
	\end{proposition}
	
	\begin{proof}
		We show that $F \in\mcl{M}_{\alpha}$ for some $0\le\al\le\bt$. Then, since $F\ll F'$, it follows that $F\prec F'$ by \Cref{def:prec}. Suppose that $F\in\mcl{M}_{\alpha}$ for some $\bt<\alpha\le p-1$. We have $i_1<\om_i<i_2$ and $\mb{c}-\frac{p}{2}<2\mb{c}-i_2\le\om_i<i_2<\mb{c}+\frac{p}{2}$. Thus, $F$ satisfies one of the conditions from \X{\al}{4} to \X{\al}{7}. 
		
		Suppose $F$ satisfies \X{\al}{4}. Then $i_1\ge\om_i-p+\al$. Since $\om_i\ge2\mb{c}-i_2$ and $\al>\bt$, it follows that $i_1>2\mb{c}-i_2-p+\bt$, a contradiction. Hence $F$ does not satisfy \X{\al}{4}.
		Suppose $F$ satisfies \X{\al}{5}. Then $\om_i\le\mb{c}-\frac{\al+1}{2}$ and $i_1\ge2\mb{c}-\om_i-p$. We get $2\mb{c}-\om_i-p\le i_1=2\mb{c}-i_2-p+\bt-1<\om_i-p+\al-1$, which implies that $\om_i>\mb{c}-\frac{\al-1}{2}$, a contradiction. Hence $F$ does not satisfy \X{\al}{5}. 
		Suppose $F$ satisfies \X{\al}{6}. Then $i_1=2\mb{c}-\om_i-p+\al-1$. Since $i_1=2\mb{c}-i_2-p+\bt-1$ and $\al>\bt$, we get $\om_i=i_2+\al-\bt>i_2$, a contradiction as $\om_i<i_2$. Hence $F$ does not satisfy \X{\al}{6}. 
		Suppose $F$ satisfies \X{\al}{7}. Then $i_2\ge2\mb{c}-\om_i+p$. Since $\om_i<i_2$ and $i_2<\mb{c}+\frac{p}{2}$, it follows that $i_2>2\mb{c}-i_2+p>\mb{c}+\frac{p}{2}$, a contradiction. Hence $F$ does not satisfy \X{\al}{7}.
		
		Therefore $F \in\mcl{M}_{\alpha}$ for some $0\le\al\le\bt$.
	\end{proof} 
	
	\begin{proposition}\label{proposition:x1 and x6}
		Let $F, F' \in\md$. Let $F \in\mcl{A}_i$ and $F' \in\mcl{A}_j$ such that $F^c=\{\om_i\}\sqcup\{i_1,i_2\}$ and $F'^c=\{\om_j\} \sqcup \{j_1, j_2\}$ with $j_1<\om_j<i_2$, $j_1\le i_1<i_2$ and $F\prec F'$. 
		Assume that if $\om_i=\om_j$, then $j_1<i_1$. 
		Let $F'\in\mcl{M}_{\bt}$ for some $\beta \in [p-1]$, where $\mb{c}+\frac{\bt}{2}\le\om_j<\mb{c}+\frac{p}{2}$ and $j_1=2\mb{c}-\om_j-p+\bt-1$.
		Suppose $i_2\in F'$, $(F'(i_2,j_2))^c=\{\om_j\}\sqcup\{j_1,i_2\}$ and $F'(i_2,j_2)\in\mcl{M}_{\gm}$ for some $\bt\le\gm\le p-1$. Then $i_2\le2\mb{c}-\om_j+p$ and $2\mb{c}-\om_j\le\om_i<\om_j$.
	\end{proposition} 
	
	\begin{proof}
		We have $(F'(i_2,j_2))^c=\{\om_j\}\sqcup\{j_1,i_2\}$ and $F'(i_2,j_2)\in\mcl{M}_{\gm}$ for some $1\le\bt\le\gm\le p-1$. Since $j_1<\om_j<i_2$ and $\mb{c}<\mb{c}+\frac{\bt}{2}\le\om_j<\mb{c}+\frac{p}{2}$, $F'(i_2,j_2)$ satisfies \X{\gm}{6} or \X{\gm}{7} by \Cref{proposition:om_i>c x6 x7 x8}.
		
		If $F'(i_2,j_2)$ satisfies \X{\gm}{6}, then $j_1=2\mb{c}-\om_j-p+\gm-1$ and $j_1\ge i_2-2p+\gm$, and by \Cref{proposition:M_al conditions}\ref{ob_x6}, $i_2\le\om_j+p-\gm-1$. Moreover, if $F'(i_2,j_2)$ satisfies \X{\gm}{7}, then $j_1=i_2-2p+\gm-1$ and $i_2\le\om_j+p-\gm$, and by \Cref{proposition:M_al conditions}\ref{ob_x7}, $j_1\ge2\mb{c}-\om_j-p+\gm-1$.
		Therefore, in both cases, we have $j_1\ge2\mb{c}-\om_j-p+\gm-1$, $j_1\ge i_2-2p+\gm-1$, and $i_2\le\om_j+p-\gm$. 
		Further, using $j_1=2\mb{c}-\om_j-p+\bt-1$ and $\gm\ge\bt$, we obtain $\gm=\bt$. Thus $j_1\ge i_2-2p+\bt-1$ and $i_2\le\om_j+p-\bt$. Note that $i_2\le j_1+2p-\bt+1=2\mb{c}-\om_j+p$. Since $\om_j>\mb{c}$, we get $2\mb{c}-\om_j<\mb{c}<\om_j$. 
		
		We first show that $\om_i\ge2\mb{c}-\om_j$. Suppose that $\om_i<2\mb{c}-\om_j$. Since $\om_j>\mb{c}$, $\om_j\os\om_i$ by \Cref{remark:order in S}\ref{part:ii}. 
		Hence $F'\ll F$ by \Cref{definition:order_ll}\ref{om_i<om_j}. We have $i_1\ge j_1\ge i_2-2p+\bt-1$, $i_1\ge j_1=2\mb{c}-\om_j-p+\bt-1>\om_i-p+\bt-1\ge\om_i-p$ and $i_2\le\om_j+p-\bt<2\mb{c}-\om_i+p-\bt$. 
		Therefore, if $i_1<\om_i$, then $F'\prec F$ by \Cref{corollary:notin M_0}\ref{satisfies 3_4_5}, a contradiction to our assumption that $F\prec F'$. Hence $i_1>\om_i$. 
		Since $\om_i\os i_1$ by \Cref{remark:Aj}, and $\om_i<2\mb{c}-\om_j<\mb{c}$, we get $i_1\ge2\mb{c}-\om_i$ by \Cref{remark:order in S}\ref{part:i}. This implies $i_1>\om_j$.
		Using $i_2\le\om_j+p-\bt$, we obtain $i_1>i_2-p+\bt>i_2-p$. Moreover, $\om_i\ge2\mb{c}-i_1>2\mb{c}-i_2\ge2\mb{c}-\om_j-p+\bt=j_1+1$. 
		Therefore, $j_1\ge i_2-2p+\bt-1$ implies that $i_2\le j_1+2p-\bt+1<\om_i+2p-\bt$. Since $\om_j\ge\mb{c}+\frac{\bt}{2}$, by \Cref{proposition: c and p relation}\ref{c-(bt+1)/2<=n-2p-1}, we get $2\mb{c}-\om_j\le n-2p-1$. Thus, $\om_i<2\mb{c}-\om_j$ implies that $\om_i<n-2p-1$. We have $\om_i<i_1<i_2$. Using \Cref{proposition:in M_bt+gm}\ref{satisfies x9}\ref{satisfies x9 ii}, it follows that $F'\prec F$, which is a contradiction. 
		Hence $\om_i\ge2\mb{c}-\om_j$.
		
		Now, suppose that $\om_i\ge\om_j$. Then $i_2>\om_j>2\mb{c}-\om_j\ge2\mb{c}-\om_i$. Since $\om_i\os i_2$ by \Cref{remark:Aj}, and $\om_i\ge\om_j>\mb{c}$, we have $i_2>\om_i$ by \Cref{remark:order in S}\ref{part:i}.
		We have $i_2\le\om_j+p-\bt\le\om_i+p-\bt$, $i_1\ge j_1\ge i_2-2p+\bt-1$ and $i_1\ge j_1=2\mb{c}-\om_j-p+\bt-1\ge2\mb{c}-\om_i-p+\bt-1$.
		If $\om_i>\om_j$, then $\om_j>\mb{c}$ implies that $\om_j\os\om_i$ by \Cref{remark:order in S}\ref{part:ii}, and hence $F'\ll F$ by \Cref{definition:order_ll}\ref{om_i<om_j}. Moreover, if $\om_i=\om_j$, then $j_1<i_1$ implies that $F'\ll F$ by \Cref{definition:order_ll}\ref{om_i=om_j}. In either case, $F'\ll F$. 
		It follows that $F'\prec F$ by \Cref{corollary:notin M_0}\ref{satisfies 6_7_8}, a contradiction. Hence $\om_i<\om_j$.
		
		We conclude that $2\mb{c}-\om_j\le\om_i<\om_j$.
	\end{proof}
	
	\begin{proposition}\label{proposition: al_i precedes lambda}
		Let $F\in\md$ such that $F\in\mcl{A}_i$ and $F^c=\{\om_i\}\sqcup\{i_1,i_2\}$. Let $u\in V(C_n^p)$. If either $u<i_1<\om_i$, or $\om_i<i_2<u $, then $\om_i\os u $.
	\end{proposition}
	
	\begin{proof}
		By \Cref{remark:Aj}, $\om_i\os i_1,i_2$. First, suppose $u<i_1<\om_i$. If $\om_i<\mb{c}$, then $u<\om_i$ implies that $\om_i\os u$ by \Cref{remark:order in S}\ref{part:i}. Now, let $\om_i>\mb{c}$. Since $\om_i\os i_1$ and $i_1<\om_i$, using \Cref{remark:order in S}\ref{part:ii}, we get $i_1<2\mb{c}-\om_i$. Thus $u<2\mb{c}-\om_i$, which implies that $\om_i\os u$ by \Cref{remark:order in S}\ref{part:ii}. 
		
		Now, suppose $\om_i<i_2<u $. Let $\om_i<\mb{c}$. Since $\om_i\os i_2$ and $\om_i<i_2$, from \Cref{remark:order in S}\ref{part:ii}, we get $i_2\ge2\mb{c}-\om_i$. Thus $u>2\mb{c}-\om_i$, and we get $\om_i\os u$. Now, if $\om_i>\mb{c}$, then $u>\om_i$ implies that $\om_i\os u$ by \Cref{remark:order in S}\ref{part:ii}. 
	\end{proof}
	
	\begin{proposition}\label{proposition: F_r1_s1 or F_r2_s2}
		Let $F\in\md$ such that $F\in\mcl{A}_i$ and $F^c=\{\om_i\}\sqcup\{i_1,i_2\}$ with $i_1<\om_i<i_2$. Let $u\in F$.
		\begin{enumerate}[label=(\arabic*)]
			\item\label{F_r1_s1 is a facet} 
			Suppose $u<i_1$. Then $(\Fuv{u}{i_1})^c=\{\om_i\}\sqcup\{u,i_2\}$. Moreover, if either $i_2\le i_1+2p$ and $i_2\nsim\om_i$, or $i_2<n-p$ and $i_1\nsim\om_i$, then $\Fuv{u}{i_1}\in\md$.
			
			\item\label{F_r2_s2 is a facet} 
			Suppose $u>i_2$. Then $(\Fuv{u}{i_2})^c=\{\om_i\}\sqcup\{i_1, u\}$. Moreover, if either $i_1\ge i_2-2p$ and $i_1\nsim\om_i$, or $i_1>p$ and $i_2\nsim\om_i$, then $\Fuv{u}{i_2}\in\md$.
			
			\item\label{F_s1_s2 precedes F_s} Let $F\in\mcl{M}_{\bt}$ for some $\bt\in [p-1]$ and $i_1=i_2-2p+\bt-1$. 
			\begin{enumerate}[label=(\roman*)] 
				\item \label{prob:f_r1_s1_3_i} If $u<i_1$ and $\Fuv{u}{i_1}\in\md$, then $\Fuv{u}{i_1}\prec F$. 
				\item \label{prob:f_r1_s1_3_ii} If $u>i_2$ and $\Fuv{u}{i_2}\in\md$, then $\Fuv{u}{i_2}\prec F$.
			\end{enumerate} 
		\end{enumerate} 
	\end{proposition}
	
	\begin{proof}
		\begin{enumerate}[label=(\arabic*)]
			\item Since $u<i_1<\om_i$, we get $\om_i\os u$ by \Cref{proposition: al_i precedes lambda}. Moreover, $\om_i\os i_2$ by \Cref{remark:Aj}, and $u<i_1<i_2$. Thus $(\Fuv{u}{i_1})^c=\{\om_i\}\sqcup\{u,i_2\}$. First, assume that $i_2\le i_1+2p$ and $i_2\nsim\om_i$. Since $i_2>\om_i$, we have $\om_i-i_2\ ( \text{mod} \ n)=n+\om_i-i_2$. Then $i_2\le i_1+2p\le\om_i+2p-1$ and $n\ge 6p-3$ implies that $\om_i-i_2\ ( \text{mod} \ n)\ge 4p-2$.
			Further, $p\ge2$ implies that $\om_i>i_2+2p\ ( \text{mod} \ n)$. Therefore, since $u<\om_i<i_2$ and $i_2\nsim\om_i$, we get $u\nsim i_2$ or $u\nsim\om_i$. This means that $i_2$ or $\om_i$ is an isolated vertex in $\cn{(\Fuv{u}{i_1})}$. Thus $\cn{(\Fuv{u}{i_1})}$ is disconnected, and hence $\Fuv{u}{i_1}\in\md$.
			
			Now, assume that $i_2<n-p$ and $i_1\nsim\om_i$. We have $u<i_1<\om_i<i_2<n-p$. Since $i_1\nsim\om_i$, we get $u\nsim\om_i$ and $u\nsim i_2$. Hence $u $ is an isolated vertex in $\cn{(\Fuv{u}{i_1})}$, and thus $\Fuv{u}{i_1}\in\md$.
			
			\item Since $\om_i<i_2<u$, \Cref{proposition: al_i precedes lambda} implies $\om_i\os u$. We have $\om_i\os i_1$ by \Cref{remark:Aj}, and $i_1<i_2<u$.
			Thus $(\Fuv{u}{i_2})^c=\{\om_i\}\sqcup\{i_1, u\}$. First, assume that $i_1\ge i_2-2p$ and $i_1\nsim\om_i$. 
			Since $i_1<\om_i$, $i_1\ge i_2-2p\ge\om_i-2p+1$ and $n\ge 6p-3$, it follows that $i_1-\om_i\ ( \text{mod} \ n)=n+i_1-\om_i\ge 4p-2$.
			Further, $p\ge2$ implies that $i_1>\om_i+2p\ (\text{mod} \ n)$. Using $i_1<\om_i<u$ and $i_1\nsim\om_i$, we obtain $\om_i\nsim u$ or $u\nsim i_1$. Hence $\Fuv{u}{i_2}\in\md$.
			
			Now, assume that $i_1>p$ and $i_2\nsim\om_i$. Since $p<i_1<\om_i<i_2<u $ and $i_2\nsim\om_i$, we get $u\nsim\om_i$ and $u\nsim i_1$. Hence $\Fuv{u}{i_1}\in\md$.
			
			\item First, let $u<i_1$.
			Then $(\Fuv{u}{i_1})^c=\{\om_i\}\sqcup\{u,i_2\}$ by \ref{F_r1_s1 is a facet}. Suppose $\Fuv{u}{i_1}\in\mcl{M}_{\al}$ for some $\bt\le\al\le p-1$. Since $u<\om_i<i_2$, $\Fuv{u}{i_1}$ satisfies one of the conditions from \X{\al}{3} to \X{\al}{8}. This implies $u\ge i_2-2p+\al-1$. It follows that $i_1>i_2-2p+\bt-1$ (as $u<i_1$ and $\al\ge\bt$), a contradiction to our assumption that $i_1=i_2-2p+\bt-1$.
			Therefore, $\Fuv{u}{i_1}\in\mcl{M}_{\al'}$ for some $0\le\al'<\bt$. Hence $\Fuv{u}{i_1}\prec F$ by \Cref{def:prec}\ref{def 2}.
			
			Now, let $u>i_2$.
			Then $(\Fuv{u}{i_2})^c=\{\om_i\}\sqcup\{i_1, u\}$ by \ref{F_r2_s2 is a facet}. Suppose $\Fuv{u}{i_2}\in\mcl{M}_{\gm}$ for some $\bt\le\gm\le p-1$. Then $i_1<\om_i<u $ implies that $\Fuv{u}{i_2}$ satisfies one of the conditions from \X{\gm}{3} to \X{\gm}{8}. We get $i_1\ge u-2p+\gm-1$. Since $i_2<u $ and $\gm\ge\bt$, it follows that $i_1>i_2-2p+\bt-1$, a contradiction.
			Therefore, $\Fuv{u}{i_1}\in\mcl{M}_{\gm'}$ for some $0\le\gm'<\bt$. Hence $\Fuv{u}{i_1}\prec F$ by \Cref{def:prec}\ref{def 2}.
		\end{enumerate}
		\vspace{-0.2 cm}
	\end{proof}
	
	\begin{proposition}\label{proposition:s_1<=r_1<r_2<=s_2}
		Let $F, F' \in\md$. Let $F \in\mcl{A}_i$ and $F' \in\mcl{A}_j$ for some $i\ne j$ such that $F^c=\{\om_i\}\sqcup\{i_1,i_2\}$ and $F'^c=\{\om_j\} \sqcup \{j_1, j_2\}$ with $j_1<\om_j<j_2$ and $j_1\le i_1<i_2\le j_2$.
		\begin{enumerate}[label=(\roman*)]
			\item \label{al_r<s_1} Suppose $\om_i<j_1$. Then $(F'(\om_i,\om_j))^c=\{j_1\}\sqcup\{\om_i,j_2\}$. If $j_2\le n-p$ and $j_1\nsim j_2$, then $F'(\om_i,\om_j)\in\md$. 
			
			\item \label{al_r>s_2} Suppose $\om_i>j_2$. Then $(F'(\om_i,\om_j))^c=\{j_2\}\sqcup\{j_1,\om_i\}$. If $j_1>p$ and $j_1\nsim j_2$, then $F'(\om_i,\om_j)\in\md$. 
			
			\item \label{F_al_s precedes F_s} Suppose $F'(\om_i,\om_j)\in\md$, and either $\om_i<j_1$ or $\om_i>j_2$. If $F'\in\mcl{M}_{\bt}$ for some $\bt\in[p-1]$ and $j_1=j_2-2p+\bt-1$, then $F'(\om_i,\om_j)\prec F'$.
		\end{enumerate}
	\end{proposition}
	
	\begin{proof}
		By \Cref{remark:Aj}, we have $\om_i\os i_1,i_2$ and $\om_j\os j_1,j_2$.
		\begin{enumerate}[label=(\roman*)]
			\item If $\om_j<\mb{c}$, then $j_1<\om_j$ implies $j_1<\mb{c}$. If $\om_j\ge\mb{c}$, then by \Cref{remark:order in S}\ref{part:ii}, $\om_j\os j_1$ and $j_1<\om_j$ imply $j_1<2\mb{c}-\om_j$, and hence $j_1<\mb{c}$. Therefore $\om_i<j_1<\mb{c}$. This gives $j_1\os\om_i$. We have $\om_i<\mb{c}$, $\om_i<j_1\le i_1$ and $\om_i\os i_1$. By \Cref{remark:order in S}\ref{part:i}, $i_1\ge2\mb{c}-\om_i$.
			It follows that $2\mb{c}-j_1<2\mb{c}-\om_i\le i_1<i_2\le j_2$. Since $j_1<\mb{c}$, we get $j_1\os j_2$ by \Cref{remark:order in S}\ref{part:i}. Moreover, $\om_i<j_1<j_2$ implies $\om_i<j_2$. Hence $(F'(\om_i,\om_j))^c=\{j_1\}\sqcup\{\om_i,j_2\}$.
			
			Suppose $j_2\le n-p$ and $j_1\nsim j_2$. Observe that since $\om_i\os i_1,i_2$, using the definition of $\Omega$, $\om_i\ne 0$. This means that $0<\om_i<j_1<j_2\le n-p$. Therefore, $j_1\nsim j_2$ implies that $\om_i\nsim j_2$. Hence $F'(\om_i,\om_j)\in\md$.
			
			\item If $\om_j<\mb{c}$, then $\om_j\os j_2$ and $\om_j<j_2$ imply $j_2\ge2\mb{c}-\om_j$ by \Cref{remark:order in S}\ref{part:i}, and hence $j_2>\mb{c}$. If $\om_j\ge\mb{c}$, then $\om_j<j_2$ implies $j_2>\mb{c}$. Therefore $\om_i>j_2>\mb{c}$. This means that $j_2\os\om_i$. 
			Since $\om_i>\mb{c}$, $i_2\le j_2<\om_i$ and $\om_i\os i_2$, by \Cref{remark:order in S}\ref{part:ii}, $i_2<2\mb{c}-\om_i$. This gives $j_1\le i_1<i_2<2\mb{c}-\om_i<2\mb{c}-j_2$. By \Cref{remark:order in S}\ref{part:ii}, $j_2>\mb{c}$ implies $j_2\os j_1$. Since $j_1<j_2<\om_i$, we have $j_1<\om_i$. Hence $(F'(\om_i,\om_j))^c=\{j_2\}\sqcup\{j_1,\om_i\}$.
			
			Suppose $j_1>p$ and $j_1\nsim j_2$. Then $p<j_1<j_2<\om_i$. Since $j_1\nsim j_2$, it follows that $\om_i\nsim j_1$. Hence $F'(\om_i,\om_j)\in\md$.
			
			\item First, let $\om_i<j_1$. By \ref{al_r<s_1}, $(F'(\om_i,\om_j))^c=\{j_1\}\sqcup\{\om_i,j_2\}$. Suppose $F'(\om_i,\om_j)\in\mcl{M}_{\al}$ for some $\bt\le\al\le p-1$. Then $\om_i<j_1<j_2$ implies $F'(\om_i,\om_j)$ satisfies one of the conditions from \X{\al}{3} to \X{\al}{8}. This yields $\om_i\ge j_2-2p+\al-1$. Since $\om_i<j_1$ and $\al\ge\bt$, we get $j_1>j_2-2p+\bt-1$, which contradicts our assumption $j_1=j_2-2p+\bt-1$.
			Therefore, $F'(\om_i,\om_j)\in\mcl{M}_{\al'}$ for some $0\le\al'<\bt$. Hence $F'(\om_i,\om_j)\prec F$ by \Cref{def:prec}\ref{def 2}.
			
			Now, let $\om_i>j_2$.
			Then $(F'(\om_i,\om_j))^c=\{j_2\}\sqcup\{j_1,\om_i\}$ by \ref{al_r>s_2}. Suppose $F'(\om_i,\om_j)\in\mcl{M}_{\gm}$ for some $\bt\le\gm\le p-1$. Since $j_1<j_2<\om_i$, it follows that $F'(\om_i,\om_j)$ satisfies one of the conditions from \X{\gm}{3} to \X{\gm}{8}. We get $j_1\ge\om_i-2p+\gm-1$. Using $j_2<\om_i$ and $\gm\ge\bt$, we obtain $j_1>j_2-2p+\bt-1$, a contradiction.
			Therefore, $F'(\om_i,\om_j)\in\mcl{M}_{\gm'}$ for some $0\le\gm'<\bt$. Hence $F'(\om_i,\om_j)\prec F'$ by \Cref{def:prec}\ref{def 2}.
		\end{enumerate}
		\vspace{-0.2 cm}
	\end{proof}
	
	\begin{proposition}\label{proposition:al_i j_1 j_2 is a facet}
		Let $F\in\md$ such that $F\in\mcl{A}_i$ and $F^c=\{\om_i\}\sqcup\{i_1,i_2\}$. Let $j_1,j_2\in V(C_n^p)$ and $F'^c=\{\om_i,j_1,j_2\}$. If $j_1\nsim j_2$, $j_1<\om_i<j_2$ and $j_1\le i_1<i_2\le j_2$, then $F'\in\md$.
	\end{proposition}
	
	\begin{proof}
		We have $j_1\nsim j_2$, $j_1<\om_i<j_2$ and $j_1\le i_1<i_2\le j_2$. Observe that if $\om_i\sim j_1$ and $\om_i\sim j_2$, then $\om_i\sim i_1$ and $\om_i\sim i_2$. This implies $\cn{F}$ is connected, a contradiction to $F\in\md$. 
		Hence $\om_i\nsim j_1$ or $\om_i\nsim j_2$. Therefore, since $j_1\nsim j_2$, it follows that $F'\in\md$.
	\end{proof}
	
	\begin{proposition}\label{proposition:al_r s_1 s_2}
		Let $F, F' \in\md$. Let $F \in\mcl{A}_i$ and $F' \in\mcl{A}_j$ such that $F^c=\{\om_i\}\sqcup\{i_1,i_2\}$ and $F'^c=\{\om_j\} \sqcup \{j_1, j_2\}$ with $j_1<\om_j<j_2$. 
		Suppose that either $\om_j<\om_i<2\mb{c}-\om_j$ with $\om_j<\mb{c}$, or $2\mb{c}-\om_j\le\om_i<\om_j$ with $\om_j>\mb{c}$. 
		\begin{enumerate}[label=(\roman*)]
			\item \label{al_r s_1 s_2 complement} Then $\om_i\in F'\setminus F$, $(F'(\om_i,\om_j))^c=\{\om_i\}\sqcup\{j_1,j_2\}$ and $j_1<\om_i<j_2$.
			
			\item \label{al_r s_1 s_2 is a facet} If $j_1\le i_1<i_2\le j_2$ and $j_1\nsim j_2$, then $F'(\om_i,\om_j)\in\md$.
			
			\item \label{al_r s_1 s_2 precedes F} If $F'(\om_i,\om_j)\in\md$, $F'\in\mcl{M}_{\bt}$ for some $\bt\in[p-1]$ and $j_1=j_2-2p+\bt-1$, then $F'(\om_i,\om_j)\prec F'$.
		\end{enumerate}
	\end{proposition}
	
	\begin{proof}
		If $\om_j<\om_i<2\mb{c}-\om_j$ with $\om_j<\mb{c}$, then $\om_i\os\om_j$ by \Cref{remark:order in S}\ref{part:iii}, and if $2\mb{c}-\om_j\le\om_i<\om_j$ with $\om_j>\mb{c}$, then $\om_i\os\om_j$ by \Cref{remark:order in S}\ref{part:iv}. Hence $\om_i\os\om_j$.
		\begin{enumerate}[label=(\roman*)]
			\item Since $\om_j\os j_1,j_2$ (\Cref{remark:Aj}) and $\om_i\os\om_j$, we get $\om_i\os j_1,j_2$. It follows that $\om_i\notin\{\om_j,j_1,j_2\}$, and thus $\om_i\in F'\setminus F$. Moreover, $(F'(\om_i,\om_j))^c=\{\om_i\}\sqcup\{j_1,j_2\}$. 
			
			First, assume that $\om_j<\om_i<2\mb{c}-\om_j$ and $\om_j<\mb{c}$. Since $\om_j\os j_2$ and $\om_j<j_2$, we get $j_2\ge2\mb{c}-\om_j$ by \Cref{remark:order in S}\ref{part:i}, which implies $j_1<\om_j<\om_i<2\mb{c}-\om_j\le j_2$. 
			Now, assume that $2\mb{c}-\om_j\le\om_i<\om_j$ and $\om_j>\mb{c}$. Since $\om_j\os j_1$ and $j_1<\om_j$, we get $j_1<2\mb{c}-\om_j$ by \Cref{remark:order in S}\ref{part:ii}, which implies $j_1<2\mb{c}-\om_j\le\om_i<\om_j<j_2$. In either case, $j_1<\om_i<j_2$. 
			
			\item From \ref{al_r s_1 s_2 complement}, $j_1<\om_i<j_2$. Therefore, $j_1\nsim j_2$ and $j_1\le i_1<i_2\le j_2$ imply $F'(\om_i,\om_j)\in\md$ by \Cref{proposition:al_i j_1 j_2 is a facet}.
			
			\item We have $(F'(\om_i,\om_j))^c=\{\om_i\}\sqcup\{j_1,j_2\}$ and $j_1<\om_i<j_2$ by \ref{al_r s_1 s_2 complement}. Suppose $F'(\om_i,\om_j)\in\mcl{M}_{\al}$ for some $\bt<\al\le p-1$. Then $F'(\om_i,\om_j)$ satisfies one of the conditions from \X{\al}{3} to \X{\al}{8}. 
			We get $j_1\ge j_2-2p+\al-1$. Since $\al>\bt$, it follows that $j_1>j_2-2p+\bt-1$, a contradiction. 
			Therefore, $F'(\om_i,\om_j)\in\mcl{M}_{\al'}$ for some $0\le\al'\le\bt$. Further, $\om_i\os\om_j$ implies that $F'(\om_i,\om_j)\ll F'$ by \Cref{definition:order_ll}\ref{om_i<om_j}. Hence $F'(\om_i,\om_j)\prec F'$ by \Cref{def:prec}.
		\end{enumerate}
		\vspace{-0.2 cm}
	\end{proof}
	
	\begin{proposition}\label{proposition:9 or 10 u<om_i}
		Let $F\in\md$ such that $F\in\mcl{A}_i$ and $F^c=\{\om_i\}\sqcup\{i_1,i_2\}$. Suppose $F\in\mcl{M}_{\bt}$ for some $\bt\in[p-1]$ such that $F$ satisfies \X{\bt}{9} or \X{\bt}{10}. Let $u\in F$.
		\begin{enumerate}[label=(\roman*)]
			\item \label{u<om_i} If $u<\om_i$, then $F(u,i_1)\in\md$ and $F(u,i_1)\prec F$.
			
			\item \label{om_i<u<i_1} If $\om_i<u<i_1$, $F(u,i_1)\in\md$ and $(F(u,i_1))^c=\{\om_i\}\sqcup\{u,i_2\}$, then $F(u,i_1)\prec F$.
		\end{enumerate} 
	\end{proposition}
	
	\begin{proof}
		Since $F$ satisfies \X{\bt}{9} or \X{\bt}{10}, it follows that $\om_i<i_1<i_2$. 
		\begin{enumerate}[label=(\roman*)]
			\item By \Cref{proposition:M_al conditions}\ref{ob_x9} and \ref{ob_x10}, we have $\om_i\ge2p$ and $i_1\sim i_2$. Therefore, $F\in\md$ implies $i_2\nsim\om_i$. Note that $(F(u,i_1))^c=\{\om_i,i_2, u\}$, where $u<\om_i<i_2\le n-1$ and $i_2\nsim\om_i$. 
			If $i_2\sim u $ and $u\sim\om_i$, then $\om_i\le i_2+2p\ (\text{mod} \ n)\le2p-1$, a contradiction. Hence $i_2\nsim u$ or $u\nsim\om_i$. Thus $F(u,i_1)\in\md$. 
			
			We now show that $F(u,i_1)\prec F$. By \Cref{remark:Aj}, $\om_i\os i_2$. Since $u\ne\om_i$, either $\om_i\os u$ or $u\os\om_i$.
			
			First, assume that $\om_i\os u$. Then $u<\om_i<i_2$ and $\om_i\os i_2$ imply $(F(u,i_1))^c=\{\om_i\}\sqcup\{u,i_2\}$. We show that $F \in\mcl{M}_{\alpha}$ for some $0\le\alpha<\beta$, which implies $F\prec F'$ by \Cref{def:prec}\ref{def 2}. Suppose $F(u,i_1)\in\mcl{M}_{\al}$ for some $\bt\le\al\le p-1$. 
			Since $u<\om_i<i_2$, $F(u,i_1)$ satisfies one of the conditions from \X{\al}{3} to \X{\al}{8}. If $F(u,i_1)$ satisfies \X{\al}{6} or \X{\al}{7} or \X{\al}{8}, then $i_2\sim\om_i$ by \Cref{proposition:M_al conditions}\ref{ob_x6}, \ref{ob_x7} and \ref{ob_x8}, a contradiction. 
			Therefore, $F(u,i_1)$ satisfies \X{\al}{3} or \X{\al}{4} or \X{\al}{5}. This yields $i_2\le u+2p-\al+1$. Using $u<\om_i$ and $\al\ge\bt$, we obtain $i_2<\om_i+2p-\bt+1$. If $F$ satisfies \X{\bt}{9}, then $i_2=\om_i+2p-\bt+1$, a contradiction. Hence $F$ satisfies \X{\bt}{10}. 
			By \Cref{proposition:M_al conditions}\ref{ob_x10}, $\om_i\ge\mb{c}$. If $F(u,i_1)$ satisfies \X{\al}{3} or \X{\al}{5}, then $\om_i<\mb{c}$; and if $F(u,i_1)$ satisfies \X{\al}{4}, then $\om_i<\mb{c}$ by \Cref{proposition:M_al conditions}\ref{ob_x4}. Both contradict $\om_i\ge\mb{c}$. Therefore, $F(u,i_1)\in\mcl{M}_{\al}$ for some $0\le\alpha<\beta$. 
			
			Now, assume that $u\os\om_i$. Since $\om_i\os i_2$, we get $u\os i_2$, and thus $\om_i<i_2$ implies $(F(u,i_1))^c=\{ u\}\sqcup\{\om_i,i_2\}$. Suppose $F(u,i_1)\in\mcl{M}_{\gm}$ for some $\gm\in[p-1]$. Then $u<\om_i<i_2$ implies that $F(u,i_1)$ satisfies \X{\gm}{9} or \X{\gm}{10}. By \Cref{proposition:M_al conditions}\ref{ob_x9} and \ref{ob_x10}, $i_2\sim\om_i$, a contradiction. Hence $F(u,i_1)\in\mcl{M}_0$, and thus $F(u,i_1)\prec F$ by \Cref{def:prec}\ref{def 2}.
			
			\item By \Cref{definition:order_ll}\ref{om_i=om_j}, $u<i_1$ implies $F(u,i_1)\ll F$. Suppose $F(u,i_1)\in\mcl{M}_{\al}$ for some $\bt<\al\le p-1$. We have $\om_i<u<i_2$. If $\om_i\le n-2p-2$, then $F(u,i_1)$ satisfies \X{\al}{9} and $F_i$ satisfies \X{\bt}{9}. This implies $i_2=\om_i+2p-\al+1<\om_i+2p-\bt+1=i_2$, a contradiction. 
			If $\om_i>n-2p-2$, then $F(u,i_1)$ satisfies \X{\al}{10} and $F_i$ satisfies \X{\bt}{10}. This implies $i_2=n-\al-1<n-\bt-1=i_2$, again a contradiction. Therefore, $F(u,i_1)\in\mcl{M}_{\al'}$ for some $0\le\al'\le\bt$. 
			Since $F(u,i_1)\ll F_i$, $\al'\le\bt$ implies $F(u,i_1)\prec F_i$ by \Cref{def:prec}.
		\end{enumerate}
		\vspace{-0.2 cm}
	\end{proof}
	
	We now show that $\prec$ provides a shelling order for the facets of $\Delta_3(C_n^p)$. By the definition of shellability, it suffices to prove that for any $F_r, F_s\in\md$ with $F_r \prec F_s$, there is an $F_t\in\md$ such that
	\begin{equation} \label{eqn:asterick}
		F_t \prec F_s \text{ and } F_t \cap F_s=F_s\setminus\{u\} \text{ for some } u\in F_s\setminus F_r. \hfill \tag{\raisebox{-0.35ex}{\huge$\ast$}}
	\end{equation}
	
	Let $F_r,F_s\in\md$ such that $F_r \prec F_s$. Let $F_r \in\mcl{A}_{r_0}$ and $F_s \in\mcl{A}_{s_0}$, where $F_r^c=\{\om_{r_0}\}\sqcup\{r_1,r_2\}$ and $F_s^c=\{\om_{s_0}\}\sqcup\{s_1,s_2\}$. 
	To simplify the notation, we write $\om_r$ and $\om_s$ for $\om_{r_0}$ and $\om_{s_0}$, respectively. We have $r_1<r_2$ and $s_1<s_2$. Moreover, $\om_r\os r_1,r_2$ and $\om_s\os s_1, s_2$ by \Cref{remark:Aj}. We aim to find an $F_t\in\md$ such that $F_t$ satisfies \eqref{eqn:asterick} for the pair $F_r\prec F_s$.
	
	\begin{remark}\label{remark:asterisk}
		For $u\in F_s\setminus F_r$ and $v\in F_s^c$, if $\Frs{u}{v} \in\md$ such that $\Frs{u}{v}\prec F_s$, then $F_t=\Frs{u}{v}$ satisfies \eqref{eqn:asterick} for the pair $F_r \prec F_s$. 
	\end{remark}
	
	We separately deal with the cases $F_s \in\mcl{M}_{0}$ and $F_s \notin \mcl{M}_{0}$. The former case is further divided into two subcases: $r_0=s_0$ (\Cref{lemma: case M_0 a}) and $r_0\ne s_0$ (\Cref{lemma: case M_0 b}). The latter case is addressed in \Cref{lemma: case M_al}.
	
	\begin{lemma}\label{lemma: case M_0 a}
		Suppose $F_s\in\mcl{M}_0$. If $F_r\in\mcl{A}_{s_0}$ (\textit{i.e.}, $r_0=s_0$), then there exists a facet $F_t\in\md$ that satisfies \eqref{eqn:asterick} for the pair $F_r\prec F_s$.
	\end{lemma}
	
	\begin{proof}
		Since $F_r\in\mcl{A}_{s_0}$, we have $\om_r=\om_s$. This means that $F_r^c=\{\om_s\}\sqcup\{r_1,r_2\}$ and $F_s^c=\{\om_s\}\sqcup\{s_1,s_2\}$. Moreover, $\om_s\os r_1,r_2,s_1,s_2$. Observe that if $\{r_1, r_2\} \cap \{s_1,s_2\}\ne \emptyset$, then \eqref{eqn:asterick} is satisfied by taking $F_t=F_r$. Henceforth, we assume that $\{r_1, r_2\} \cap \{s_1,s_2\}=\emptyset$. So $r_1,r_2\in F_s\setminus F_r$.
		
		We have $F_r \prec F_s$ and $F_s \in\mcl{M}_0$. By \Cref{def:prec}, $F_r \in\mcl{M}_0$ and $F_r \ll F_s$. Therefore, $\om_r=\om_s$ implies that either $r_1<s_1$, or $r_1=s_1$ and $r_2<s_2$ by \Cref{definition:order_ll}\ref{om_i=om_j}. 
		Since $r_1, r_2, s_1$ and $s_2$ are distinct, we get $r_1<s_1<s_2$. Recall that $r_1<r_2$. We now consider the following cases: (I) $s_1<s_2<\om_s$, (II) $s_1<\om_s<s_2$, (III) $\om_s<s_1<s_2$.
		\begin{enumerate}[label=(\MakeUppercase{\roman*})]
			\item $s_1<s_2<\om_s$.
			
			We show that $F_t=\Frs{r_1}{s_1}$ satisfies \eqref{eqn:asterick}. We have $\Frs{r_1}{s_1}=(F_s\setminus\{r_1\})\sqcup\{s_1\}$ and $r_1<s_1<s_2<\om_s$. Since $\om_s\os r_1,s_2$, we get $(\Frs{r_1}{s_1})^c=(F_s^c\setminus\{s_1\})\sqcup\{r_1\}=\{\om_s\}\sqcup\{r_1,s_2\}$. 
			By \Cref{proposition:i1_i2 nsim al_i}\ref{nsim i1} for $F_s$, we get $s_1\nsim\om_s$. Suppose $r_1 \sim \om_s$. Then $r_1<s_1<\om_s$ and $s_1\nsim\om_s$ imply $r_1\le\om_s+p\ (\text{mod} \ n)$.
			Further, since $F_r\in\md$, $\cn{F_r}$ is disconnected. Therefore $r_2\nsim\om_s$, which implies $r_2<\om_s$. Using \Cref{proposition:i1_i2 nsim al_i}\ref{nsim i1} for $F_r$, we get $r_1\nsim\om_s$, a contradiction. Hence $r_1\nsim\om_s$.
			
			If $s_2\nsim\om_s$, then $C_n^p[{(\Frs{r_1}{s_1}})^c]$ is disconnected and hence $\Frs{r_1}{s_1}\in\md$. Now, let $s_2\sim\om_s$. 
			Since $F_s\in\md$, $\cn{F_r}$ is disconnected. This implies $s_1\nsim s_2$, and therefore $r_1\nsim s_2$. Thus $\Frs{r_1}{s_1}\in\md$.
			
			Suppose $\Frs{r_1}{s_1}\in\mcl{M}_{\al}$ for some $\al\in[p-1]$. Since $r_1<s_2<\om_s$, $\Frs{r_1}{s_1}$ satisfies \X{\al}{1} or \X{\al}{2}. 
			First, suppose that $\Frs{r_1}{s_1}$ satisfies \X{\al}{1}. Then $\om_s<\mb{c}+\frac{p}{2}$ and $r_1=2\mb{c}-\om_s-p+\al-1$. 
			Since $s_1>r_1$, we get $s_1\ge2\mb{c}-\om_s-p+\al$. By \Cref{proposition:in M_bt+gm}\ref{satisfies x1}\ref{satisfies x1 ii}, $F_s\notin\mcl{M}_{0}$, a contradiction. Hence $\Frs{r_1}{s_1}$ satisfies \X{\al}{2}. Then $\om_s\ge\mb{c}+\frac{p}{2}$ and $r_1=\om_s-2p+\al-1$.
			Since $s_1>r_1$, we get $s_1\ge\om_s-2p+\al$. By \Cref{proposition:in M_bt+gm}\ref{satisfies x2}\ref{satisfies x2 ii}, $F_s\notin\mcl{M}_{0}$, again a contradiction. Hence $\Frs{r_1}{s_1}$ does not satisfy \X{\al}{2}.
			
			Therefore $\Frs{r_1}{s_1}\in\mcl{M}_0$. Since $r_1<s_1$, we have $\Frs{r_1}{s_1}\ll F_s$ by \Cref{definition:order_ll}\ref{om_i=om_j}. Hence $\Frs{r_1}{s_1}\prec F_s$ by \Cref{def:prec}\ref{def 1}. We have $r_1\in F_s\setminus F_r$. Thus $F_t=\Frs{r_1}{s_1}$ satisfies \eqref{eqn:asterick} by \Cref{remark:asterisk}. 
			
			\item $s_1<\om_s<s_2$.
			
			We have $\Frs{r_1}{s_2}=(F_s\setminus\{r_1\})\sqcup\{s_2\}$, $\om_s\os r_1,s_1$ and $r_1<s_1$. It follows that $(\Frs{r_1}{s_2})^c=\{\om_s\}\sqcup\{r_1,s_1\}$. There are two possibilities: either $\Frs{r_1}{s_2}\in\md$ or $\Frs{r_1}{s_2}\notin\md$. 
			\begin{enumerate}[label=(\alph*)]
				\item $\Frs{r_1}{s_2}\in\md$. 
				
				We first assume that $\Frs{r_1}{s_2}\in\mcl{M}_0$. Since $r_1<s_1$, we get $\Frs{r_1}{s_2}\ll F_s$ by \Cref{definition:order_ll}\ref{om_i=om_j}, and therefore $\Frs{r_1}{s_2} \prec F_s$ by \Cref{def:prec}\ref{def 1}. We have $r_1\in F_s\setminus F_r$. Hence, by taking $F_t=\Frs{r_1}{s_2}$, \eqref{eqn:asterick} is satisfied by \Cref{remark:asterisk}. 
				
				Now, assume that $\Frs{r_1}{s_2}\notin\mcl{M}_0$. We show that $F_t=\Frs{r_1}{s_1}$ satisfies \eqref{eqn:asterick}.
				
				Since $\Frs{r_1}{s_2}\notin\mcl{M}_0$, it follows that $\Frs{r_1}{s_2}\in\mcl{M}_{\al}$ for some $\al\in[p-1]$. Hence, $r_1<s_1<\om_s$ implies $\Frs{r_1}{s_2}$ satisfies \X{\al}{1} or \X{\al}{2}. By \Cref{proposition:M_al conditions}\ref{ob_x1} and \ref{ob_x2}, $\om_s>\mb{c}$.
				If $\Frs{r_1}{s_2}$ satisfies \X{\al}{1}, then $\om_s<\mb{c}+\frac{p}{2}$ and $r_1=2\mb{c}-\om_s-p+\al-1$; and if $\Frs{r_1}{s_2}$ satisfies \X{\al}{2}, then $\om_s\ge\mb{c}+\frac{p}{2}$ and $r_1=\om_s-2p+\al-1$. In either case, $r_1\ge\mb{c}-\frac{3p}{2}+\al-1$, which implies $r_1\ge\mb{c}-\frac{3p}{2}$. By \Cref{proposition: c and p relation}\ref{>=p}, $r_1\ge p$. Since $r_1<s_1<\om_s$ and $\Frs{r_1}{s_2}\in\md$, we get $r_1\nsim\om_s$ by \Cref{proposition:i1_i2 nsim al_i}\ref{nsim i1}. Therefore, $p\le r_1<\om_s<s_2\le n-1$ implies $r_1\nsim s_2$. 
				Using $\om_s\os r_1,s_2$ and $r_1<s_2$, we obtain $(\Frs{r_1}{s_1})^c=\{\om_s\}\sqcup\{r_1,s_2\}$. Observe that $\Frs{r_1}{s_1}\in M(\Delta_{3}(C_n^p))$. 
				
				\begin{claim}\label{claim:1case1}
					$\Frs{r_1}{s_1}\in\mcl{M}_{0}$.
				\end{claim}
				\begin{proof}[Proof of \Cref{claim:1case1}]
					Suppose $\Frs{r_1}{s_1}\in\mcl{M}_{\bt}$ for some $\bt \in[p-1]$. 
					We have $(\Frs{r_1}{s_1})^c$ $=\{\om_s\}\sqcup\{r_1,s_2\}$, $r_1<\om_s<s_2$ and $\om_s>\mb{c}$. By \Cref{proposition:om_i>c x6 x7 x8}, it follows that $\Frs{r_1}{s_1}$ satisfies \X{\bt}{6} or \X{\bt}{7} or \X{\bt}{8}.
					
					Suppose $\Frs{r_1}{s_1}$ satisfies \X{\bt}{6}. Then $r_1=2\mb{c}-\om_s-p+\bt-1$ and $s_2\le2\mb{c}-\om_s+p-1$. 
					Since $s_1>r_1$, $s_1\ge2\mb{c}-\om_s-p+\bt$. Moreover, $\om_s>\mb{c}$ implies $s_2<\om_s+p$. By \Cref{proposition:M_al conditions}\ref{ob_x6}, $\om_s\le\mb{c}+\frac{p-2}{2}<\mb{c}+\frac{p}{2}$. Thus $F_s\notin\mcl{M}_0$ by \Cref{proposition:in M_bt+gm}\ref{satisfies x6}\ref{satisfies x6 ii}, a contradiction.
					Hence $\Frs{r_1}{s_1}$ does not satisfy \X{\bt}{6}. 
					
					Suppose $\Frs{r_1}{s_1}$ satisfies \X{\bt}{7}. Then $\om_s<\mb{c}+\frac{p}{2}$, $r_1=s_2-2p+\bt-1$ and $2\mb{c}-\om_s+p\le s_2\le\om_s+p-\bt$. Since $s_1>r_1$, $s_1\ge s_2-2p+\bt$. Note that $s_2<\om_s+p$.  By \Cref{proposition:in M_bt+gm}\ref{satisfies x7}\ref{satisfies x7 ii}, $F_s\notin\mcl{M}_0$, a contradiction. Hence $\Frs{r_1}{s_1}$ does not satisfy \X{\bt}{7}.
					
					Suppose $\Frs{r_1}{s_1}$ satisfies \X{\bt}{8}. Then $\om_s\ge\mb{c}+\frac{p}{2}$ and $r_1=s_2-2p+\bt-1$. Since $s_1>r_1$, we get $s_1\ge s_2-2p+\bt$. 
					By \Cref{proposition:in M_bt+gm}\ref{satisfies x8}\ref{satisfies x8 ii}, $F_s\notin \mcl{M}_0$, a contradiction. Hence $\Frs{r_1}{s_1}$ does not satisfy \X{\bt}{8}.
					
					We conclude that $\Frs{r_1}{s_1}\in\mcl{M}_0$. This completes the proof of \Cref{claim:1case1}. 
				\end{proof} 
				
				Since $r_1<s_1$, we have $\Frs{r_1}{s_1}\ll F_s$ by \Cref{definition:order_ll}\ref{om_i=om_j}. Hence $\Frs{r_1}{s_1} \prec F_s$ by \Cref{def:prec}\ref{def 1}. Thus, $F_t=\Frs{r_1}{s_1}$ satisfies \eqref{eqn:asterick} by \Cref{remark:asterisk} (as $r_1\in F_s\setminus F_r$). 
				
				\item $\Frs{r_1}{s_2}\notin\md$.
				
				In this case, $C_n^p[(\Frs{r_1}{s_2})^c]$ is connected. We have $(\Frs{r_1}{s_2})^c=\{\om_s\}\sqcup\{r_1,s_1\}$, where $r_1<s_1<\om_s=\om_r$. Suppose $s_2\sim\om_s$. Then $F_s\in\md$ implies $s_1\nsim\om_s$. It follows that $r_1\sim\om_s$. Observe that $r_1\le\om_s+p\ (\text{mod} \ n)$. Moreover, since $F_r\in\md$, we have $r_2\nsim\om_s$. This implies $r_2<\om_s$. By \Cref{proposition:i1_i2 nsim al_i}\ref{nsim i1} for $F_r$, we get $r_1\nsim\om_s$, a contradiction. Hence $s_2\nsim\om_s$.
				
				We now show that if $\Frs{r_1}{s_1}\in\md$, then $F_t=\Frs{r_1}{s_1}$ satisfies \eqref{eqn:asterick}, otherwise $F_t=\Frs{r_2}{s_2}$ satisfies \eqref{eqn:asterick}.
				
				First, assume that $\Frs{r_1}{s_1}\in\md$.
				
				\begin{claim}\label{claim:2case1}
					$\Frs{r_1}{s_1}\in\mcl{M}_{0}$.
				\end{claim}
				\begin{proof}[Proof of \Cref{claim:2case1}]
					Suppose that $\Frs{r_1}{s_1}\in\mcl{M}_{\gm}$ for some $\gm \in[p-1]$. Since $\om_s\os r_1,s_2$ and $r_1<s_2$, we have ${(\Frs{r_1}{s_1}})^c=\{\om_s\}\sqcup\{r_1,s_2\}$. Therefore, $r_1<\om_s<s_2$ implies $\Frs{r_1}{s_1}$ satisfies one of the conditions from \X{\gm}{3} to \X{\gm}{8}. If $\Frs{r_1}{s_1}$ satisfies \X{\gm}{6} or \X{\gm}{7} or \X{\gm}{8}, then using \Cref{proposition:M_al conditions}\ref{ob_x6}, \ref{ob_x7} and \ref{ob_x8}, we get $s_2\sim\om_s$, a contradiction. It follows that $\Frs{r_1}{s_1}$ satisfies \X{\gm}{3} or \X{\gm}{4} or \X{\gm}{5}. 
					
					Suppose $\Frs{r_1}{s_1}$ satisfies \X{\gm}{3}. Then $\om_s<\mb{c}-\frac{p}{2}$ and $s_2=r_1+2p-\gm+1$. Since $r_1<s_1$, we get $s_2\le s_1+2p-\gm$. By \Cref{proposition:in M_bt+gm}\ref{satisfies x3}\ref{satisfies x3 ii}, $F_s\notin\mcl{M}_0$, a contradiction.
					Hence $\Frs{r_1}{s_1}$ does not satisfy \X{\gm}{3}.
					
					Suppose $\Frs{r_1}{s_1}$ satisfies \X{\gm}{4}. Then $r_1\ge\om_s-p+\gm$ and $s_2=r_1+2p-\gm+1$. Since $s_1>r_1$, we get $s_1>\om_s-p+\gm>\om_s-p$ and $s_2\le s_1+2p-\gm$.
					Moreover, $s_2\le2\mb{c}-\om_s+p-\gm$ by \Cref{proposition:M_al conditions}\ref{ob_x4}. By \Cref{corollary:notin M_0}\ref{satisfies 3_4_5}, $F_s\notin\mcl{M}_0$, a contradiction.
					Hence $\Frs{r_1}{s_1}$ does not satisfy \X{\gm}{4}. 
					
					Suppose $\Frs{r_1}{s_1}$ satisfies \X{\gm}{5}. Then $\om_s\le\mb{c}-\frac{\gm+1}{2}$, $s_2=2\mb{c}-\om_s+p-\gm$ and $r_1\ge s_2-2p+\gm$. 
					Since $s_1>r_1\ge s_2-2p+\gm$, $F_s$ satisfies \X{\gm}{5}. Therefore $F_s\in\mcl{M}_{\gm}$, a contradiction to $F_s\in\mcl{M}_0$. Hence $\Frs{r_1}{s_1}$ does not satisfy \X{\gm}{5}. 
					
					Therefore, $\Frs{r_1}{s_1}\in\mcl{M}_0$. This completes the proof of \Cref{claim:2case1}.
				\end{proof}
				
				Since $r_1<s_1$, we get $\Frs{r_1}{s_1}\ll F_s$ by \Cref{definition:order_ll}\ref{om_i=om_j}. Thus $\Frs{r_1}{s_1} \prec F_s$ by \Cref{def:prec}\ref{def 1}. Hence, $F_t=\Frs{r_1}{s_1}$ satisfies \eqref{eqn:asterick} by \Cref{remark:asterisk} (as $r_1\in F_s\setminus F_r$). 
				
				Now, assume that $\Frs{r_1}{s_1}\notin\md$. We have ${\Frs{r_1}{s_1}}=(F_s\setminus\{r_1\})\sqcup\{s_1\}$, where $r_1\in F_s$, $r_1<\om_s<s_2$ and $s_2\nsim\om_s$. By \Cref{proposition: om_i<=2p-1} (for $F=F_r$ and $F'=F_s$), we get $\om_s\le2p-1$, $\om_s<r_2<s_2$, $r_2\nsim\om_s$ and $r_1\nsim r_2$.
				Therefore, since $r_1<s_1<\om_s$, it follows that $r_2\nsim s_1$. Using $\om_s\os s_1,r_2$ and $s_1<r_2$, we get $(\Frs{r_2}{s_2})^c=\{\om_s\}\sqcup\{s_1,r_2\}$. Clearly, $\Frs{r_2}{s_2}\in\md$. 
				Further, $s_1<\om_s<r_2$ and $\om_s\le2p-1$ implies that $\Frs{r_2}{s_2}\in\mcl{M}_0$ by \Cref{proposition:in M_0}. Since $r_2<s_2$, $\Frs{r_2}{s_2}\ll F_s$ by \Cref{definition:order_ll}\ref{om_i=om_j}. Hence $\Frs{r_2}{s_2}\prec F_s$ by \Cref{def:prec}\ref{def 1}. We have $r_2\in F_s\setminus F_r$. Thus $F_t=\Frs{r_2}{s_2}$ satisfies \eqref{eqn:asterick} by \Cref{remark:asterisk}.
			\end{enumerate}	
			
			\item $\om_s<s_1<s_2$. 
			
			From \Cref{proposition:i1_i2 nsim al_i}\ref{nsim i2} for $F_s$, we have $s_2 \nsim\om_s$. Since $\om_s\os r_1$, it follows that either $r_1<\om_s$ or $r_1>\om_s$. 
			\begin{enumerate}[label=(\alph*)]
				\item $r_1<\om_s$. 
				
				Since $\om_s\os r_1,s_2$ and $r_1<s_2$, we have $(\Frs{r_1}{s_1})^c=\{\om_s\}\sqcup\{r_1,s_2\}$. We show that if $\Frs{r_1}{s_1}\in\md$, then $F_t=\Frs{r_1}{s_1}$ satisfies \eqref{eqn:asterick}. Otherwise, we find some $F_t\ne \Frs{r_1}{s_1}$ that satisfies \eqref{eqn:asterick}.
				\begin{itemize}
					\item $\Frs{r_1}{s_1}\in\md$.
					
					We show that $\Frs{r_1}{s_1}\in\mcl{M}_0$. On the contrary, suppose that $\Frs{r_1}{s_1}\in\mcl{M}_{\al}$ for some $\al \in[p-1]$. 
					Since ${(\Frs{r_1}{s_1}})^c=\{\om_s\}\sqcup\{r_1,s_2\}$ with $r_1<\om_s<s_2$, $\Frs{r_1}{s_1}$ satisfies one of the conditions from \X{\al}{3} to \X{\al}{8}. 
					
					If $\Frs{r_1}{s_1}$ satisfies \X{\al}{6} or \X{\al}{7} or \X{\al}{8}, then \Cref{proposition:M_al conditions}\ref{ob_x6}, \ref{ob_x7} and \ref{ob_x8} contradict $s_2\nsim\om_s$. Hence $\Frs{r_1}{s_1}$ satisfies \X{\al}{3} or \X{\al}{4} or \X{\al}{5}. 
					
					If $\Frs{r_1}{s_1}$ satisfies \X{\al}{3} or \X{\al}{5}, then $\om_s<\mb{c}$; and if $\Frs{r_1}{s_1}$ satisfies \X{\al}{4}, then $\om_s<\mb{c}$ by \Cref{proposition:M_al conditions}\ref{ob_x4}. Thus $\om_s<\mb{c}$. Then $\om_s\os s_1$ and $s_1>\om_s$ implies $s_1\ge2\mb{c}-\om_s$ by \Cref{remark:order in S}\ref{part:i}.
					
					Suppose $\Frs{r_1}{s_1}$ satisfies \X{\al}{3}. Then $\om_s<\mb{c}-\frac{p}{2}$ and $s_2=r_1+2p-\al+1$. By \Cref{proposition: c and p relation}\ref{c-p/2<=n-2p-1}, $\om_s<n-2p-1$. We have $s_1\ge2\mb{c}-\om_s>\mb{c}+\frac{p}{2}>\om_s+p$. 
					Moreover, $r_1<\om_s$ implies $s_2\le\om_s+2p-\al$.
					Therefore, by \Cref{proposition:in M_bt+gm}\ref{satisfies x9}\ref{satisfies x9 ii}, $F_s\notin\mcl{M}_{0}$, a contradiction.
					Hence $\Frs{r_1}{s_1}$ does not satisfy \X{\al}{3}. 
					
					Suppose $\Frs{r_1}{s_1}$ satisfies \X{\al}{4}. Then $\om_s\ge\mb{c}-\frac{p}{2}$, and by \Cref{proposition:M_al conditions}\ref{ob_x4}, $s_2\le2\mb{c}-\om_s+p-\al$. It follows that $s_2\le\om_s+2p-\al$. Moreover, $s_1\ge2\mb{c}-\om_s$ implies $s_1\ge s_2-p+\al>s_2-p$. 
					Therefore, if $\om_s\le n-2p-2$, then \Cref{proposition:in M_bt+gm}\ref{satisfies x9}\ref{satisfies x9 ii} contradicts $F_s\in\mcl{M}_{0}$. Hence $\om_s>n-2p-2$.
					By \Cref{proposition:M_al conditions}\ref{ob_x4}, $\om_s\le\mb{c}-\frac{\al+1}{2}$. Using $\al\ge 1$ and $n-2p-2<\om_s\le\mb{c}-\frac{\al+1}{2}$, we obtain $\om_s=\mb{c}-\frac{\al+1}{2}=n-2p-1$ and $p=2$ by \Cref{proposition: c and p relation}\ref{c-(bt+1)/2<=n-2p-1}. 
					Since $\al\in[p-1]=[1]$, $\al=1$. Thus $\om_s=\mb{c}-1$, and $2\mb{c}-\om_s\le s_1<s_2\le2\mb{c}-\om_s+p-\al$ implies $s_2=2\mb{c}-\om_s+1=\mb{c}+2=\om_s+3$. This means that $C_n^p[F_s^c]$ is connected, a contradiction to $F_s\in\md$. Therefore $\Frs{r_1}{s_1}$ does not satisfy \X{\al}{4}.
					
					Suppose $\Frs{r_1}{s_1}$ satisfies \X{\al}{5}. Then $\om_s\le\mb{c}-\frac{\al+1}{2}$, $s_2=2\mb{c}-\om_s+p-\al$ and $r_1\ge s_2-2p+\al$. We get $s_2\le r_1+2p-\al<\om_s+2p-\al$. Moreover, $s_1\ge2\mb{c}-\om_s$ implies $s_1\ge s_2-p+\al>s_2-p$. 
					If $\om_s\le n-2p-2$, then $F_s\notin\mcl{M}_{0}$ by \Cref{proposition:in M_bt+gm}\ref{satisfies x9}\ref{satisfies x9 ii}, a contradiction. Hence $\om_s>n-2p-2$. 
					Since $\al\ge 1$ and $\om_s\le\mb{c}-\frac{\al+1}{2}$, we get $\om_s=\mb{c}-\frac{\al+1}{2}=n-2p-1$ and $p=2$ by \Cref{proposition: c and p relation}\ref{c-(bt+1)/2<=n-2p-1}.
					This implies $\al=1$, and hence $ \om_s=\mb{c}-1$. By \Cref{proposition:M_al conditions}\ref{ob_x5}, $\om_s\ge\mb{c}-\frac{p-1}{2}=\mb{c}-\frac{1}{2}>\mb{c}-1$, again a contradiction. Hence $\Frs{r_1}{s_1}$ does not satisfy \X{\al}{5}. 
					
					We conclude that $\Frs{r_1}{s_1}\in\mcl{M}_0$. Since $r_1<s_1$, we get $\Frs{r_1}{s_1}\ll F_s$ by \Cref{definition:order_ll}\ref{om_i=om_j}. Hence $\Frs{r_1}{s_1}\prec F_s$ by \Cref{def:prec}\ref{def 1}. Moreover, $r_1\in F_s\setminus F_r$. Therefore $F_t=\Frs{r_1}{s_1}$ satisfies \eqref{eqn:asterick} by \Cref{remark:asterisk}.
					
					\item $\Frs{r_1}{s_1}\notin\md$. 
					
					We have ${(\Frs{r_1}{s_1}})^c=\{\om_s\}\sqcup\{r_1,s_2\}$, where $r_1\in F_s$, $r_1<\om_s<s_2$ and $s_2\nsim\om_s$. By \Cref{proposition: om_i<=2p-1} (for $F=F_r$ and $F'=F_s$), it follows that $\om_s\le2p-1$, $\om_s<r_2<s_2$, $r_2\nsim\om_s$ and $r_1\nsim r_2$.
					Suppose $s_1\sim\om_s$. Then $\om_s<s_1<s_2$ and $s_2\nsim\om_s$ imply $s_1\le\om_s+p\le3p-1$. By \Cref{proposition: c and p relation}\ref{range of c}, $s_1\le\mb{c}$. Since $\om_s<s_1$, $s_1\os\om_s$, a contradiction. Hence $s_1\nsim\om_s$.
					Note that either $s_1<r_2$ or $s_1>r_2$.
					
					First, let $s_1<r_2$. Then $r_1<\om_s<s_1<r_2$. Since $s_1\nsim\om_s$ and $r_1\nsim r_2$, it follows that $s_1\nsim r_1$. Observe that $(\Frs{r_1}{s_2})^c=\{\om_s\}\sqcup\{r_1,s_1\}$ and $\Frs{r_1}{s_2}\in\md$. Further, $r_1<\om_s<s_1$ and $\om_s\le2p-1$ implies that $\Frs{r_1}{s_2}\in\mcl{M}_0$ by \Cref{proposition:in M_0}. Since $r_1<s_1$, $\Frs{r_1}{s_2}\ll F_s$ by \Cref{definition:order_ll}\ref{om_i=om_j}. Hence $\Frs{r_1}{s_2}\prec F_s$ by \Cref{def:prec}\ref{def 1}. Therefore, $F_t=\Frs{r_1}{s_2}$ satisfies \eqref{eqn:asterick} by \Cref{remark:asterisk} (as $r_1\in F_s\setminus F_r$).
					
					Now, assume that $s_1>r_2$. Then $(\Frs{r_2}{s_1})^c=\{\om_s\}\sqcup\{r_2,s_2\}$. Since $r_2\nsim\om_s$ and $s_2\nsim\om_s$, it follows that $\Frs{r_2}{s_1}\in\md$. Suppose $\Frs{r_2}{s_1}\in\mcl{M}_{\al}$ for some $\al \in[p-1]$. Note that $\om_s<r_2<s_2$ implies $\Frs{r_2}{s_1}$ satisfies \X{\al}{9} or \X{\al}{10}. If $\Frs{r_2}{s_1}$ satisfies \X{\al}{9}, then $\om_s\le n-2p-2$, $r_2\ge\om_s+p+1$ and $s_2=\om_s+2p-\al+1$. Moreover, if $\Frs{r_2}{s_1}$ satisfies \X{\al}{10}, then $\om_s>n-2p-2$, $r_2\ge\om_s+p+1$ and $s_2=n-\al-1$.
					Since $s_1>r_2$, we get $s_1>\om_s+p+1$, and hence $F_s\notin\mcl{M}_0$ by \Cref{proposition:in M_bt+gm}\ref{satisfies x9}\ref{satisfies x9 i} and \ref{satisfies x10}\ref{satisfies x10 i}, a contradiction. 
					Thus $\Frs{r_2}{s_1}\in\mcl{M}_0$. Since $r_2<s_1$, \Cref{definition:order_ll}\ref{om_i=om_j} gives $\Frs{r_2}{s_1}\ll F_s$. By \Cref{def:prec}\ref{def 1}, $\Frs{r_2}{s_1}\prec F_s$. Therefore $F_t=\Frs{r_2}{s_1}$ satisfies \eqref{eqn:asterick} by \Cref{remark:asterisk} (as $r_2\in F_s\setminus F_r$). 
				\end{itemize}
				
				\item $r_1>\om_s$. 
				
				Recall that $r_1<s_1$. We have $(\Frs{r_1}{s_1})^c=\{\om_s\}\sqcup \{r_1,s_2\}$ with $\om_s<r_1<s_1<s_2$ and $s_2\nsim\om_s$. Suppose $\om_s\sim r_1$ and $r_1\sim s_2$. Then $r_1\le\om_s+p$ and $s_2\le r_1+p$. Hence $s_2\le\om_s+2p=\om_s+2p-1+1$ and $s_1>r_1\ge s_2-p$. 
				Therefore, if $\om_s\le n-2p-2$, then by \Cref{proposition:in M_bt+gm}\ref{satisfies x9} (here, if $s_2=\om_s+2p$, then condition \ref{satisfies x9 i} holds with $\beta=1$; and if $s_2<\om_s+2p$, then condition \ref{satisfies x9 ii} holds with $\beta=1$), $F_s\notin\mcl{M}_{0}$, a contradiction. Hence $\om_s>n-2p-2$.
				Now, suppose $s_2=n-1$. Then $r_2\ne s_2$ implies $\om_s<r_1<r_2<s_2$. Using $r_1\sim s_2$ and $s_2\nsim\om_s$, we get $r_1\sim r_2$. Therefore, $\om_s\sim r_1$ implies $C_n^p[F_r^c]$ is connected, contradicting $F_r\in\md$. This gives $s_2\le n-2=n-1-1$. 
				Since $\om_s<s_1<s_2$ and $s_1>s_2-p$, it follows that $F_s\notin\mcl{M}_0$ by \Cref{proposition:in M_bt+gm}\ref{satisfies x10} (for $\beta=1$), again a contradiction. Hence our assumption that $\om_s \sim r_1$ and $r_1\sim s_2$ is false. Thus $\om_s \nsim r_1$ or $r_1\nsim s_2$. Since $s_2\nsim\om_s$, we obtain $\Frs{r_1}{s_1}\in\md$. 
				
				Suppose $\Frs{r_1}{s_1}\in\mcl{M}_{\al}$ for some $\al\in[p-1]$. Then $\om_s<r_1<s_2$ implies that $\Frs{r_1}{s_1}$ satisfies \X{\al}{9} or \X{\al}{10}. If $\Frs{r_1}{s_1}$ satisfies \X{\al}{9}, then $\om_s\le n-2p-2$, $r_1\ge\om_s+p+1$ and $s_2=\om_s+2p-\al+1$; and if $\Frs{r_1}{s_1}$ satisfies \X{\al}{10}, then $\om_s>n-2p-2$, $r_1\ge\om_s+p+1$ and $s_2=n-\al-1$.
				In either case, $s_1>r_1$ implies $s_1>\om_s+p+1$. Observe that \Cref{proposition:in M_bt+gm}\ref{satisfies x9}\ref{satisfies x9 i} and \ref{satisfies x10}\ref{satisfies x10 i} yield $F_s\notin\mcl{M}_0$, a contradiction. Hence $\Frs{r_1}{s_1}\in\mcl{M}_0$. 
				Since $r_1<s_ 1$, \Cref{definition:order_ll}\ref{om_i=om_j} gives $\Frs{r_1}{s_1}\ll F_s$. By \Cref{def:prec}\ref{def 1}, $\Frs{r_1}{s_1}\prec F_s$. Therefore $F_t=\Frs{r_1}{s_1}$ satisfies \eqref{eqn:asterick} by \Cref{remark:asterisk} (as $r_1\in F_s\setminus F_r$). 
			\end{enumerate}		 
		\end{enumerate}
		\vspace{-0.2 cm}
	\end{proof}
	
	\begin{lemma}\label{lemma: case M_0 b}
		Suppose $F_s\in\mcl{M}_0$. If $F_r\notin\mcl{A}_{s_0}$ $(\textit{i.e.,}~ r_0\ne s_0)$, then there exists a facet $F_t\in\md$ that satisfies \eqref{eqn:asterick} for the pair $F_r\prec F_s$.
	\end{lemma} 
	
	\begin{proof}	
		Since $F_r\notin\mcl{A}_{s_0}$, we have $F_r^c=\{\om_r\}\sqcup\{r_1,r_2\}$ and $F_s^c=\{\om_s\}\sqcup\{s_1,s_2\}$ with $\om_r\ne\om_s$. By \Cref{def:prec}, $F_r \prec F_s$ and $F_s \in\mcl{M}_0$ imply $F_r\in\mcl{M}_0$ and $F_r\ll F_s$. Therefore $\om_r\os \om_s$ by \Cref{definition:order_ll}. 
		Recall that $\om_s\os s_1,s_2$. It follows that $\om_r\os s_1,s_2$. Hence $\om_r\in F_s\setminus F_r$. 
		
		For $v\in F_s^c$, we have $(\Frs{\om_r}{v})^c=(F_s^c\setminus\{v\})\sqcup\{\om_r\}$. Since $\om_r\os \om_s,s_1,s_2$, we get $\Frs{\om_r}{v}\in\mcl{A}_{r_0}$, and thus $\Frs{\om_r}{v}\ll F_s$ by \Cref{definition:order_ll}\ref{om_i<om_j}. 
		Therefore, if $\Frs{\om_r}{v}\in\mcl{M}_0\subseteq\md$, then $\Frs{\om_r}{v}\prec F_s$ by \Cref{def:prec}\ref{def 1}, and hence $F_t=\Frs{\om_r}{v}$ satisfies \eqref{eqn:asterick} by \Cref{remark:asterisk}. 
		
		In this proof, we show that $\Frs{\om_r}{v}\in\mcl{M}_0$ for some $v\in F_s^c$. 
		Clearly, $\om_r\os \om_s$ implies that $\om_s\ne\mb{c}$. Therefore, we have two cases: (A) $\om_s<\mb{c}$ and (B) $\om_s>\mb{c}$.
		
		\vspace{.5\baselineskip}
		
		\noindent{\bf Case A:} $\om_s<\mb{c}$.
		
		Since $\om_r\os \om_s$ implies, we get $\om_s<\om_r<2\mb{c}-\om_s$ by \Cref{remark:order in S}\ref{part:iii}.
		\begin{enumerate}[label=(\MakeUppercase{\roman*})]    
			\item $s_1<s_2<\om_s$.
			
			By \Cref{proposition:i1_i2 nsim al_i}\ref{nsim i1} for $F_s$, we have $s_1\nsim\om_s$. Moreover, $\Frs{\om_r}{\om_s}\in\mcl{A}_{r_0}$ and $s_1<s_2$ imply $(\Frs{\om_r}{\om_s})^c=\{\om_r\}\sqcup\{s_1,s_2\}$. 
			
			We first show that $\Frs{\om_r}{\om_s}\in\md$. Since $F_s\in\md$, $C_n^p[F_s^c]$ is disconnected. Therefore $\om_s\ge p+2$, and hence $2\mb{c}-\om_s\le2\mb{c}-p-2\le n-p-1$ (as $2\mb{c}\le n+1$). It follows that $s_1<s_2<\om_s<\om_r<2\mb{c}-\om_s\le n-p-1$.
			Since $s_1\nsim\om_s$, we get $s_1\nsim\om_r$. If $s_1\nsim s_2$, then $\Frs{\om_r}{\om_s}\in\md$. Now, let $s_1\sim s_2$. Then $C_n^p[F_s^c]$ is disconnected implies $s_2\nsim\om_s$, and thus $s_2\nsim\om_r$. Therefore $\Frs{\om_r}{\om_s}\in\md$.
			
			We now show that $\Frs{\om_r}{\om_s}\in\mcl{M}_0$. Suppose $\Frs{\om_r}{\om_s}\in\mcl{M}_{\al}$ for some $\al\in[p-1]$. Since $s_1<s_2<\om_r$, $\Frs{\om_r}{\om_s}$ satisfies \X{\al}{1} or \X{\al}{2}. Note that $s_1\nsim\om_s$ and $s_1<\om_s$ imply $s_1<\om_s-p$.
			If $\Frs{\om_r}{\om_s}$ satisfies \X{\al}{1}, then $s_1=2\mb{c}-\om_r-p+\al-1$. Using $\om_r<2\mb{c}-\om_s$, we obtain $s_1>\om_s-p+\al-1\ge\om_s-p$, a contradiction. Further, if $\Frs{\om_r}{\om_s}$ satisfies \X{\al}{2}, then $\om_r\ge\mb{c}+\frac{p}{2}$ and $s_1=\om_r-2p+\al-1$. It follows that $s_1\ge2\mb{c}-\om_r-p+\al-1\ge2\mb{c}-\om_r-p>\om_s-p$, again a contradiction. Hence $\Frs{\om_r}{\om_s}\in\mcl{M}_0$.
			
			\item $s_1<\om_s<s_2$. 
			
			Since $\om_s\os s_2$, $\om_s<\mb{c}$ and $s_2>\om_s$, \Cref{remark:order in S}\ref{part:i} implies $s_2\ge2\mb{c}-\om_s$. Thus $s_1<\om_s<\om_r<2\mb{c}-\om_s\le s_2$. Observe that $(\Frs{\om_r}{\om_s})^c=\{\om_r\}\sqcup\{s_1,s_2\}$. We show that if $\Frs{\om_r}{\om_s}\in\md$, then $\Frs{\om_r}{\om_s}\in\mcl{M}_0$. Otherwise, $\Frs{\om_r}{s_2}\in\mcl{M}_0$.
			\begin{enumerate}[label=(\alph*)]
				\item $\Frs{\om_r}{\om_s}\in\md$. 
				
				Suppose $\Frs{\om_r}{\om_s}\in\mcl{M}_{\al}$ for some $\al\in[p-1]$. Since $(\Frs{\om_r}{\om_s})^c=\{\om_r\}\sqcup\{s_1,s_2\}$ and $s_1<\om_r<s_2$, $\Frs{\om_r}{\om_s}$ satisfies one of the conditions from \X{\al}{3} to \X{\al}{8}.
				
				Suppose $\Frs{\om_r}{\om_s}$ satisfies \X{\al}{3}. Then $\om_r<\mb{c}-\frac{p}{2}$ and $s_2=s_1+2p-\al+1$. Since $\om_s<\om_r$, $F_s\notin\mcl{M}_0$ by \Cref{proposition:in M_bt+gm}\ref{satisfies x3}\ref{satisfies x3 i}, a contradiction. Hence $\Frs{\om_r}{\om_s}$ does not satisfy \X{\al}{3}.
				
				Suppose $\Frs{\om_r}{\om_s}$ satisfies \X{\al}{4}. Then $\om_r-p+\al\le s_1\le2\mb{c}-\om_r-p-1$ and $s_2=s_1+2p-\al+1$. If $\om_s<\mb{c}-\frac{p}{2}$, then $F_s\notin\mcl{M}_0$ by \Cref{proposition:in M_bt+gm}\ref{satisfies x3}\ref{satisfies x3 i}, a contradiction. 
				Hence $\om_s\ge\mb{c}-\frac{p}{2}$. 
				Since $\om_s<\om_r$, we have $\om_s-p<\om_s-p+\al<s_1<2\mb{c}-\om_s-p-1$. Therefore $F_s\notin\mcl{M}_0$ by \Cref{proposition:in M_bt+gm}\ref{satisfies x4}\ref{satisfies x4 i}, again a contradiction. It follows that $\Frs{\om_r}{\om_s}$ does not satisfy \X{\al}{4}. 
				
				Suppose $\Frs{\om_r}{\om_s}$ satisfies \X{\al}{5}. Then $s_2=2\mb{c}-\om_r+p-\al$ and $s_1\ge s_2-2p+\al=2\mb{c}-\om_r-p$. Since $\om_s<\om_r<2\mb{c}-\om_s$, it follows that $s_2<2\mb{c}-\om_s+p-\al$ and $s_1>\om_s-p$. Moreover, we have $s_2\le s_1+2p-\al$.  
				By \Cref{corollary:notin M_0}\ref{satisfies 3_4_5}, $F_s\notin\mcl{M}_0$, a contradiction. Hence $\Frs{\om_r}{\om_s}$ does not satisfy \X{\al}{5}. 
				
				Suppose $\Frs{\om_r}{\om_s}$ satisfies \X{\al}{6}. Then $s_1=2\mb{c}-\om_r-p+\al-1$ and $s_2\le s_1+2p-\al=2\mb{c}-\om_r+p-1$. If $\om_s<\mb{c}-\frac{p}{2}$, then $F_s\notin\mcl{M}_0$ by \Cref{proposition:in M_bt+gm}\ref{satisfies x3}\ref{satisfies x3 ii}, a contradiction. Hence $\om_s\ge\mb{c}-\frac{p}{2}$.
				Since $\om_s<\om_r<2\mb{c}-\om_s$, we have $s_1>\om_s-p+\al-1\ge\om_s-p$ and $s_2<2\mb{c}-\om_s+p-1$. 
				Now, if $s_1\le2\mb{c}-\om_s-p-1$, then $F_s\notin\mcl{M}_0$ by \Cref{proposition:in M_bt+gm}\ref{satisfies x4}\ref{satisfies x4 ii}; and if $s_1\ge2\mb{c}-\om_s-p$, then $F_s\notin\mcl{M}_0$ by \Cref{proposition:in M_bt+gm}\ref{satisfies x5}\ref{satisfies x5 ii} (for $\bt=1$). Each case contradicts $F_s\in\mcl{M}_0$. Therefore $\Frs{\om_r}{\om_s}$ does not satisfy \X{\al}{6}. 
				
				Suppose $\Frs{\om_r}{\om_s}$ satisfies \X{\al}{7}. Then $s_1=s_2-2p+\al-1$ and $2\mb{c}-\om_r+p\le s_2\le\om_r+p-\al$.  
				Observe that if $\om_s<\mb{c}-\frac{p}{2}$, then \Cref{proposition:in M_bt+gm}\ref{satisfies x3}\ref{satisfies x3 i} contradicts $F_s\in\mcl{M}_0$. Hence $\om_s\ge\mb{c}-\frac{p}{2}$. 
				Since $\om_r<2\mb{c}-\om_s$, we get $\om_s+p<s_2<2\mb{c}-\om_s+p-\al$. It follows that $\om_s-p<\om_s-p+\al\le s_1<2\mb{c}-\om_s-p-1$. 
				By \Cref{proposition:in M_bt+gm}\ref{satisfies x4}\ref{satisfies x4 i}, $F_s\notin\mcl{M}_0$, again a contradiction.
				Therefore $\Frs{\om_r}{\om_s}$ does not satisfy \X{\al}{7}.
				
				Suppose $\Frs{\om_r}{\om_s}$ satisfies \X{\al}{8}. Then $\om_r\ge\mb{c}+\frac{p}{2}$ and $s_1=s_2-2p+\al-1$. 
				Since $\om_r<2\mb{c}-\om_s$, we have $\om_s<2\mb{c}-\om_r\le\mb{c}-\frac{p}{2}$. Therefore, $F_s\notin\mcl{M}_0$ by \Cref{proposition:in M_bt+gm}\ref{satisfies x3}\ref{satisfies x3 i}, a contradiction. Hence $\Frs{\om_r}{\om_s}$ does not satisfy \X{\al}{8}.
				
				We conclude that $\Frs{\om_r}{\om_s}\in\mcl{M}_0$.
				
				\item $\Frs{\om_r}{\om_s}\notin\md$. 
				
				We first prove that $\Frs{\om_r}{s_2}\in\md$. We have $(\Frs{\om_r}{\om_s})^c=\{\om_r\}\sqcup\{s_1,s_2\}$ and $s_1<\om_s<\om_r<s_2$. Since $\Frs{\om_r}{\om_s}\notin\md$, $\cn{(\Frs{\om_r}{\om_s})}$ is connected. If $s_2\nsim\om_r$, then $s_1\sim s_2$ and $s_1\sim \om_r$, which implies $s_1\sim\om_s$. This contradicts $F_s\in\md$. 
				Hence $s_2\sim\om_r$. Since $C_n^p[F_s^c]$ is disconnected, $s_2\le\om_r+p$.
				
				Suppose $s_1\sim\om_r$. It follows that $s_1\ge\om_r-p$ or $s_1\le\om_r+p\ (\text{mod} \ n)$. First, let $s_1\ge\om_r-p$. Then $s_1>\om_s-p$. Moreover, $s_2\le\om_r+p$ implies $s_2\le s_1+2p=s_1+2p-1+1$. Since $\om_r<2\mb{c}-\om_s$, we have $s_2\le2\mb{c}-\om_s+p-1$. By \Cref{corollary:notin M_0}\ref{satisfies 3_4_5} (for $\bt=1$), $F_s\notin\mcl{M}_0$, a contradiction.
				Hence $s_1\le\om_r+p\ (\text{mod} \ n)$. Using $ s_1<\om_s<\om_r<s_2$, we get $s_1\sim s_2$ and $\om_r\ge n-p$. Thus $n-p\le\om_r<2\mb{c}-\om_s\le n+1-\om_s$ (as $2\mb{c}\le n+1$). This gives $\om_s<p+1$, which implies  
				$s_1\sim\om_s$, contradicting $F_s\in\md$. Therefore $s_1\nsim\om_r$. 
				
				Since $\cn{(\Frs{\om_r}{\om_s})}$ is connected and $s_1 \nsim\om_r$, it follows that $s_1\sim s_2$. Therefore $F_s\in\md$ implies $s_1\nsim\om_s$. 
				Using $\Frs{\om_r}{s_2}\in\mcl{A}_{r_0}$ and $s_1<\om_s$, we obtain $(\Frs{\om_r}{s_2})^c=\{\om_r\}\sqcup\{s_1,\om_s\}$. Observe that $\Frs{\om_r}{s_2}\in\md$. 
				
				We now show that $\Frs{\om_r}{s_2}\in\mcl{M}_0$. Suppose $\Frs{\om_r}{s_2}\in\mcl{M}_{\al}$ for some $\al\in[p-1]$. Then $s_1<\om_s<\om_r$ implies $\Frs{\om_r}{s_2}$ satisfies \X{\al}{1} or \X{\al}{2}.
				If $\Frs{\om_r}{s_2}$ satisfies \X{\al}{1}, then $s_1=2\mb{c}-\om_r-p+\al-1$. Since $\om_r<2\mb{c}-\om_s$, we get $s_1>\om_s-p+\al-1\ge\om_s-p$, which contradicts $s_1\nsim\om_s$. Hence $\Frs{\om_r}{s_2}$ satisfies \X{\al}{2}. This gives $\om_r\ge\mb{c}+\frac{p}{2}$ and $s_1=\om_r-2p+\al-1$. 
				It follows that $s_1\ge2\mb{c}-\om_r-p+\al-1>\om_s-p+\al-1$, again contradicting $s_1\nsim\om_s$.
				Therefore $\Frs{\om_r}{s_2}\in\mcl{M}_0$. 
			\end{enumerate}
			
			\item $\om_s<s_1<s_2$. 
			
			Recall that $\om_s<\om_r<2\mb{c}-\om_s$. By \Cref{remark:order in S}\ref{part:i}, $\om_s<\mb{c}$, $\om_s\os s_1$ and $s_1>\om_s$  imply $s_1\ge2\mb{c}-\om_s$. 
			Thus, $\om_s<\om_r<2\mb{c}-\om_s\le s_1<s_2$. Moreover, $s_2\nsim\om_s$ by \Cref{proposition:i1_i2 nsim al_i}\ref{nsim i2} for $F_s$. Since $\Frs{\om_r}{s_1}\in\mcl{A}_{r_0}$ and $\om_s<s_2$, we have $(\Frs{\om_r}{s_1})^c=\{\om_r\}\sqcup\{\om_s,s_2\}$. 
			
			We first show that $\Frs{\om_r}{s_1}\in\md$. If $s_2\nsim\om_r$, then $s_2\nsim\om_s$ implies $\Frs{\om_r}{s_1}\in\md$. So, let $s_2\sim \om_r$. Since $s_2\nsim\om_s$, we get $s_2\le\om_r+p$. It follows that $s_2<2\mb{c}-\om_s+p\le s_1+p$, and hence $s_1>s_2-p$. 
			Suppose $s_2\le\om_s+p+1$. Then $\om_s<s_1<s_2\le\om_s+p+1$, which implies $s_1\sim \om_s$ and $s_1\sim s_2$, contradicting $F_s\in\md$. Hence $s_2\ge\om_s+p+2$. Using $s_2<2\mb{c}-\om_s+p$, we obtain $\om_s+p+2<2\mb{c}-\om_s+p$. This yields $\om_s<\mb{c}-1=\mb{c}-\frac{1+1}{2}$. 
			By \Cref{proposition: c and p relation}\ref{c-(bt+1)/2<=n-2p-1}, $\om_s<n-2p-1$. Therefore, if $s_2\le\om_s+2p=\om_s+2p-1+1$, then $F_s\notin\mcl{M}_0$ by \Cref{proposition:in M_bt+gm}\ref{satisfies x9} (for $\bt=1$), a contradiction. Hence $s_2>\om_s+2p$.
			Since $s_2\le\om_r+p$, we have $\om_r>\om_s+p$, which implies $\om_s\nsim\om_r$ (as $s_2\nsim\om_s$). Therefore, $\Frs{\om_r}{s_1}\in\md$.
			
			\begin{claim}\label{claim:case_A(III)}
				$\Frs{\om_r}{s_1}\in\mcl{M}_{0}$.
			\end{claim}
			
			\begin{proof}[Proof of \Cref{claim:case_A(III)}]
				Suppose $\Frs{\om_r}{s_1}\in\mcl{M}_{\al}$ for some $\al\in[p-1]$. Then $\om_s<\om_r<s_2$ implies that $\Frs{\om_r}{s_1}$ satisfies one of the conditions from \X{\al}{3} to \X{\al}{8}. 
				
				Suppose $\Frs{\om_r}{s_1}$ satisfies \X{\al}{3}. Then $\om_r<\mb{c}-\frac{p}{2}$ and $s_2=\om_s+2p-\al+1$. Note that $2\mb{c}-\om_s>2\mb{c}-\om_r>\om_r+p>\om_s+p$. Since $s_1\ge2\mb{c}-\om_s$, $s_1>\om_s+p$. By \Cref{proposition: c and p relation}\ref{c-p/2<=n-2p-1}, $\mb{c}-\frac{p}{2}\le n-2p-1$, which implies $\om_s<\om_r<n-2p-1$. Therefore, $F_s\notin\mcl{M}_0$ by \Cref{proposition:in M_bt+gm}\ref{satisfies x9}\ref{satisfies x9 i}, a contradiction. Hence $\Frs{\om_r}{s_1}$ does not satisfy \X{\al}{3}.
				
				Suppose $\Frs{\om_r}{s_1}$ satisfies \X{\al}{4}. Then $\om_s\le2\mb{c}-\om_r-p-1$ and $s_2=\om_s+2p-\al+1$. By \Cref{proposition:M_al conditions}\ref{ob_x4}, $\om_r\le\mb{c}-\frac{\al+1}{2}$. 
				Therefore, since $\om_s<\om_r$, \Cref{proposition: c and p relation}\ref{c-(bt+1)/2<=n-2p-1} implies $\om_s<n-2p-1$. We have $s_1\ge2\mb{c}-\om_s\ge\om_r+p+1>\om_s+p+1$. 
				Thus $F_s\notin\mcl{M}_0$ by \Cref{proposition:in M_bt+gm}\ref{satisfies x9}\ref{satisfies x9 i}, a contradiction. Hence $\Frs{\om_r}{s_1}$ does not satisfy \X{\al}{4}.
				
				Suppose $\Frs{\om_r}{s_1}$ satisfies \X{\al}{5}. Then $\om_r\le\mb{c}-\frac{\al+1}{2}$, $s_2=2\mb{c}-\om_r+p-\al$ and $\om_s\ge s_2-2p+\al$. Since $\om_s<\om_r$, \Cref{proposition: c and p relation}\ref{c-(bt+1)/2<=n-2p-1} implies $\om_s<n-2p-1$. Using $s_1\ge2\mb{c}-\om_s$, we get $s_1>2\mb{c}-\om_r=s_2-p+\al>s_2-p$. 
				Therefore $F_s\notin\mcl{M}_0$ by \Cref{proposition:in M_bt+gm}\ref{satisfies x9}\ref{satisfies x9 ii}, a contradiction. Hence $\Frs{\om_r}{s_1}$ does not satisfy \X{\al}{5}. 
				
				Suppose $\Frs{\om_r}{s_1}$ satisfies \X{\al}{6}. Then $\om_r\ge\mb{c}+\frac{\al}{2}$, $\om_s=2\mb{c}-\om_r-p+\al-1$ and $s_2\le2\mb{c}-\om_r+p-1$.
				It follows that $s_2>s_1\ge2\mb{c}-\om_s=\om_r+p-\al+1\ge\mb{c}-\frac{\al}{2}+p+1\ge2\mb{c}-\om_r+p+1$, a contradiction. 
				Hence $\Frs{\om_r}{s_1}$ does not satisfy \X{\al}{6}.
				
				Suppose $\Frs{\om_r}{s_1}$ satisfies \X{\al}{7}. Then $\om_r<\mb{c}+\frac{p}{2}$, $\om_s=s_2-2p+\al-1$ and $s_2\le\om_r+p-\al$. Observe that $\om_r\ge\om_s+p+1$. Since $s_1\ge2\mb{c}-\om_s>\om_r$, we have $s_1>\om_s+p+1$. Moreover, $\om_s\le\om_r-p-1<\mb{c}-\frac{p}{2}$ implies $\om_s<n-p-1$ by \Cref{proposition: c and p relation}\ref{c-p/2<=n-2p-1}.  
				Therefore, \Cref{proposition:in M_bt+gm}\ref{satisfies x9}\ref{satisfies x9 i} contradicts $F_s\in\mcl{M}_0$. Hence $\Frs{\om_r}{s_1}$ does not satisfy \X{\al}{7}.
				
				Suppose $\Frs{\om_r}{s_1}$ satisfies \X{\al}{8}. Then $\om_r\ge\mb{c}+\frac{p}{2}$ and $\om_s=s_2-2p+\al-1$. Since $\om_r<2\mb{c}-\om_s$, it follows that $\om_s<2\mb{c}-\om_r\le\mb{c}-\frac{p}{2}\le n-2p-1$ by \Cref{proposition: c and p relation}\ref{c-p/2<=n-2p-1}. 
				Moreover, $s_1\ge2\mb{c}-\om_s$ implies $s_1>\om_r\ge\mb{c}-\frac{p}{2}+p\ge2\mb{c}-\om_r+p>\om_s+p$. Thus $F_s\notin\mcl{M}_0$ by \Cref{proposition:in M_bt+gm}\ref{satisfies x9}\ref{satisfies x9 i}, a contradiction. Hence $\Frs{\om_r}{s_1}$ does not satisfy \X{\al}{8}.
				
				Therefore $\Frs{\om_r}{s_1}\in\mcl{M}_0$.
			\end{proof}
		\end{enumerate}
		
		\noindent{\bf Case B:} $\om_s>\mb{c}$.
		
		Since $\om_r\os \om_s$, we get $2\mb{c}-\om_s\le\om_r<\om_s$ by \Cref{remark:order in S}\ref{part:iv}.
		\begin{enumerate}[label=(\MakeUppercase{\roman*})] 
			\item $s_1<s_2<\om_s$. 
			
			We have $\Frs{\om_r}{s_2}\in\mcl{A}_{r_0}$ and $s_1<\om_s$. Hence $(\Frs{\om_r}{s_2})^c=\{\om_r\}\sqcup\{s_1,\om_s\}$.
			Since $\om_s>\mb{c}$, $\om_s\os s_2$ and $s_2<\om_s$, \Cref{remark:order in S}\ref{part:ii} implies $s_2<2\mb{c}-\om_s$. 
			It follows that $s_1<s_2<2\mb{c}-\om_s\le\om_r<\om_s$. 
			By \Cref{proposition:i1_i2 nsim al_i}\ref{nsim i1} for $F_s$, $s_1\nsim\om_s$.
			
			If $s_1\nsim\om_r$, then $s_1\nsim\om_s$ implies $\Frs{\om_r}{s_2}\in\md$. Now, let $s_1\sim \om_r$. Since $s_1\nsim\om_s$, we get $s_1\ge\om_r-p$. 
			This gives $s_1\ge2\mb{c}-\om_s-p=2\mb{c}-\om_s-p+1-1$. Observe that if $\om_s<\mb{c}+\frac{p}{2}$, then \Cref{proposition:in M_bt+gm}\ref{satisfies x1} (for $\bt=1$) contradicts $F_s\in\mcl{M}_{0}$.
			Hence $\om_s\ge\mb{c}+\frac{p}{2}$. Now, if $s_1\ge\om_s-2p=\om_s-2p+1-1$, then $F_s\notin\mcl{M}_{0}$ by \Cref{proposition:in M_bt+gm}\ref{satisfies x2} (for $\bt=1$), again a contradiction. Therefore $s_1<\om_s-2p$.
			Since $s_1\ge\om_r-p$, we have $\om_r<\om_s-p$, which implies $\om_s\nsim\om_r$ (as $s_1\nsim\om_s$). Hence $\Frs{\om_r}{s_2}\in\md$.
			
			\begin{claim}\label{claim:case_B(I)}
				$\Frs{\om_r}{s_2}\in\mcl{M}_{0}$.
			\end{claim}
			\begin{proof}[Proof of \Cref{claim:case_B(I)}]
				Suppose $\Frs{\om_r}{s_2}\in\mcl{M}_{\al}$ for some $\al\in[p-1]$. Then $s_1<\om_r<\om_s$ implies $\Frs{\om_r}{s_2}$ satisfies one of the conditions from \X{\al}{3} to \X{\al}{8}.
				
				Suppose $\Frs{\om_r}{s_2}$ satisfies \X{\al}{3}. Then $\om_r<\mb{c}-\frac{p}{2}$ and $\om_s=s_1+2p-\al+1$. 
				Since $\om_r\ge2\mb{c}-\om_s$, $\om_s\ge2\mb{c}-\om_r>\mb{c}+\frac{p}{2} $. Moreover, $s_1=\om_s-2p+\al-1$. 
				Therefore $F_s\notin\mcl{M}_0$ by \Cref{proposition:in M_bt+gm}\ref{satisfies x2}\ref{satisfies x2 i}, a contradiction. 
				Hence $\Frs{\om_r}{s_2}$ does not satisfy \X{\al}{3}.
				
				Suppose $\Frs{\om_r}{s_2}$ satisfies \X{\al}{4}. Then $\om_r\ge\mb{c}-\frac{p}{2}$, $s_1\ge\om_r-p+\al$ and $\om_s=s_1+2p-\al+1$.
				This implies $\om_s\ge\om_r+p+1>\mb{c}+\frac{p}{2}$ and $s_1=\om_s-2p+\al-1$. 
				By \Cref{proposition:in M_bt+gm}\ref{satisfies x2}\ref{satisfies x2 i}, $F_s\notin\mcl{M}_0$, a contradiction.
				Hence $\Frs{\om_r}{s_2}$ does not satisfy \X{\al}{4}.
				
				Suppose $\Frs{\om_r}{s_2}$ satisfies \X{\al}{5}. Then $\om_r\le\mb{c}-\frac{\al+1}{2}$, $\om_s=2\mb{c}-\om_r+p-\al$ and $s_1\ge2\mb{c}-\om_r-p$. It follows that $s_1<s_2<2\mb{c}-\om_s=\om_r-p+\al\le\mb{c}+\frac{\al+1}{2}-p-1\le2\mb{c}-\om_r-p-1$, a contradiction.
				Hence $\Frs{\om_r}{s_2}$ does not satisfy \X{\al}{5}. 
				
				Suppose $\Frs{\om_r}{s_2}$ satisfies \X{\al}{6}. Then $s_1=2\mb{c}-\om_r-p+\al-1$ and $\om_s\le s_1+2p-\al$. Since $\om_s>\om_r$, we have $s_1\ge2\mb{c}-\om_s-p+\al$. 
				If $\om_s<\mb{c}+\frac{p}{2}$, then \Cref{proposition:in M_bt+gm}\ref{satisfies x1}\ref{satisfies x1 ii} contradicts $F_s\notin\mcl{M}_{0}$. Hence $\om_s\ge\mb{c}+\frac{p}{2}$. Note that $s_1\ge\om_s-2p+\al$.
				By \Cref{proposition:in M_bt+gm}\ref{satisfies x2}\ref{satisfies x2 ii}, $F_s\notin\mcl{M}_{0}$, again a contradiction. Therefore $\Frs{\om_r}{s_2}$ does not satisfy \X{\al}{6}.
				
				Suppose $\Frs{\om_r}{s_2}$ satisfies \X{\al}{7}. Then $s_1=\om_s-2p+\al-1$ and $\om_s\ge2\mb{c}-\om_r+p$.
				It follows that $s_1\ge2\mb{c}-\om_r-p+\al-1$. Since $\om_s>\om_r$, we get $s_1\ge2\mb{c}-\om_s-p+\al$.
				If $\om_s<\mb{c}+\frac{p}{2}$, then $F_s\notin\mcl{M}_{0}$ by \Cref{proposition:in M_bt+gm}\ref{satisfies x1}\ref{satisfies x1 ii}, a contradiction. 
				Hence $\om_s\ge\mb{c}+\frac{p}{2}$. Then $s_1=\om_s-2p+\al-1$ implies $F_s\notin\mcl{M}_0$ by \Cref{proposition:in M_bt+gm}\ref{satisfies x2}\ref{satisfies x2 i}, again a contradiction. Therefore $\Frs{\om_r}{s_2}$ does not satisfy \X{\al}{7}.
				
				Suppose $\Frs{\om_r}{s_2}$ satisfies \X{\al}{8}. Then $\om_r\ge\mb{c}+\frac{p}{2}$ and $s_1=\om_s-2p+\al-1$. Since $\om_s>\om_r$, we get $\om_s>\mb{c}+\frac{p}{2}$.
				Therefore, $F_s\notin\mcl{M}_0$ by \Cref{proposition:in M_bt+gm}\ref{satisfies x2}\ref{satisfies x2 i}, a contradiction. Hence $\Frs{\om_r}{s_2}$ does not satisfy \X{\al}{8}.
				
				Therefore $\Frs{\om_r}{s_2}\in\mcl{M}_0$.
			\end{proof}
			
			\item $s_1<\om_s<s_2$. 
			
			Recall that $2\mb{c}-\om_s\le\om_r<\om_s$.  Since $\om_s\os s_1$, $\om_s>\mb{c}$ and $s_1<\om_s$, we get $s_1<2\mb{c}-\om_s$ by \Cref{remark:order in S}\ref{part:ii}. Thus $s_1<2\mb{c}-\om_s\le\om_r<\om_s<s_2$.
			Observe that $(\Frs{\om_r}{\om_s})^c=\{\om_r\}\sqcup\{s_1, s_2\}$. We show that if $\Frs{\om_r}{\om_s}\in\md$, then $\Frs{\om_r}{\om_s}\in\mcl{M}_0$. Otherwise, $\Frs{\om_r}{s_1}\in\mcl{M}_0$.
			\begin{enumerate}[label=(\alph*)]
				\item $\Frs{\om_r}{\om_s}\in\md$. 
				
				Suppose $\Frs{\om_r}{\om_s}\in\mcl{M}_{\al}$ for some $\al\in[p-1]$. Then $s_1<\om_r<s_2$ implies $\Frs{\om_r}{\om_s}$ satisfies one of the conditions from \X{\al}{3} to \X{\al}{8}.
				
				Suppose $\Frs{\om_r}{\om_s}$ satisfies \X{\al}{3}. Then $\om_r<\mb{c}-\frac{p}{2}$ and $s_2=s_1+2p-\al+1$. 
				Since $\om_r\ge2\mb{c}-\om_s$, $\om_s\ge2\mb{c}-\om_r>\mb{c}+\frac{p}{2}$. Moreover, $s_1=s_2-2p+\al-1$. By \Cref{proposition:in M_bt+gm}\ref{satisfies x8}\ref{satisfies x8 i}, $F_s\notin \mcl{M}_0$, a contradiction. Hence $\Frs{\om_r}{\om_s}$ does not satisfy \X{\al}{3}.
				
				Suppose $\Frs{\om_r}{\om_s}$ satisfies \X{\al}{4}. Then $\om_r-p+\al\le s_1\le2\mb{c}-\om_r-p-1$ and $s_2=s_1+2p-\al+1$. If $\om_s\ge\mb{c}+\frac{p}{2}$, then \Cref{proposition:in M_bt+gm}\ref{satisfies x8}\ref{satisfies x8 i} contradicts $F_s\in\mcl{M}_0$. Hence $\om_s<\mb{c}+\frac{p}{2}$. 
				Since $\om_r\ge2\mb{c}-\om_s$, we get $2\mb{c}-\om_s-p+\al\le s_1\le\om_s-p-1$. It follows that $2\mb{c}-\om_s+p+1\le s_2\le\om_s+p-\al<\om_s+p$. 
				By \Cref{proposition:in M_bt+gm}\ref{satisfies x7}\ref{satisfies x7 i}, $F_s\notin\mcl{M}_0$, again a contradiction. Therefore $\Frs{\om_r}{\om_s}$ does not satisfy \X{\al}{4}. 
				
				Suppose $\Frs{\om_r}{\om_s}$ satisfies \X{\al}{5}. Then $s_2=2\mb{c}-\om_r+p-\al$ and $s_1\ge s_2-2p+\al=2\mb{c}-\om_r-p$.
				If $\om_s\ge\mb{c}+\frac{p}{2}$, then $F_s\notin\mcl{M}_0$ by \Cref{proposition:in M_bt+gm}\ref{satisfies x8}\ref{satisfies x8 ii}, a contradiction. Hence $\om_s<\mb{c}+\frac{p}{2}$.
				Since $2\mb{c}-\om_s\le\om_r<\om_s$, we get $s_2\le\om_s+p-\al<\om_s+p$ and $s_1\ge2\mb{c}-\om_s-p+1$. 
				Now, if $s_2\le2\mb{c}-\om_s+p-1$, then \Cref{proposition:in M_bt+gm}\ref{satisfies x6}\ref{satisfies x6 ii} (for $\bt=1$) contradicts $F_s\in\mcl{M}_0$; and if $s_2\ge2\mb{c}-\om_s+p$, then \Cref{proposition:in M_bt+gm}\ref{satisfies x7}\ref{satisfies x7 ii} contradicts $F_s\in\mcl{M}_0$. Therefore $\Frs{\om_r}{\om_s}$ does not satisfy \X{\al}{5}.
				
				Suppose $\Frs{\om_r}{\om_s}$ satisfies \X{\al}{6}. Then $s_1=2\mb{c}-\om_r-p+\al-1$ and $s_2\le s_1+2p-\al=2\mb{c}-\om_r+p-1$.
				Since $2\mb{c}-\om_s\le\om_r<\om_s$, we get $s_1>2\mb{c}-\om_s-p+\al-1$ and $s_2<\om_s+p$. By \Cref{corollary:notin M_0}\ref{satisfies 6_7_8}, $F_s\notin\mcl{M}_0$, a contradiction. Hence $\Frs{\om_r}{\om_s}$ does not satisfy \X{\al}{6}.
				
				Suppose $\Frs{\om_r}{\om_s}$ satisfies \X{\al}{7}. Then $s_1=s_2-2p+\al-1$ and $2\mb{c}-\om_r+p\le s_2\le\om_r+p-\al$. 
				Since $\om_r<\om_s$, $2\mb{c}-\om_s+p<s_2<\om_s+p-\al<\om_s+p$. If $\om_s<\mb{c}+\frac{p}{2}$, then \Cref{proposition:in M_bt+gm}\ref{satisfies x7}\ref{satisfies x7 i} contradicts $F_s\in\mcl{M}_0$; and 
				if $\om_s\ge\mb{c}+\frac{p}{2}$, then \Cref{proposition:in M_bt+gm}\ref{satisfies x8}\ref{satisfies x8 i} contradicts $F_s\in\mcl{M}_0$. 
				Hence $\Frs{\om_r}{\om_s}$ does not satisfy \X{\al}{7}. 
				
				Suppose $\Frs{\om_r}{\om_s}$ satisfies \X{\al}{8}. Then $\om_r\ge\mb{c}+\frac{p}{2}$ and $s_1=s_2-2p+\al-1$. 
				Since $\om_s>\om_r$, \Cref{proposition:in M_bt+gm}\ref{satisfies x8}\ref{satisfies x8 i} contradicts $F_s\in\mcl{M}_0$. Hence $\Frs{\om_r}{\om_s}$ does not satisfy \X{\al}{8}.
				
				Therefore $\Frs{\om_r}{\om_s}\in\mcl{M}_0$.
				
				\item $\Frs{\om_r}{\om_s}\notin\md$. 
				
				In this case, we show that $\Frs{\om_r}{s_1}\in\mcl{M}_0$. We have $(\Frs{\om_r}{\om_s})^c=\{\om_r\}\sqcup\{s_1,s_2\}$ and $s_1<\om_r<\om_s<s_2$. Since $\Frs{\om_r}{\om_s}\notin\md$, $\cn{(\Frs{\om_r}{\om_s})}$ is connected. If $s_1\nsim\om_r$, then $s_1\sim s_2$ and $s_2\sim \om_r$, which implies $s_2\sim\om_s$. This contradicts $F_s\in\md$. 
				Hence $s_1\sim\om_r$. Since $C_n^p[F_s^c]$ is disconnected, $s_1\ge\om_r-p$. 
				
				Suppose $s_2\sim\om_r$. It follows that $s_2\le\om_r+p$ or $s_2\ge\om_r-p\ (\text{mod} \ n)$.
				First, let $s_2\le\om_r+p$. Then $s_2<\om_s+p$. Moreover, $s_1\ge\om_r-p$ implies $s_1\ge s_2-2p=s_2-2p+1-1$. Since $\om_r\ge2\mb{c}-\om_s$, we have $s_1\ge2\mb{c}-\om_s-p=2\mb{c}-\om_s-p+1-1$. By \Cref{corollary:notin M_0}\ref{satisfies 6_7_8}, $F_s\notin\mcl{M}_0$, a contradiction. Hence $s_2\ge\om_r-p\ (\text{mod} \ n)$. 
				Since $s_1<\om_r<\om_s<s_2$, we get $s_1\sim s_2$ and $\om_r\le p-1$. This gives $n-\om_s\le2\mb{c}-\om_s\le\om_r\le p-1$ (as $2\mb{c}\ge n$), and thus $\om_s\ge n-p+1$. 
				Therefore $s_2\sim\om_s$, contradicting $F_s\in\md$. Hence $s_2\nsim\om_r$. 
				
				Since $\cn{(\Frs{\om_r}{\om_s})}$ is connected and $s_2\nsim\om_r$, we get $s_1\sim s_2$. Then $F_s\in\md$ implies $s_2\nsim\om_s$. 
				We have $\Frs{\om_r}{s_1}\in\mcl{A}_{r_0}$ and $\om_s<s_2$. Hence $(\Frs{\om_r}{s_1})^c=\{\om_r\}\sqcup\{\om_s, s_2\}$. Observe that $\Frs{\om_r}{s_1}\in\md$. 
				
				Suppose $\Frs{\om_r}{s_1}\in\mcl{M}_{\al}$ for some $\al\in[p-1]$. Since $\om_r<\om_s<s_2$, $\Frs{\om_r}{s_1}$ satisfies \X{\al}{9} or \X{\al}{10}.
				By \Cref{proposition:M_al conditions}\ref{ob_x9} and \ref{ob_x10}, $s_2\sim \om_s$, a contradiction. Therefore $\Frs{\om_r}{s_1}\in\mcl{M}_0$.
			\end{enumerate}
			
			\item $\om_s<s_1<s_2$. 
			
			We have $2\mb{c}-\om_s\le\om_r<\om_s$. Moreover, $(\Frs{\om_r}{\om_s})^c=\{\om_r\}\sqcup\{s_1,s_2\}$. Since $C_n^p[F_s^c]$ is disconnected, $\om_s\le n-p-3$. This implies $2\mb{c}-\om_s\ge2\mb{c}-n+p+3\ge p+3$ (as $2\mb{c}\ge n$), and hence $p+3\le2\mb{c}-\om_s\le\om_r<\om_s<s_1<s_2$.
			By \Cref{proposition:i1_i2 nsim al_i}\ref{nsim i2} for $F_s$, $s_2\nsim\om_s$, which implies $s_2\nsim\om_r$. If $s_1\nsim s_2$, then ${\Frs{\om_r}{\om_s}}\in\md$. Now, let $s_1\sim s_2$. Since $F_s\in\md$, $s_1\nsim\om_s$. It follows that $s_1\nsim\om_r$. Therefore $\Frs{\om_r}{\om_s}\in\md$.
			
			Suppose $\Frs{\om_r}{\om_s}\in\mcl{M}_{\al}$ for some $\al\in[p-1]$. Since $\om_r<s_1<s_2$,  $\Frs{\om_r}{\om_s}$ satisfies \X{\al}{9} or \X{\al}{10}.
			
			Suppose $\Frs{\om_r}{\om_s}$ satisfies \X{\al}{9}. Then $\om_r\le n-2p-2$, $s_1\ge\om_r+p+1$ and $s_2=\om_r+2p-\al+1$. This gives $s_1\ge s_2-p+\al>s_2-p$ and $s_2\le n-\al-1$. Since $\om_r<\om_s$, we have $s_2\le\om_s+2p-\al$. If $\om_s\le n-2p-2$, then \Cref{proposition:in M_bt+gm}\ref{satisfies x9}\ref{satisfies x9 ii} contradicts $F_s\in\mcl{M}_0$; and if $\om_s>n-2p-2$, then \Cref{proposition:in M_bt+gm}\ref{satisfies x10} contradicts $F_s\in\mcl{M}_0$. Hence $\Frs{\om_r}{\om_s}$ does not satisfy \X{\al}{9}.
			
			Suppose $\Frs{\om_r}{\om_s}$ satisfies \X{\al}{10}. Then $\om_r>n-2p-2$, $s_1\ge\om_r+p+1$ and $s_2=n-\al-1$. 
			It follows that $s_1>n-p-1=s_2-p+\al>s_2-p$. Since $\om_s>\om_r>n-2p-2$, \Cref{proposition:in M_bt+gm}\ref{satisfies x10}\ref{satisfies x10 i} contradicts $F_s\in\mcl{M}_0$. Hence $\Frs{\om_r}{\om_s}$ does not satisfy \X{\al}{10}.
			
			Therefore $\Frs{\om_r}{\om_s}\in\mcl{M}_0$. 
		\end{enumerate}
		\vspace{-0.2 cm}
	\end{proof}
	
	\begin{lemma}\label{lemma: case M_al}
		Suppose $F_s\notin\mcl{M}_0$, \textit{i.e.}, $F_s\in\mcl{M}_{\al}$ for some $\al \in[p-1]$. Then there exists a facet $F_t\in\md$ that satisfies \eqref{eqn:asterick} for the pair $F_r\prec F_s$.
	\end{lemma}
	
	\begin{proof}
		We have $F_r^c=\{\om_r\}\sqcup\{r_1,r_2\}$ and $F_s^c=\{\om_s\}\sqcup\{s_1,s_2\}$ with $F_r\prec F_s$. Moreover, $\om_r\os r_1,r_2$ and $\om_s\os s_1,s_2$.
		Observe that if $\om_r=\om_s$ and $\{r_1, r_2\}\cap\{s_1, s_2\}\ne\emptyset$, then $F_t=F_r$ satisfies \eqref{eqn:asterick}. Therefore, we assume that $\{r_1, r_2\}\cap\{s_1, s_2\}=\emptyset$ whenever $\om_r=\om_s$.
		
		Recall that $\Frs{u}{v}=(F_s\setminus\{u\})\sqcup\{v\}$, where $u\in F_s\setminus F_r$ and $v\in F_s^c$. If $\Frs{u}{v}\in\md$ and $\Frs{u}{v}\prec F_s$, then $F_t=\Frs{u}{v}$ satisfies \eqref{eqn:asterick} by \Cref{remark:asterisk}.
		Thus, it suffices to show that there exist $u\in F_s\setminus F_r$ and $v\in F_s^c$ such that $\Frs{u}{v}\in\md$ and $\Frs{u}{v}\prec F_s$.
		Since $F_s\in\mcl{M}_{\al}$, $F_s$ satisfies one of the conditions from \X{\al}{1} to \X{\al}{10}.
		\begin{enumerate}[label=(\arabic*)]
			\item $F_s$ satisfies \X{\al}{1} or \X{\al}{2}. We first prove the following claim.
			
			\begin{claim}\label{claim:case12condition}
				$s_1<s_2<\om_s$, $\mb{c}<\om_s<n-p$ and $s_1\nsim\om_s$.
			\end{claim}
			
			\begin{proof}[Proof of \Cref{claim:case12condition}]
				Since $F_s$ satisfies \X{\al}{1} or \X{\al}{2}, $s_1<s_2<\om_s$. Hence $s_1\nsim\om_s$ by \Cref{proposition:i1_i2 nsim al_i}\ref{nsim i1}. If $F_s$ satisfies \X{\al}{1}, then $\om_s>\mb{c}$ by \Cref{proposition:M_al conditions}\ref{ob_x1}; and if $F_s$ satisfies \X{\al}{2}, then $\om_s>\mb{c}$ by definition. Thus $\om_s>\mb{c}$. Moreover, \Cref{proposition:M_al conditions}\ref{ob_x1} and \ref{ob_x2} imply $\om_s<n-p$. 
			\end{proof} 
			
			We consider three cases: (I) $r_1<s_1$, (II) $s_1\le r_1<r_2\le\om_s$, (III) $s_1\le r_1$ and $\om_s<r_2$.    
			\begin{enumerate}[label=(\Roman*)]
				\item Let $r_1<s_1$. Since $s_1<s_2<\om_s$ by \Cref{claim:case12condition}, $r_1<s_1<\om_s$. By \Cref{proposition: al_i precedes lambda}, $\om_s\os r_1$. Observe that $r_1\in F_s\setminus F_r$ and $(\Frs{r_1}{s_2})^c=\{\om_s\}\sqcup\{r_1,s_1\}$. 
				By \Cref{claim:case12condition}, $\om_s<n-p$ and $s_1\nsim\om_s$, which implies $r_1\nsim\om_s$. Therefore $\Frs{r_1}{s_2}\in\md$.   
				
				We now show that $\Frs{r_1}{s_2}\prec F_s$. Suppose $\Frs{r_1}{s_2}\in\mcl{M}_{\al'}$ for some $\al\le\al'\le p-1$. We have $r_1<s_1<\om_s$.
				If $F_s$ satisfies \X{\al}{1}, then $\om_s<\mb{c}+\frac{p}{2}$ and $s_1=2\mb{c}-\om_s-p+\al-1$. Since $\om_s<\mb{c}+\frac{p}{2}$, $\Frs{r_1}{s_2}$ satisfies \X{\al'}{1}, which implies $r_1=2\mb{c}-\om_s-p+\al'-1\ge2\mb{c}-\om_s-p+\al-1=s_1$, a contradiction. 
				If $F_s$ satisfies \X{\al}{2}, then $\om_s\ge\mb{c}+\frac{p}{2}$ and $s_1=\om_s-2p+\al-1$. Since $\om_s\ge\mb{c}+\frac{p}{2}$, $\Frs{r_1}{s_2}$ satisfies \X{\al'}{2}, which implies $r_1=\om_s-2p+\al'-1\ge\om_s-2p+\al-1=s_1$, again a contradiction. 
				Therefore, $\Frs{r_1}{s_2}\in\mcl{M}_{\al''}$ for some $0\le\al''<\al$. 
				Hence $\Frs{r_1}{s_2}\prec F_s$ by \Cref{def:prec}\ref{def 2}.
				
				\item Let $s_1\le r_1<r_2\le\om_s$. By \Cref{claim:case12condition}, $\om_s>\mb{c}$. Hence $2\mb{c}-\om_s<\om_s$. We consider the following subcases based on the value of $\om_r$:
				\begin{enumerate}[label=(\alph*)]
					\item $\om_r<2\mb{c}-\om_s$.
					
					If $F_s$ satisfies \X{\al}{1}, then $s_1=2\mb{c}-\om_s-p+\al-1$. Moreover, if $F_s$ satisfies \X{\al}{2}, then $\om_s\ge\mb{c}+\frac{p}{2}$ and $s_1=\om_s-2p+\al-1$, which yields $s_1\ge2\mb{c}-\om_s-p+\al-1$. Since $r_1\ge s_1$ and $\al\ge 1$, we obtain $r_1\ge2\mb{c}-\om_s-p+\al-1>\om_r-p$ in both cases.
					Therefore, $\cn{F_r}$ is disconnected and $r_1<r_2$ imply $r_2>\om_r$. Using $\om_s>\mb{c}$ and $\om_r<2\mb{c}-\om_s$, we get $\om_r<\mb{c}$. By \Cref{remark:order in S}\ref{part:i}, $\om_r\os r_2$ implies $r_2\ge2\mb{c}-\om_r$, and hence $r_2>\om_s$. This contradicts $r_2\le\om_s$. Thus $\om_r\not<2\mb{c}-\om_s$. 
					
					\item $2\mb{c}-\om_s\le\om_r<\om_s$.
					
					Using $\om_s\os s_1$, $s_1<\om_s$ and $\om_s>\mb{c}$, \Cref{remark:order in S}\ref{part:ii} implies $s_1<2\mb{c}-\om_s$. Hence $s_1<\om_r$. Further, since $\om_s>\mb{c}$, $\om_r\os\om_s$ by \Cref{remark:order in S}\ref{part:iv}. Then $\om_s\os s_1,s_2$ implies $\om_r\os s_1,s_2$. Thus $\om_r\in F_s\setminus F_r$ and $(\Frs{\om_r}{s_2})^c=\{\om_r\}\sqcup\{s_1,\om_s\}$. 
					
					\begin{claim}\label{claim:1lastcase} 
						$\Frs{\om_r}{s_2}\in\md$.
					\end{claim}
					
					\begin{proof}[Proof of \Cref{claim:1lastcase}]
						By \Cref{claim:case12condition}, $s_1\nsim\om_s$. We have either $r_1<r_2<\om_r$ or $r_1<\om_r<r_2$ or $\om_r<r_1<r_2$.
						
						First, let $r_1<r_2<\om_r$. Then $r_1\nsim\om_r$ by \Cref{proposition:i1_i2 nsim al_i}\ref{nsim i1}. Since $s_1\le r_1<r_2<\om_r<\om_s$ and $s_1\nsim\om_s$, we get $s_1\nsim\om_r$. Hence $\Frs{\om_r}{s_2}\in\md$.
						
						Now, let $r_1<\om_r<r_2$. Since $F_r\in\md$, we have $r_1\nsim\om_r$ or $r_2\nsim\om_r$. Using $s_1\le r_1<\om_r<r_2\le\om_s$ and $s_1\nsim\om_s$, it follows that if $r_1\nsim\om_r$, then $s_1\nsim\om_r$; and if $r_2\nsim\om_r$, then $\om_s\nsim\om_r$. Hence $\Frs{\om_r}{s_2}\in\md$. 
						
						Finally, let $\om_r<r_1<r_2$. Then $r_2\nsim\om_r$ by \Cref{proposition:i1_i2 nsim al_i}\ref{nsim i2}. Since $s_1<\om_r$, we get $s_1<\om_r<r_1<r_2\le\om_s$. Therefore $s_1\nsim\om_s$ implies $\om_s\nsim\om_r$, and hence $\Frs{\om_r}{s_2}\in\md$.
						This completes the proof of \Cref{claim:1lastcase}.
					\end{proof}
					
					We now show that $\Frs{\om_r}{s_2}\prec F_s$. Since $\om_r\os\om_s$, we get $\Frs{\om_r}{s_2}\ll F_s$ by \Cref{definition:order_ll}\ref{om_i<om_j}. First, assume that $F_s$ satisfies \X{\al}{1}. Then $\om_s<\mb{c}+\frac{p}{2}$ and $s_1=2\mb{c}-\om_s-p+\al-1$. Since $s_1<\om_r<\om_s$ and $2\mb{c}-\om_s\le\om_r<\om_s$, \Cref{proposition:2c-i_2<=al_i<i_2} implies $\Frs{\om_r}{s_2}\prec F_s$. Now, let $F_s$ satisfies \X{\al}{2}. Then
					$s_1=\om_s-2p+\al-1$. Suppose $\Frs{\om_r}{s_2}\in\mcl{M}_{\al'}$ for some $\al<\al'\le p-1$. Since $s_1<\om_r<\om_s$, $\Frs{\om_r}{s_2}$ satisfies one of the conditions from \X{\al'}{3} to \X{\al'}{8}. Hence $s_1\ge\om_s-2p+\al'-1$. Then $\al'>\al$ implies $s_1>\om_s-2p+\al-1$, a contradiction. 
					Therefore, $\Frs{\om_r}{s_2}\in\mcl{M}_{\al''}$ for some $0\le\al''\le\al$. Since $\Frs{\om_r}{s_2}\ll F_s$, $\Frs{\om_r}{s_2}\prec F_s$ by \Cref{def:prec}.
					
					\item $\om_r=\om_s$. 
					
					In this case, $\{r_1, r_2\}\cap\{s_1, s_2\}=\emptyset$ by our assumption. Hence $s_1\le r_1$ implies $s_1<r_1$. Moreover, $r_1<r_2\le\om_s$ and $\om_r\ne r_2$ imply $r_1<r_2<\om_r$. If $F_s$ satisfies \X{\al}{1}, then $\om_r=\om_s<\mb{c}+\frac{p}{2}$ and $r_1>s_1=2\mb{c}-\om_r-p+\al-1$. Since $F_s\in\mcl{M}_{\al}$, it follows that $F_s\prec F_r$ by \Cref{proposition:in M_bt+gm}\ref{satisfies x1}\ref{satisfies x1 ii}, a contradiction. 
					If $F_s$ satisfies \X{\al}{2}, then $\om_r=\om_s\ge\mb{c}+\frac{p}{2}$ and $r_1>s_1=\om_r-2p+\al-1$. By \Cref{proposition:in M_bt+gm}\ref{satisfies x2}\ref{satisfies x2 ii}, $F_s\prec F_r$, again a contradiction. Therefore $\om_r\ne\om_s$.
					
					\item $\om_r>\om_s$.
					
					By \Cref{claim:case12condition}, $s_1<s_2<\om_s$, $\om_s>\mb{c}$ and $s_1\nsim\om_s$. Note that $\om_r>\mb{c}$, $\om_s\os\om_r$ and $s_1<\om_r$. Moreover, $\om_r\in F_s\setminus F_r$ and $(\Frs{\om_r}{s_2})^c=\{\om_s\}\sqcup\{s_1,\om_r\}$. We show that $\Frs{\om_r}{s_2}\in\md$ and $\Frs{\om_r}{s_2}\prec F_s$.
					
					We have $s_1\le r_1<r_2\le\om_s<\om_r$ and $\om_r>\mb{c}$. By \Cref{remark:order in S}\ref{part:ii}, $\om_r\os r_2$ implies $r_2<2\mb{c}-\om_r$. This gives $s_1\le r_1\le2\mb{c}-\om_r-2$. Hence $\om_r\le2\mb{c}-s_1-2$.
					
					If $F_s$ satisfies \X{\al}{1}, then $\om_s<\mb{c}+\frac{p}{2}$ and $s_1=2\mb{c}-\om_s-p+\al-1\ge2\mb{c}-\om_s-p$, which implies $\om_r\le\om_s+p-2<\mb{c}+\frac{3p-4}{2}$.
					If $F_s$ satisfies \X{\al}{2}, then $\om_s\ge\mb{c}+\frac{p}{2}$ and $s_1=\om_s-2p+\al-1\ge\om_s-2p$, which implies $\om_r\le2\mb{c}-\om_s+2p-2\le\mb{c}-\frac{p}{2}+2p-2=\mb{c}+\frac{3p-4}{2}$.
					Therefore, in both cases, $\om_r\le\mb{c}+\frac{3p-4}{2}$. By \Cref{proposition: c and p relation}\ref{<=n-p}, $\om_r<n-p$. Then, $ s_1<\om_s<\om_r$ and $s_1\nsim\om_s$ imply $s_1\nsim\om_r$. Hence $\Frs{\om_r}{s_2}\in\md$. 
					
					Suppose $\Frs{\om_r}{s_2}\in\mcl{M}_{\al'}$ for some $\al\le\al'\le p-1$. We have $s_1<\om_s<\om_r$. 
					
					First, assume that $F_s$ satisfies \X{\al}{1}. Then $\om_s<\mb{c}+\frac{p}{2}$ and $s_1=2\mb{c}-\om_s-p+\al-1$. Since $r_1\ge s_1$ and $\om_r>\om_s$, we get $r_1\ge2\mb{c}-\om_r-p+\al$. We have $r_1<r_2\le\om_s<\om_r$.   
					If $\om_r<\mb{c}+\frac{p}{2}$, then $F_s\prec F_r$ by \Cref{proposition:in M_bt+gm}\ref{satisfies x1}\ref{satisfies x1 ii}, a contradiction. Hence $\om_r\ge\mb{c}+\frac{p}{2}$.
					Note that $\om_s>\mb{c}$ implies $\Frs{\om_r}{s_2}$ satisfies \X{\al'}{6} or \X{\al'}{7} by \Cref{proposition:om_i>c x6 x7 x8}. This gives $s_1\ge\om_r-2p+\al'-1$. Using $r_1\ge s_1$ and $\al'\ge\al$, it follows that $r_1\ge\om_r-2p+\al-1$.   
					Since $\om_s\os\om_r$, $F_s\ll F_r$ by \Cref{definition:order_ll}\ref{om_i<om_j}. Therefore $F_s\prec F_r$ by \Cref{proposition:in M_bt+gm}\ref{satisfies x2}, again a contradiction. 
					
					We now assume that $F_s$ satisfies \X{\al}{2}. Then $\om_s\ge\mb{c}+\frac{p}{2}$ and $s_1=\om_s-2p+\al-1$. Since $\om_s\ge\mb{c}+\frac{p}{2}$, $\Frs{\om_r}{s_2}$ satisfies \X{\al'}{4} or \X{\al'}{6} or \X{\al'}{8}. If $\Frs{\om_r}{s_2}$ satisfies \X{\al'}{4}, then $\om_s\le\mb{c}-\frac{\al'+1}{2}$ by \Cref{proposition:M_al conditions}\ref{ob_x4}; and if $\Frs{\om_r}{s_2}$ satisfies \X{\al'}{6}, then $\om_s\le\mb{c}+\frac{p-2}{2}$ by \Cref{proposition:M_al conditions}\ref{ob_x6}. Both contradict $\om_s\ge\mb{c}+\frac{p}{2}$. Now, if $\Frs{\om_r}{s_2}$ satisfies \X{\al'}{8}, then $s_1=\om_r-2p+\al'-1$. Since $\om_r>\om_s$ and $\al'\ge\al$, we get $s_1>\om_s-2p+\al-1$, a contradiction. 
					
					Therefore $\Frs{\om_r}{s_2}\in\mcl{M}_{\al''}$ for some $0\le\al''<\al$. By \Cref{def:prec}\ref{def 2}, $\Frs{\om_r}{s_2}\prec F_s$.
				\end{enumerate}
				
				\item Let $s_1\le r_1$ and $\om_s<r_2$. By \Cref{claim:case12condition}, $s_1<s_2<\om_s$, $\om_s>\mb{c}$ and $s_1\nsim\om_s$. Note that $s_1<r_2$ and $r_2\in F_s\setminus F_r$. Since $r_2>\om_s>\mb{c}$, $\om_s\os r_2$. Thus $(\Frs{r_2}{s_2})^c=\{\om_s\}\sqcup\{s_1,r_2\}$. 
				If $F_s$ satisfies \X{\al}{1}, then $\om_s<\mb{c}+\frac{p}{2}$ and $s_1=2\mb{c}-\om_s-p+\al-1$, which implies $s_1\ge2\mb{c}-\om_s-p>\mb{c}-\frac{3p}{2}$.
				If $F_s$ satisfies \X{\al}{2}, then $\om_s\ge\mb{c}+\frac{p}{2}$ and $s_1=\om_s-2p+\al-1$, which implies $s_1\ge\om_s-2p\ge\mb{c}-\frac{3p}{2}$. In both cases, $s_1\ge p$ by \Cref{proposition: c and p relation}\ref{>=p}.
				Therefore, $s_1<\om_s<r_2$ and $s_1\nsim\om_s$ imply $s_1\nsim r_2$. Hence $\Frs{r_2}{s_2}\in\md$.
				\begin{itemize}
					\item $F_s$ satisfies \X{\al}{1}. 
					
					Since $\Frs{r_2}{s_2}\in\md$, $\Frs{r_2}{s_2}\in\mcl{M}_{\al'}$ for some $\al'\in[0,p-1]$. 
					If $0\le\al'<\al$, then $\Frs{r_2}{s_2}\prec F_s$ by \Cref{def:prec}\ref{def 2}. 
					
					So, assume that $\al\le\al'\le p-1$.
					Since $F_s$ satisfies \X{\al}{1}, we have $\om_s<\mb{c}+\frac{p}{2}$ and $s_1=2\mb{c}-\om_s-p+\al-1$. By \Cref{proposition:M_al conditions}\ref{ob_x1}, $\om_s\ge\mb{c}+\frac{\al+1}{2}$. Therefore, $\mb{c}+\frac{\al}{2}<\om_s<\mb{c}+\frac{p}{2}$.
					Note that $s_1\le r_1<r_2$, $s_1<\om_s<r_2$ and $r_2\in F_s$. Moreover, if $\om_r=\om_s$, then $\{r_1,r_2\}\cap\{s_1,s_2\}=\emptyset$, which implies $s_1<r_1$. By \Cref{proposition:x1 and x6} for $F=F_r$ and $F'=F_s$, $r_2\le2\mb{c}-\om_s+p$ and $2\mb{c}-\om_s\le\om_r<\om_s$. 
					
					Since $\om_s>\mb{c}$, we get $\om_r\os\om_s$ by \Cref{remark:order in S}\ref{part:iv}. Then $\om_s\os s_1,s_2$ implies $\om_r\os s_1,s_2$. Hence $\om_r\in F_s\setminus F_r$ and $(\Frs{\om_r}{s_2})^c=\{\om_r\}\sqcup\{s_1,\om_s\}$. 
					Using $\om_r<\om_s<r_2\le2\mb{c}-\om_s+p\le\om_r+p$, we obtain $r_1\le\om_r-p-1$ by \Cref{proposition:i1_i2 nsim al_i}\ref{sim i2}. 
					This implies $s_1\le r_1\le\om_r-p-1<\om_r<\om_s$. Since $s_1\nsim\om_s$, $s_1\nsim\om_r$. Hence $\Frs{\om_r}{s_2}\in\md$.
					By \Cref{definition:order_ll}\ref{om_i<om_j}, $\om_r\os\om_s$ implies $\Frs{\om_r}{s_2}\ll F_s$. We have $s_1<\om_r<\om_s$, $2\mb{c}-\om_s\le\om_r<\om_s$, $s_1=2\mb{c}-\om_s-p+\al-1$ and $\om_s<\mb{c}+\frac{p}{2}$. Therefore $\Frs{\om_r}{s_2}\prec F_s$ \Cref{proposition:2c-i_2<=al_i<i_2}. 
					
					\item $F_s$ satisfies \X{\al}{2}. 
					
					Then $\om_s\ge\mb{c}+\frac{p}{2}$ and $s_1=\om_s-2p+\al-1$. We show that $\Frs{r_2}{s_2}\prec F_s$. Suppose $\Frs{r_2}{s_2}\in\mcl{M}_{\al'}$ for some $\al\le\al'\le p-1$. Since $s_1<\om_s<r_2$ and $\om_s\ge\mb{c}+\frac{p}{2}>\mb{c}$, $\Frs{r_2}{s_2}$ satisfies \X{\al'}{6} or \X{\al'}{7} or \X{\al'}{8} by \Cref{proposition:om_i>c x6 x7 x8}. 
					If $\Frs{r_2}{s_2}$ satisfies \X{\al'}{6}, then $\om_s\le\mb{c}+\frac{p-2}{2}$ by \Cref{proposition:M_al conditions}\ref{ob_x6}; and if $\Frs{r_2}{s_2}$ satisfies \X{\al'}{7}, then $\om_s<\mb{c}+\frac{p}{2}$. Both contradict $\om_s\ge\mb{c}+\frac{p}{2}$. Now, if $\Frs{r_2}{s_2}$ satisfies \X{\al'}{8}, then $s_1=r_2-2p+\al'-1$. Since $\al'\ge\al$ and $r_2>\om_s$, it follows that $s_1>\om_s-2p+\al-1$, a contradiction.
					Therefore, $\Frs{r_2}{s_2}\in\mcl{M}_{\al''}$ for some $0\le\al''<\al$. Hence $\Frs{r_2}{s_2}\prec F_s$ by \Cref{def:prec}\ref{def 2}.
				\end{itemize}
			\end{enumerate}
			This completes the case.
			
			\item $F_s$ satisfies \X{\al}{3} or \X{\al}{4} or \X{\al}{5}. 
			
			\begin{claim}\label{claim:case345condition}
				$s_1<\om_s<s_2$, $\om_s<\mb{c}$, $s_2\le s_1+2p-\al+1\le s_1+2p$, $s_1>p$, $s_2\le n-p$, $s_1\sim\om_s$, $s_1\nsim s_2$ and $s_2\nsim\om_s$.
			\end{claim} 
			
			\begin{proof}[Proof of \Cref{claim:case345condition}]
				Since $F_s$ satisfies \X{\al}{3} or \X{\al}{4} or \X{\al}{5}, $s_1<\om_s<s_2$.
				If $F$ satisfies \X{\al}{3} or \X{\al}{5}, then $\om_s<\mb{c}$; and if $F$ satisfies \X{\al}{4}, then $\om_s<\mb{c}$ by \Cref{proposition:M_al conditions}\ref{ob_x4}. 
				Moreover, if $F_s$ satisfies \X{\al}{3} or \X{\al}{4}, then $s_2=s_1+2p-\al+1$; and if $F_s$ satisfies \X{\al}{5}, then $s_1\ge s_2-2p+\al$. Hence, in all these three cases, we get $\om_s<\mb{c}$ and $s_2\le s_1+2p-\al+1\le s_1+2p$.
				Using \Cref{proposition:M_al conditions}\ref{ob_x3}, \ref{ob_x4} and \ref{ob_x5}, it follows that $s_1>p$, $s_2\le n-p$ and $s_1\sim\om_s$. Therefore, $F_s\in\md$ implies $s_1\nsim s_2$ and $s_2\nsim\om_s$.
			\end{proof}
			
			We consider three cases: (I) $r_1<s_1$, (II) $s_1\le r_1<r_2\le s_2$, (III) $s_1\le r_1$ and $s_2<r_2$.
			\begin{enumerate}[label=(\Roman*)] 
				\item Let $r_1<s_1$. By \Cref{claim:case345condition}, $s_1<\om_s<s_2$, $s_2\le s_1+2p$ and $s_2\nsim\om_s$. Hence $r_1<s_1<\om_s<s_2$, which implies $r_1\in F_s\setminus F_r$. Moreover, $(\Frs{r_1}{s_1})^c=\{\om_s\}\sqcup\{r_1,s_2\}$ and $\Frs{r_1}{s_1}\in\md$ by \Cref{proposition: F_r1_s1 or F_r2_s2}\ref{F_r1_s1 is a facet}.
				
				If $F_s$ satisfies \X{\al}{3} or \X{\al}{4}, then $s_1=s_2-2p+\al-1$. Hence $\Frs{r_1}{s_1}\prec F_s$ by \Cref{proposition: F_r1_s1 or F_r2_s2}\ref{F_s1_s2 precedes F_s}\ref{prob:f_r1_s1_3_i}. Now, let $F_s$ satisfies \X{\al}{5}. Then $\om_s\le\mb{c}-\frac{\al+1}{2}$, $s_2=2\mb{c}-\om_s+p-\al$, and by \Cref{proposition:M_al conditions}\ref{ob_x5}, $\om_s\ge\mb{c}-\frac{p-1}{2}$.
				Suppose $\Frs{r_1}{s_1}\in\mcl{M}_{\al'}$ for some $\al<\al'\le p-1$. Since $r_1<\om_s<s_2$ and  $\mb{c}-\frac{p-1}{2}\le\om_s<\mb{c}$, $\Frs{r_1}{s_1}$ satisfies \X{\al'}{4} or \X{\al'}{5} or \X{\al'}{7}. If $\Frs{r_1}{s_1}$ satisfies \X{\al'}{7}, then $\om_s>\mb{c}$ by \Cref{proposition:M_al conditions}\ref{ob_x7}, a contradiction. Now, if $\Frs{r_1}{s_1}$ satisfies \X{\al'}{4}, then by \Cref{proposition:M_al conditions}\ref{ob_x4}, $s_2\le2\mb{c}-\om_s+p-\al'$; and if $\Frs{r_1}{s_1}$ satisfies \X{\al'}{5}, then $s_2=2\mb{c}-\om_s+p-\al'$. In both cases, $s_2\le2\mb{c}-\om_s+p-\al'<2\mb{c}-\om_s+p-\al$, a contradiction.
				Therefore, $\Frs{r_1}{s_1}\in\mcl{M}_{\al''}$ for some $0\le\al''\le\al$. By \Cref{definition:order_ll}\ref{om_i=om_j}, $r_1<s_1$ implies $\Frs{r_1}{s_1}\ll F_s$, and hence $\Frs{r_1}{s_1}\prec F_s$ by \Cref{def:prec}.
				
				\item Let $s_1\le r_1<r_2\le s_2$. By \Cref{claim:case345condition}, $s_1<\om_s<s_2$ and $\om_s<\mb{c}$. Since $\om_s\os s_2$, \Cref{remark:order in S}\ref{part:i} implies $s_2\ge2\mb{c}-\om_s$. Moreover, $2\mb{c}-\om_s>\mb{c}>\om_s$. Thus $s_1<\om_s<2\mb{c}-\om_s\le s_2$. We deal with the following subcases based on the value of $\om_r$:
				\begin{enumerate}[label=(\alph*)]
					\item $\om_r<s_1$.
					
					We have $\om_r<s_1<\om_s<s_2$. Clearly, $\om_r\in F_s\setminus F_r$. By \Cref{claim:case345condition}, $s_2\le n-p$ and $s_1\nsim s_2$. Therefore, by \Cref{proposition:s_1<=r_1<r_2<=s_2}\ref{al_r<s_1}, $(\Frs{\om_r}{\om_s})^c=\{s_1\}\sqcup\{\om_r,s_2\}$ and $\Frs{\om_r}{\om_s}\in\md$. 
					
					If $F_s$ satisfies \X{\al}{3} or \X{\al}{4}, then $s_1=s_2-2p+\al-1$, which implies $\Frs{\om_r}{\om_s}\prec F_s$ by \Cref{proposition:s_1<=r_1<r_2<=s_2}\ref{F_al_s precedes F_s}. 
					
					Now, let $F_s$ satisfies \X{\al}{5}. Then $\om_s\le\mb{c}-\frac{\al+1}{2}$ and $s_2=2\mb{c}-\om_s+p-\al$. Suppose $\Frs{\om_r}{\om_s}\in\mcl{M}_{\al'}$ for some $\al\le\al'\le p-1$. Since $\om_r<s_1<s_2$, $\Frs{\om_r}{\om_s}$ satisfies one of the conditions from \X{\al'}{3} to \X{\al'}{8}. 
					It follows that $s_2\le\om_r+2p-\al'+1$. Using $r_2\le s_2$ and $\al\le\al'$, we get $r_2\le\om_r+2p-\al+1$. By \Cref{proposition: c and p relation}\ref{c-(bt+1)/2<=n-2p-1},  $\om_r<\om_s\le\mb{c}-\frac{\al+1}{2}$ implies $\om_r\le n-2p-2$.
					Since $\om_r<s_1\le r_1$ and $s_1<\om_s<\mb{c}$, we have $\om_r<r_1$ and $\om_r<\mb{c}$. By \Cref{remark:order in S}\ref{part:i}, $\om_r\os r_1$ implies $r_1\ge2\mb{c}-\om_r$. This gives $r_2\le s_2=2\mb{c}-\om_s+p-\al<2\mb{c}-\om_r+p-\al\le r_1+p-\al$, and hence $r_1>r_2-p$. Note that $\om_r<r_1<r_2$.
					Moreover, $\om_r<s_1<\om_s<\mb{c}$ implies $\om_s\os\om_r$. By \Cref{definition:order_ll}\ref{om_i<om_j}, $F_s\ll F_r$. Therefore, $F_s\prec F_r$ by \Cref{proposition:in M_bt+gm}\ref{satisfies x9}, a contradiction.
					This means that $\Frs{\om_r}{\om_s}\in\mcl{M}_{\al''}$ for some $0\le\al''<\al$. Hence $\Frs{\om_r}{\om_s}\prec F_s$ by \Cref{def:prec}\ref{def 2}.
					
					\item $s_1\le\om_r\le\om_s$. 
					
					From \Cref{claim:case345condition}, $\om_s<\mb{c}$, $\om_s<s_2\le s_1+2p-\al+1$, $s_1\sim\om_s$ and $s_1\nsim s_2$. If $\om_r<\om_s$, then $\om_s<\mb{c}$ implies $\om_s\os\om_r$, and hence $F_s\ll F_r$ by \Cref{definition:order_ll}\ref{om_i<om_j}. Moreover, if $\om_r=\om_s$, then $\{r_1,r_2\}\cap\{s_1,s_2\}=\emptyset$ by our assumption, which implies $s_1<r_1$, and thus $F_s\ll F_r$ by \Cref{definition:order_ll}\ref{om_i=om_j}. In either case, $F_s\ll F_r$.
					
					\begin{claim}\label{claim:3lastcase}
						$r_1\ge\om_r+p+1$.
					\end{claim}
					
					\begin{proof}[Proof of \Cref{claim:3lastcase}]    
						Recall that $s_1\le r_1<r_2\le s_2$. Suppose $r_1<\om_r$. Then $s_1\le r_1<\om_r\le\om_s<s_2$. Since $s_1\sim\om_s$ and $s_1\nsim s_2$, we get $s_1\ge\om_s-p$, which implies $r_1\ge\om_r-p$. By \Cref{proposition:i1_i2 nsim al_i}\ref{sim i1}, $r_2>\om_r$. Thus $\om_r-p\le r_1<\om_r<r_2$. We have $r_2\le s_2\le s_1+2p-\al+1\le r_1+2p-\al+1$ and $F_s\ll F_r$. Observe that if $\om_r<\mb{c}-\frac{p}{2}$, then $F_s\prec F_r$ by \Cref{proposition:in M_bt+gm}\ref{satisfies x3}, a contradiction. So $\om_r\ge\mb{c}-\frac{p}{2}$.
						Since $\om_s\ge\om_r$, $F_s$ satisfies \X{\al}{4} or \X{\al}{5}. By \Cref{remark:x4_x5 and x6_x7}\ref{x4_x5}, $s_2\le2\mb{c}-\om_s+p-\al$. Using $r_2\le s_2$ and $\om_r\le\om_s$, we obtain $r_2\le2\mb{c}-\om_r+p-\al$. 
						By \Cref{corollary:notin M_0}\ref{satisfies 3_4_5}, $F_s\prec F_r$, again a contradiction. Hence $r_1>\om_r$. By \Cref{remark:order in S}\ref{part:i}, $\om_r\os r_1$ and $\om_r<\mb{c}$ imply $r_1\ge2\mb{c}-\om_r$. 
						
						If $F_s$ satisfies \X{\al}{3}, then $\om_r\le\om_s<\mb{c}-\frac{p}{2}$, which implies $r_1>\mb{c}+\frac{p}{2}>\om_r+p$. 
						Now, let $F_s$ satisfies \X{\al}{4} or \X{\al}{5}. Then $r_1<r_2\le2\mb{c}-\om_r+p-\al\le r_1+p-\al$, which implies $r_1\sim r_2$. Since $C_n^p[F_r^c]$ is disconnected and $r_1>\om_r$, we get $r_1\ge\om_r+p+1$. Thus $r_1\ge\om_r+p+1$.
					\end{proof}    
					
					We have $F_s\ll F_r$, $r_2\le s_2\le s_1+2p-\al+1\le\om_r+2p-\al+1$ and by \Cref{claim:3lastcase}, $\om_r<\om_r+p+1\le r_1<r_2$. 
					If $\om_r\le n-2p-2$, then \Cref{proposition:in M_bt+gm}\ref{satisfies x9} contradicts $F_r\prec F_s$. Hence $\om_r>n-2p-2$. It follows that $s_2\ge r_2>r_1\ge\om_r+p+1\ge n-p$. Further, $n\ge 6p-3$ implies $s_2>5p-3$.
					Since $\mb{c}-\frac{p}{2}\le n-2p-1$ by \Cref{proposition: c and p relation}\ref{c-p/2<=n-2p-1}, we get $\om_s\ge\om_r\ge n-2p-1\ge\mb{c}-\frac{p}{2}$. Therefore, $F_s$ satisfies \X{\al}{4} or \X{\al}{5}. By \Cref{remark:x4_x5 and x6_x7}\ref{x4_x5}, $s_2\le2\mb{c}-\om_s+p-\al$. This implies $s_2\le2\mb{c}-\om_s+p-\al\le2\mb{c}-\om_r+p-\al\le2\mb{c}-n+3p-\al+1$. Using $2\mb{c}\le n+1$ and $\al\ge 1$, it follows that $s_2\le 3p+1$. Thus $5p-3<s_2\le 3p+1$, which contradicts $p\ge2$. Hence, this case is not possible.
					
					\item $\om_s<\om_r<2\mb{c}-\om_s$.
					
					Recall that $s_1\le r_1<r_2\le s_2$. By  \Cref{claim:case345condition}, $s_1<\om_s<s_2$, $\om_s<\mb{c}$ and $s_1 \nsim s_2$. 
					Hence $\om_r\in F_s\setminus F_r$, $(\Frs{\om_r}{\om_s})^c=\{\om_r\}\sqcup\{s_1,s_2\}$ and $s_1<\om_r<s_2$ by \Cref{proposition:al_r s_1 s_2}\ref{al_r s_1 s_2 complement}. Since $s_1\nsim s_2$, $\Frs{\om_r}{\om_s}\in\md$ by \Cref{proposition:al_r s_1 s_2}\ref{al_r s_1 s_2 is a facet}. We show that 
					$\Frs{\om_r}{\om_s}\prec F_s$.
					
					If $F_s$ satisfies \X{\al}{3} or \X{\al}{4}, then $s_1=s_2-2p+\al-1$, which implies $\Frs{\om_r}{\om_s}\prec F_s$ by \Cref{proposition:al_r s_1 s_2}\ref{al_r s_1 s_2 precedes F}. 
					
					Now, let $F_s$ satisfies \X{\al}{5}. Then $s_1\ge2\mb{c}-\om_s-p$, $s_2=2\mb{c}-\om_s+p-\al$ and by \Cref{proposition:M_al conditions}\ref{ob_x5}, $\om_s>\mb{c}-\frac{p}{2}$.
					Suppose $\Frs{\om_r}{\om_s}\in\mcl{M}_{\al'}$ for some $\al<\al'\le p-1$.
					Since $\mb{c}-\frac{p}{2}<\om_s<\om_r<2\mb{c}-\om_s<\mb{c}+\frac{p}{2}$ and $s_1<\om_r<s_2$, it follows that $\Frs{\om_r}{\om_s}$ satisfies one of the conditions from \X{\al'}{4} to \X{\al'}{7}.
					
					Suppose $\Frs{\om_r}{\om_s}$ satisfies \X{\al'}{4} or \X{\al'}{5}. Then $s_2\le2\mb{c}-\om_r+p-\al'$ by \Cref{remark:x4_x5 and x6_x7}\ref{x4_x5}. Since $\om_s<\om_r$ and $\al'>\al$, we get $s_2<2\mb{c}-\om_s+p-\al$, a contradiction. 
					Hence $\Frs{\om_r}{\om_s}$ does not satisfy \X{\al'}{4} and \X{\al'}{5}.
					
					Suppose $\Frs{\om_r}{\om_s}$ satisfies \X{\al'}{6}. Then $\om_r\ge\mb{c}+\frac{\al'}{2}$ and $s_2\le2\mb{c}-\om_r+p-1$. 
					Since $\om_r>\mb{c}$, $2\mb{c}-\om_r<\om_r$. This implies $r_2\le s_2\le2\mb{c}-\om_r+p-1<\om_r+p$.
					We have $r_1\ge s_1\ge2\mb{c}-\om_s-p>\om_r-p$. Thus $\om_r-p<r_1<r_2<\om_r+p$. This contradicts $\cn{F_r}$ is disconnected. Hence $\Frs{\om_r}{\om_s}$ does not satisfy \X{\al'}{6}.
					
					Suppose $\Frs{\om_r}{\om_s}$ satisfies \X{\al'}{7}. Then $s_2\le\om_r+p-\al'$. Since $\om_r<2\mb{c}-\om_s$ and $\al'>\al$, it follows that $s_2<2\mb{c}-\om_s+p-\al$, a contradiction. Hence $\Frs{\om_r}{\om_s}$ does not satisfy \X{\al'}{7}.
					
					Therefore, $\Frs{\om_r}{\om_s}\in\mcl{M}_{\al''}$ for some $0\le\al''\le\al$. Since $\om_s<\om_r<2\mb{c}-\om_s$ and $\om_s<\mb{c}$, \Cref{remark:order in S}\ref{part:iii} implies $\om_r\os\om_s$. By \Cref{definition:order_ll}\ref{om_i<om_j}, $\Frs{\om_r}{\om_s}\ll F_s$. Hence $\Frs{\om_r}{\om_s}\prec F_s$ by \Cref{def:prec}.
					
					\item $2\mb{c}-\om_s\le\om_r\le s_2$.
					
					We have $s_2\le s_1+2p-\al+1$ and $s_1<\om_s<\mb{c}$ by \Cref{claim:case345condition}. Since $\om_s<\mb{c}$, $\om_s\os\om_r$ by \Cref{remark:order in S}\ref{part:i}. Hence $F_s\ll F_r$ by \Cref{definition:order_ll}\ref{om_i<om_j}. Moreover, $s_1\le r_1$ implies $s_2\le s_1+2p-\al+1\le r_1+2p-\al+1$.
					
					If $F_s$ satisfies \X{\al}{3}, then $s_1\ge\om_s-p+\al$ by \Cref{proposition:M_al conditions}\ref{ob_x3}. If \X{\al}{4}, then $s_1\ge\om_s-p+\al$ by definition. If $F_s$ satisfies \X{\al}{5}, then $\om_s\le\mb{c}-\frac{\al+1}{2}$ and $s_1\ge2\mb{c}-\om_s-p$, which implies $s_1\ge\mb{c}+\frac{\al+1}{2}-p\ge\om_s-p+\al+1$. 
					Thus $s_1\ge\om_s-p+\al$. Since $2\mb{c}-\om_s\le\om_r$, we get $r_1\ge s_1\ge2\mb{c}-\om_r-p+\al$. 
					
					Suppose $r_2<\om_r$. Then $r_1<r_2<\om_r$. If $\om_r<\mb{c}+\frac{p}{2}$, then $r_1\ge2\mb{c}-\om_r-p+\al$ implies $F_s\prec F_r$ by \Cref{proposition:in M_bt+gm}\ref{satisfies x1}\ref{satisfies x1 ii}, a contradiction. So, $\om_r\ge\mb{c}+\frac{p}{2}$.
					Using $\om_r\le s_2\le r_1+2p-\al+1$, we get $r_1\ge\om_r-2p+\al-1$. Since $F_s\ll F_r$, $F_s\prec F_r$ by \Cref{proposition:in M_bt+gm}\ref{satisfies x2}, again a contradiction. Hence $r_2>\om_r$. 
					
					Since $r_2\le s_2$ and $s_2\le r_1+2p-\al+1$, we get $r_1\ge r_2-2p+\al-1$. 
					
					We now show that $r_2<\om_r+p$. First, suppose $F_s$ satisfies \X{\al}{3}. Then $\om_s<\mb{c}-\frac{p}{2}$ and $s_2=s_1+2p-\al+1$. It follows that $\om_r\ge2\mb{c}-\om_s>\mb{c}+\frac{p}{2}>\om_s+p$. Since $s_1<\om_s$, we get $r_2\le s_2=s_1+2p-\al+1\le\om_s+2p-\al<\om_r+p-\al$. If $F_s$ satisfies \X{\al}{4} or \X{\al}{5}, then $s_2\le2\mb{c}-\om_s+p-\al$ by \Cref{remark:x4_x5 and x6_x7}\ref{x4_x5}, which implies $r_2\le s_2\le2\mb{c}-\om_s+p-\al\le\om_r+p-\al$. Thus $r_2\le\om_r+p-\al<\om_r+p$.
					
					Since $F_s\ll F_r$, it follows that $F_s\prec F_r$ by \Cref{corollary:notin M_0}\ref{satisfies 6_7_8}, a contradiction. Therefore, this case is not possible.
					
					\item $\om_r>s_2$.
					
					By \Cref{claim:case345condition}, $s_1<\om_s<s_2$ and $\om_s<\mb{c}$. Clearly, $\om_r\in F_s\setminus F_r$. Suppose $F_s$ satisfies \X{\al}{5}. Then $\om_s\le\mb{c}-\frac{\al+1}{2}$, $s_1\ge2\mb{c}-\om_s-p$ and $s_2=2\mb{c}-\om_s+p-\al$. Since $\al\le p-1$, we have $\om_r>s_2>2\mb{c}-\om_s>\mb{c}$. By \Cref{remark:order in S}\ref{part:ii}, $r_2\le s_2<\om_r$ and $\om_r\os r_2$ imply $r_2<2\mb{c}-\om_r$. It follows that $2\mb{c}-\om_s-p\le s_1\le r_1<r_2<2\mb{c}-\om_r<2\mb{c}-s_2\le\om_s-p+\al$.
					This gives $\om_s>\mb{c}-\frac{\al}{2}$, a contradiction. Therefore, $F_s$ satisfies \X{\al}{3} or \X{\al}{4}, and hence $s_1=s_2-2p+\al-1$. 
					Moreover, $s_1>p$ and $s_1\nsim s_2$ by \Cref{claim:case345condition}. Thus, \Cref{proposition:s_1<=r_1<r_2<=s_2}\ref{al_r>s_2} and \ref{F_al_s precedes F_s} yield $\Frs{\om_r}{\om_s}\in\md$ and $\Frs{\om_r}{\om_s}\prec F_s$, respectively.    
				\end{enumerate}
				
				\item Let $s_1\le r_1$ and $s_2<r_2$. From \Cref{claim:case345condition}, $s_1<\om_s<s_2$, $\om_s<\mb{c}$, $s_1>p$ and $s_2\nsim\om_s$. Note that $s_1<\om_s<r_2$ and $r_2\in F_s\setminus F_r$. By \Cref{proposition: F_r1_s1 or F_r2_s2}\ref{F_r2_s2 is a facet}, $(\Frs{r_2}{s_2})^c=\{\om_s\}\sqcup\{s_1,r_2\}$ and $\Frs{r_2}{s_2}\in\md$.
				
				If $F_s$ satisfies \X{\al}{3} or \X{\al}{4}, then $s_1=s_2-2p+\al-1$, and thus $\Frs{r_2}{s_2}\prec F_s$ by \Cref{proposition: F_r1_s1 or F_r2_s2}\ref{F_s1_s2 precedes F_s}\ref{prob:f_r1_s1_3_ii}. Now, let $F_s$ satisfies \X{\al}{5}. Then $s_1\ge2\mb{c}-\om_s-p$ and $s_2=2\mb{c}-\om_s+p-\al$. By \Cref{proposition:M_al conditions}\ref{ob_x5}, $\om_s>\mb{c}-\frac{p}{2}$. 
				Suppose $\Frs{r_2}{s_2}\in\mcl{M}_{\al'}$ for some $\al\le\al'\le p-1$. Since $s_1<\om_s<r_2$ and $\mb{c}-\frac{p}{2}<\om_s<\mb{c}$, $\Frs{r_2}{s_2}$ satisfies \X{\al'}{4} or \X{\al'}{5} or \X{\al'}{7}. If $\Frs{r_2}{s_2}$ satisfies \X{\al'}{4}, then $s_1\le2\mb{c}-\om_s-p-1$, a contradiction. If $\Frs{r_2}{s_2}$ satisfies \X{\al'}{7}, then $\om_s>\mb{c}$ by \Cref{proposition:M_al conditions}\ref{ob_x7}, again a contradiction. Hence $\Frs{r_2}{s_2}$ satisfies \X{\al'}{5}. This gives $s_2<r_2=2\mb{c}-\om_s+p-\al'\le2\mb{c}-\om_s+p-\al$, a contradiction.
				Therefore, $\Frs{r_2}{s_2}\in\mcl{M}_{\al''}$ for some $0\le\al''<\al$. By \Cref{def:prec}\ref{def 2}, $\Frs{r_2}{s_2}\prec F_s$. 
			\end{enumerate}
			
			\item $F_s$ satisfies \X{\al}{6} or \X{\al}{7} or \X{\al}{8}. 
			
			\begin{claim} \label{claim:case678condition}
				$s_1<\om_s<s_2$, $\om_s>\mb{c}$, $s_1\ge s_2-2p+\al-1\ge s_2-2p$, $s_2\le\om_s+p-\al$, $s_1>p$, $s_2\le n-p$, $s_2\sim\om_s$, $s_1\nsim s_2$ and $s_1\nsim\om_s$.
			\end{claim}
			
			\begin{proof}[Proof of \Cref{claim:case678condition}]
				Since $F_s$ satisfies \X{\al}{6} or \X{\al}{7} or \X{\al}{8}, $s_1<\om_s<s_2$. 
				
				If $F_s$ satisfies \X{\al}{6}, then $\om_s>\mb{c}$, $s_2\le s_1+2p-\al$ and by \Cref{proposition:M_al conditions}\ref{ob_x6}, $s_2\le\om_s+p-\al-1$. 
				If $F_s$ satisfies \X{\al}{7}, then $s_1=s_2-2p+\al-1$, $s_2\le\om_s+p-\al$, and by \Cref{proposition:M_al conditions}\ref{ob_x7}, $\om_s>\mb{c}$. 
				If $F_s$ satisfies \X{\al}{8}, then $\om_s>\mb{c}$, $s_1=s_2-2p+\al-1$, and by \Cref{proposition:M_al conditions}\ref{ob_x8}, $s_2\le\om_s+p-\al$. 
				Hence, in all these three cases, we get $\om_s>\mb{c}$, $s_1\ge s_2-2p+\al-1\ge s_2-2p$ and $s_2\le\om_s+p-\al$.
				
				Using \Cref{proposition:M_al conditions}\ref{ob_x6}, \ref{ob_x7} and \ref{ob_x8}, it follows that $s_1>p$, $s_2\le n-p$ and $s_2\sim\om_s$. Therefore, $F_s\in\md$ implies $s_1\nsim s_2$ and $s_1\nsim\om_s$.
			\end{proof}
			
			We consider three cases: (I) $r_1<s_1$, (II) $s_1\le r_1<r_2\le s_2$, (III) $s_1\le r_1$ and $s_2<r_2$.  
			\begin{enumerate}[label=(\Roman*)]
				\item Let $r_1<s_1$. By \Cref{claim:case678condition}, $s_1<\om_s<s_2$, $\om_s>\mb{c}$, $s_2\le n-p$ and $s_1\nsim\om_s$. Note that $r_1\in F_s\setminus F_r$.
				\begin{enumerate}[label=(\alph*)]
					\item Let $s_2<n-p$. Since $r_1<s_1$ and $s_1\nsim\om_s$, \Cref{proposition: F_r1_s1 or F_r2_s2}\ref{F_r1_s1 is a facet} implies $(\Frs{r_1}{s_1})^c$ $=\{\om_s\}\sqcup\{r_1,s_2\}$ and $\Frs{r_1}{s_1}\in\md$.
					
					If $F_s$ satisfies \X{\al}{7} or \X{\al}{8}, then $s_1=s_2-2p+\al-1$, and hence $\Frs{r_1}{s_1}\prec F_s$ by \Cref{proposition: F_r1_s1 or F_r2_s2}\ref{F_s1_s2 precedes F_s}\ref{prob:f_r1_s1_3_i}. Now, let $F_s$ satisfies \X{\al}{6}.
					Then $s_1=2\mb{c}-\om_s-p+\al-1$, $s_2\le2\mb{c}-\om_s+p-1$ and by \Cref{proposition:M_al conditions}\ref{ob_x6}, $\om_s<\mb{c}+\frac{p}{2}$. 
					Suppose $\Frs{r_1}{s_1}\in\mcl{M}_{\al'}$ for some $\al\le\al'\le p-1$. Since $r_1<\om_s<s_2$ and $\mb{c}<\om_s<\mb{c}+\frac{p}{2}$, $\Frs{r_1}{s_1}$ satisfies \X{\al'}{4} or \X{\al'}{6} or \X{\al'}{7}.
					If $\Frs{r_1}{s_1}$ satisfies \X{\al'}{4}, then $\om_s<\mb{c}$ by \Cref{proposition:M_al conditions}\ref{ob_x4}, a contradiction. If $\Frs{r_1}{s_1}$ satisfies \X{\al'}{7}, then $s_2\ge2\mb{c}-\om_s+p$, again a contradiction. Hence $\Frs{r_1}{s_1}$ satisfies \X{\al'}{6}. This gives $s_1>r_1=2\mb{c}-\om_s-p+\al'-1\ge2\mb{c}-\om_s-p+\al-1$, a contradiction.
					Therefore, $\Frs{r_1}{s_1}\in\mcl{M}_{\al''}$ for some $0\le\al''<\al$. By \Cref{def:prec}\ref{def 2}, $\Frs{r_1}{s_1}\prec F_s$.
					
					\item Let $s_2=n-p$. From \Cref{proposition:M_al conditions}\ref{ob_x6} and \Cref{proposition:M_al conditions}\ref{ob_x7},  $F_s$ satisfies \X{\al}{8}. By \Cref{proposition:M_al conditions}\ref{ob_x8}, $n-p=s_2\le\mb{c}+\frac{3p}{2}-1$. Since $\mb{c}\le\frac{n+1}{2}$, $n\le 5p-1$. Therefore, $n\ge 6p-3$ implies $p\le2$. Hence  $p=2$ (as $p\ge2$), and thus $\al=1$.
					
					By \Cref{proposition: al_i precedes lambda}, $r_1<s_1<\om_s$ implies $\om_s\os r_1$. Hence $(\Frs{r_1}{s_2})^c=\{\om_s\}\sqcup\{r_1,s_1\}$. 
					Since $0\le r_1<s_1<\om_s<s_2=n-p$ and $s_1\nsim\om_s$, we get $r_1\nsim\om_s$. Therefore $\Frs{r_1}{s_2}\in\md$. This means that $\Frs{r_1}{s_2}\in\mcl{M}_{\al'}$ for some $\al'\in[0,p-1]=\{0,1\}$. 
					If $\al'=0$, then $\Frs{r_1}{s_2}\prec F_s$ by \Cref{def:prec}\ref{def 2}. Now, assume $\al'=1$. Then $\al'=\al$. Since $r_1<s_1$, $\Frs{r_1}{s_2}\ll F_s$ by \Cref{definition:order_ll}\ref{om_i=om_j}. 
					Hence $\Frs{r_1}{s_2}\prec F_s$ by \Cref{def:prec}\ref{def 1}. 
				\end{enumerate}  
				
				\item Let $s_1\le r_1<r_2\le s_2$. From \Cref{claim:case678condition}, $s_1<\om_s<s_2$ and $\om_s>\mb{c}$. Since $\om_s\os s_1$, \Cref{remark:order in S}\ref{part:i} implies $s_1<2\mb{c}-\om_s$. Moreover, $2\mb{c}-\om_s<\mb{c}<\om_s$. Thus $s_1<2\mb{c}-\om_s<\om_s<s_2$. We consider the following subcases based on the value of $\om_r$:
				\begin{enumerate}[label=(\alph*)]  
					\item $\om_r<s_1$. 
					
					Suppose $F_s$ satisfies \X{\al}{6}. Then $\om_s\ge\mb{c}+\frac{\al}{2}$, $s_1=2\mb{c}-\om_s-p+\al-1$ and $s_2\le2\mb{c}-\om_s+p-1$. We have $\om_r<s_1<2\mb{c}-\om_s<\mb{c}$. By \Cref{remark:order in S}\ref{part:i}, $\om_r\os r_1$ and $\om_r<s_1\le r_1$ imply $r_1\ge2\mb{c}-\om_r$. Moreover, $\om_s+p-\al+1\le2\mb{c}-s_1<2\mb{c}-\om_r$.
					Therefore $\om_s+p-\al+1<r_1<r_2\le s_2\le2\mb{c}-\om_s+p-1$, which gives $\om_s<\mb{c}+\frac{\al-2}{2}$, a contradiction. Hence $F_s$ satisfies \X{\al}{7} or \X{\al}{8}. This means that $s_1=s_2-2p+\al-1$. 
					From \Cref{claim:case678condition}, $s_1<\om_s<s_2$, $s_2\le n-p$ and $s_1\nsim s_2$. Observe that $\om_r\in F_s\setminus F_r$. By \Cref{proposition:s_1<=r_1<r_2<=s_2}\ref{al_r<s_1} and \ref{F_al_s precedes F_s}, we obtain $\Frs{\om_r}{\om_s}\in\md$ and $\Frs{\om_r}{\om_s}\prec F_s$, respectively.
					
					\item $s_1\le\om_r<2\mb{c}-\om_s$. 
					
					We have $\om_s>\mb{c}$, $s_1\ge s_2-2p+\al-1$ and $s_2\le\om_s+p-\al$ by \Cref{claim:case678condition}. From \Cref{remark:order in S}\ref{part:ii}, $\om_s>\mb{c}$ implies $\om_s\os\om_r$. By \Cref{definition:order_ll}\ref{om_i<om_j}, $F_s\ll F_r$. 
					
					\begin{claim}\label{claim:6_7_8 II}
						$r_1\ge\om_r+p+1$.
					\end{claim}
					
					\begin{proof}[Proof of \Cref{claim:6_7_8 II}]    
						We first show that $r_1>\om_r-p$. If $F_s$ satisfies \X{\al}{6} or \X{\al}{7}, then $s_1\ge2\mb{c}-\om_s-p+\al-1$ by \Cref{remark:x4_x5 and x6_x7}\ref{x6_x7}, which implies $r_1\ge s_1\ge2\mb{c}-\om_s-p+\al-1\ge\om_r-p+\al$. Now, suppose $F_s$ satisfies \X{\al}{8}. Then $\om_s\ge\mb{c}+\frac{p}{2}$ and $s_1=s_2-2p+\al-1$. It follows that  $\om_r<2\mb{c}-\om_s\le\mb{c}-\frac{p}{2}\le\om_s-p$. Since $\om_s<s_2$, we get $r_1\ge s_1=s_2-2p+\al-1\ge\om_s-2p+\al>\om_r-p+\al$. Thus, $r_1\ge \om_r-p+\al>\om_r-p$.
						
						Suppose $r_1<\om_r$. Then $\om_r-p<r_1<\om_r$. Since $s_1\le r_1<r_2\le s_2$, it follows that $r_1\ge s_1\ge s_2-2p+\al-1\ge r_2-2p+\al-1$ and $r_2\le s_2\le\om_s+p-\al<2\mb{c}-\om_r+p-\al$.
						We have $F_s\ll F_r$. By \Cref{corollary:notin M_0}\ref{satisfies 3_4_5}, $F_s\prec F_r$, a contradiction. Hence $r_1>\om_r$.
						
						Since $\om_r\os r_1$ and $\om_r<2\mb{c}-\om_s<\mb{c}$, \Cref{remark:order in S}\ref{part:i} implies $r_1\ge2\mb{c}-\om_r$.
						Thus $\om_s<2\mb{c}-\om_r\le r_1<r_2\le s_2\le\om_s+p-\al$. It follows that $r_1\sim r_2$. Since $C_n^p[F_r^c]$ is disconnected and $r_1>\om_r$, we get $r_1\ge\om_r+p+1$. 
					\end{proof}   
					
					We have $\om_r<\om_r+p+1\le r_1<r_2$, $r_2\le s_2\le s_1+2p-\al+1\le\om_r+2p-\al+1$ and $F_s\ll F_r$. If $\om_r\le n-2p-2$, then $F_s\prec F_r$ by \Cref{proposition:in M_bt+gm}\ref{satisfies x9}, a contradiction. So, assume $\om_r>n-2p-2$.
					By \Cref{proposition: c and p relation}\ref{c-p/2<=n-2p-1}, $\mb{c}-\frac{p}{2}\le n-2p-1$. Hence $2\mb{c}-\om_s>\om_r\ge n-2p-1\ge\mb{c}-\frac{p}{2}$, which implies $\om_s<\mb{c}+\frac{p}{2}$. This means that $F_s$ satisfies \X{\al}{6} or \X{\al}{7}. 
					If $F_s$ satisfies \X{\al}{6}, then $\om_s\ge\mb{c}+\frac{\al}{2}$; and if $F_s$ satisfies \X{\al}{7}, then $\om_s\ge\mb{c}+\frac{\al}{2}$ by \Cref{proposition:M_al conditions}\ref{ob_x7}. Thus $\om_s\ge\mb{c}+\frac{\al}{2}$.
					By \Cref{proposition: c and p relation}\ref{c-(bt+1)/2<=n-2p-1}, $\om_r<2\mb{c}-\om_s\le n-2p-1$, a contradiction. Hence $\om_r\not>n-2p-2$.
					Therefore, this case is not possible.
					
					\item $2\mb{c}-\om_s\le\om_r<\om_s$.
					
					We have $s_1\le r_1<r_2\le s_2$. Moreover, $s_1<\om_s<s_2$, $\om_s>\mb{c}$ and $s_1\nsim s_2$ by \Cref{claim:case678condition}. 
					Therefore, $\om_r\in F_s\setminus F_r$, $(\Frs{\om_r}{\om_s})^c=\{\om_r\}\sqcup\{s_1,s_2\}$ and $s_1<\om_r<s_2$ by \Cref{proposition:al_r s_1 s_2}\ref{al_r s_1 s_2 complement}. Since $s_1\nsim s_2$, $\Frs{\om_r}{\om_s}\in\md$ by \Cref{proposition:al_r s_1 s_2}\ref{al_r s_1 s_2 is a facet}. We now show that $\Frs{\om_r}{\om_s}\prec F_s$. 
					
					If $F_s$ satisfies \X{\al}{7} or \X{\al}{8}, then $s_1=s_2-2p+\al-1$, which implies $\Frs{\om_r}{\om_s}\prec F_s$ by \Cref{proposition:al_r s_1 s_2}\ref{al_r s_1 s_2 precedes F}. 
					
					Now, let $F_s$ satisfies \X{\al}{6}. Then $s_1=2\mb{c}-\om_s-p+\al-1$ and $s_2\le2\mb{c}-\om_s+p-1$. From \Cref{proposition:M_al conditions}\ref{ob_x6}, $\om_s<\mb{c}+\frac{p}{2}$.
					Suppose $\Frs{\om_r}{\om_s}\in\mcl{M}_{\al'}$ for some $\al<\al'\le p-1$. Since $\mb{c}-\frac{p}{2}<2\mb{c}-\om_s\le\om_r<\om_s<\mb{c}+\frac{p}{2}$ and $s_1<\om_r<s_2$, it follows that $\Frs{\om_r}{\om_s}$ satisfies one of the conditions from \X{\al'}{4} to \X{\al'}{7}.
					
					Suppose $\Frs{\om_r}{\om_s}$ satisfies \X{\al'}{4}. Then $s_1\ge\om_r-p+\al'$. Since $\om_r\ge2\mb{c}-\om_s$ and $\al'>\al$, it follows that $s_1>2\mb{c}-\om_s-p+\al$, a contradiction. Hence $\Frs{\om_r}{\om_s}$ does not satisfy \X{\al'}{4}.
					
					Suppose $\Frs{\om_r}{\om_s}$ satisfies \X{\al'}{5}. Then $\om_r\le\mb{c}-\frac{\al'+1}{2}<\mb{c}$ and $s_1\ge2\mb{c}-\om_r-p$. 
					Since $\om_r<\mb{c}$, $2\mb{c}-\om_r>\om_r$. This implies $r_1\ge s_1\ge2\mb{c}-\om_r-p>\om_r-p$.
					We have $r_2\le s_2\le2\mb{c}-\om_s+p-1<\om_r+p$.   Thus $\om_r-p<r_1<r_2<\om_r+p$. This contradicts $\cn{F_r}$ is disconnected. Hence $\Frs{\om_r}{\om_s}$ does not satisfy \X{\al'}{5}.
					
					Suppose $\Frs{\om_r}{\om_s}$ satisfies \X{\al'}{6} or \X{\al'}{7}. Then $s_1\ge2\mb{c}-\om_r-p+\al'-1$ by \Cref{remark:x4_x5 and x6_x7}\ref{x6_x7}. Since $\om_r<\om_s$ and $\al'>\al$, we get $s_1>2\mb{c}-\om_s-p+\al-1$, a contradiction. Hence $\Frs{\om_r}{\om_s}$ does not satisfy \X{\al'}{6} and \X{\al'}{7}.
					
					Therefore, $\Frs{\om_r}{\om_s}\in\mcl{M}_{\al''}$ for some $0\le\al''\le\al$. Since $\om_s>\mb{c}$, $\om_r\os\om_s$ by \Cref{remark:order in S}\ref{part:iv}. By \Cref{definition:order_ll}\ref{om_i<om_j}, $\Frs{\om_r}{\om_s}\ll F_s$. Hence $\Frs{\om_r}{\om_s}\prec F_s$ by \Cref{def:prec}.
					
					\item $\om_s\le\om_r\le s_2$.
					
					From \Cref{claim:case678condition}, $\om_s>\mb{c}$, $\om_s>s_1\ge s_2-2p+\al-1, s_2\sim\om_s$ and $s_1\nsim s_2$. If $\om_r>\om_s$, then $\om_s>\mb{c}$ implies $\om_s\os\om_r$, and hence $F_s\ll F_r$ by \Cref{definition:order_ll}\ref{om_i<om_j}. Moreover, if $\om_r=\om_s$, then $\{r_1,r_2\}\cap\{s_1,s_2\}=\emptyset$, which implies $s_1<r_1$, and thus $F_s\ll F_r$ by \Cref{definition:order_ll}\ref{om_i=om_j}. In either case, $F_s\ll F_r$.
					
					Suppose $r_2>\om_r$. Then $s_1<\om_s\le\om_r<r_2\le s_2$. Since $s_2\sim\om_s$ and $s_1\nsim s_2$, we get $s_2\le\om_s+p$, and hence $r_2\le\om_r+p$. By \Cref{proposition:i1_i2 nsim al_i}\ref{sim i2}, $r_1<\om_r$. Thus $r_1<\om_r<r_2\le\om_r+p$. We have $r_1\ge s_1\ge s_2-2p+\al-1\ge r_2-2p+\al-1$ and $F_s\ll F_r$. Observe that if $\om_r\ge\mb{c}+\frac{p}{2}$, then $F_s\prec F_r$ by \Cref{proposition:in M_bt+gm}\ref{satisfies x8}, a contradiction. So $\om_r<\mb{c}+\frac{p}{2}$.
					Since $\om_s\le\om_r$, $F_s$ satisfies \X{\al}{6} or \X{\al}{7}. By \Cref{remark:x4_x5 and x6_x7}\ref{x6_x7}, $s_1\ge2\mb{c}-\om_s-p+\al-1$. Using $s_1\le r_1$ and $\om_s\le\om_r$, we obtain $r_1\ge2\mb{c}-\om_r-p+\al-1$. Therefore, $F_s\prec F_r$ by \Cref{corollary:notin M_0}\ref{satisfies 6_7_8}, again a contradiction. Hence $r_2<\om_r$. This gives $r_1<r_2<\om_r$.
					
					Note that $r_1\ge s_1\ge s_2-2p+\al-1\ge\om_r-2p+\al-1$. If $\om_r\ge\mb{c}+\frac{p}{2}$, then $F_s\prec F_r$ by \Cref{proposition:in M_bt+gm}\ref{satisfies x2}, a contradiction. So $\om_r<\mb{c}+\frac{p}{2}$.
					Then $F_s$ satisfies \X{\al}{6} or \X{\al}{7}. We get $r_1\ge2\mb{c}-\om_r-p+\al-1$. This implies $F_s\prec F_r$ by \Cref{proposition:in M_bt+gm}\ref{satisfies x1}, again a contradiction. Hence, this case is not possible.
					
					\item $\om_r>s_2$.
					
					From \Cref{claim:case678condition}, $s_1<\om_s<s_2$, $\om_s>\mb{c}$, $s_1>p$ and $s_1\nsim s_2$. Observe that $\om_r\in F_s\setminus F_r$. By \Cref{proposition:s_1<=r_1<r_2<=s_2}\ref{al_r>s_2}, $(\Frs{\om_r}{\om_s})^c=\{s_2\}\sqcup\{s_1,\om_r\}$ and $\Frs{\om_r}{\om_s}\in\md$. 
					
					If $F_s$ satisfies \X{\al}{7} or \X{\al}{8}, then $s_1=s_2-2p+\al-1$, which implies $\Frs{\om_r}{\om_s}\prec F_s$ by \Cref{proposition:s_1<=r_1<r_2<=s_2}\ref{F_al_s precedes F_s}. 
					
					Now, let $F_s$ satisfies \X{\al}{6}.
					Then $\om_s\ge\mb{c}+\frac{\al}{2}$ and $s_1=2\mb{c}-\om_s-p+\al-1$. Suppose $\Frs{\om_r}{\om_s}\in\mcl{M}_{\al'}$ for some $\al\le\al'\le p-1$. Since $s_1<s_2<\om_r$, $\Frs{\om_r}{\om_s}$ satisfies one of the conditions from \X{\al'}{3} to \X{\al'}{8}.
					This implies $s_1\ge\om_r-2p+\al'-1$. Using $s_1\le r_1$ and $\al\le\al'$, we get $r_1\ge\om_r-2p+\al-1$.
					Moreover, $\om_r>\om_s$ implies $r_1\ge s_1=2\mb{c}-\om_s-p+\al-1>2\mb{c}-\om_r-p+\al-1$.
					Since $\om_r>s_2>\om_s>\mb{c}$, $\om_s\os\om_r$. By \Cref{definition:order_ll}\ref{om_i<om_j}, $F_s\ll F_r$. We have $r_1<r_2<\om_r$ (as $r_2\le s_2<\om_r$).
					Observe that if $\om_r<\mb{c}+\frac{p}{2}$, then $F_s\prec F_r$ by \Cref{proposition:in M_bt+gm}\ref{satisfies x1}\ref{satisfies x1 ii}, and if $\om_r\ge\mb{c}+\frac{p}{2}$, then $F_s\prec F_r$ by \Cref{proposition:in M_bt+gm}\ref{satisfies x2}. Both cases contradict $F_r\prec F_s$.
					Therefore, $\Frs{\om_r}{\om_s}\in\mcl{M}_{\al''}$ for some $0\le\al''<\al$. By \Cref{def:prec}\ref{def 2}, $\Frs{\om_r}{\om_s}\prec F_s$.
				\end{enumerate} 
				
				\item Let $s_1\le r_1$ and $s_2<r_2$. We have $s_1<\om_s<s_2$, $\om_s>\mb{c}$, $s_1\ge s_2-2p$, $s_1 \nsim s_2$ and $s_1\nsim\om_s$ by \Cref{claim:case678condition}. Note that $r_2\in F_s\setminus F_r$.
				By \Cref{proposition: F_r1_s1 or F_r2_s2}\ref{F_r2_s2 is a facet}, $(\Frs{r_2}{s_2})^c=\{\om_s\}\sqcup\{s_1,r_2\}$ and $\Frs{r_2}{s_2}\in\md$.
				
				If $F_s$ satisfies \X{\al}{7} or \X{\al}{8}, then $s_1=s_2-2p+\al-1$, which implies that $\Frs{r_2}{s_2}\prec F_s$ by \Cref{proposition: F_r1_s1 or F_r2_s2}\ref{F_s1_s2 precedes F_s}\ref{prob:f_r1_s1_3_ii}. 
				
				Now, let $F_s$ satisfies \X{\al}{6}. Then $\om_s\ge\mb{c}+\frac{\al}{2}$, $s_1=2\mb{c}-\om_s-p+\al-1$ and by \Cref{proposition:M_al conditions}\ref{ob_x6}, $\om_s<\mb{c}+\frac{p}{2}$.
				Since $\Frs{r_2}{s_2}\in\md$, $\Frs{r_2}{s_2}\in\mcl{M}_{\al'}$ for some $\al'\in[0,p-1]$. 
				If $0\le\al'<\al$, then $\Frs{r_2}{s_2}\prec F_s$ by \Cref{def:prec}\ref{def 2}. 
				
				So, assume that $\al\le\al'\le p-1$.
				
				\begin{claim}\label{claim:case678_last_condiction}
					$\Frs{\om_r}{\om_s} \in\md$ and $\Frs{\om_r}{\om_s}\prec F_s$.
				\end{claim}
				
				\begin{proof}[Proof of \Cref{claim:case678_last_condiction}]       
					We have $\mb{c}+\frac{\al}{2}\le\om_s<\mb{c}+\frac{p}{2}$.   
					Recall that if $\om_r=\om_s$, then $\{r_1,r_2\}\cap\{s_1,s_2\}=\emptyset$, which implies $s_1<r_1$. Moreover, $s_1<\om_s<r_2$, $s_1\le r_1<r_2$ and $r_2\in F_s$. By \Cref{proposition:x1 and x6} for $F=F_r$ and $F'=F_s$, we get $r_2\le2\mb{c}-\om_s+p$ and $2\mb{c}-\om_s\le\om_r<\om_s$. Since $\om_s>\mb{c}$, \Cref{proposition:al_r s_1 s_2}\ref{al_r s_1 s_2 complement} yields $\om_r\in F_s\setminus F_r$, $(\Frs{\om_r}{\om_s})^c=\{\om_r\}\sqcup\{s_1,s_2\}$ and $s_1<\om_r<s_2$.
					
					Note that $\om_r<s_2<r_2\le2\mb{c}-\om_s+p\le\om_r+p$. By \Cref{proposition:i1_i2 nsim al_i}\ref{sim i2}, $r_1\le\om_r-p-1$. This gives $s_1\le r_1\le\om_r-p-1<\om_r<s_2$. Since $s_1\nsim s_2$, we get $s_1\nsim\om_r$. Hence $\Frs{\om_r}{\om_s}\in\md$.
					
					Suppose $\Frs{\om_r}{\om_s}\in\mcl{M}_{\bt}$ for some $\al<\bt\le p-1$. Then $s_1<\om_r<s_2$ and $\om_r<\om_s<\mb{c}+\frac{p}{2}$ implies that $\Frs{\om_r}{\om_s}$ satisfies one of the conditions from \X{\bt}{3} to \X{\bt}{7}. If $\Frs{\om_r}{\om_s}$ satisfies \X{\bt}{3} or \X{\bt}{4} or \X{\bt}{5}, then using \Cref{proposition:M_al conditions}\ref{ob_x3}, \ref{ob_x4} and \ref{ob_x5} we get $s_1\sim\om_r$, a contradiction. 
					Hence $\Frs{\om_r}{\om_s}$ satisfies \X{\bt}{6} or \X{\bt}{7}. By \Cref{remark:x4_x5 and x6_x7}\ref{x6_x7}, $s_1\ge2\mb{c}-\om_r-p+\bt-1$. Using $\om_r<\om_s$ and $\al<\bt$, we get $s_1>2\mb{c}-\om_s-p+\al-1$, a contradiction as $s_1=2\mb{c}-\om_s-p+\al-1$. 
					Therefore, $\Frs{\om_r}{\om_s}\in\mcl{M}_{\bt'}$ for some $0\le\bt'\le\al$. Since $\om_s>\mb{c}$ and and $2\mb{c}-\om_s\le\om_r<\om_s$, \Cref{remark:order in S}\ref{part:iv} implies $\om_r\os\om_s$.
					By \Cref{definition:order_ll}\ref{om_i<om_j}, $\Frs{\om_r}{\om_s}\ll F_s$. Hence $\Frs{\om_r}{\om_s}\prec F_s$ by \Cref{def:prec}. 
				\end{proof}
				
				This completes the case. 
			\end{enumerate}   
			
			\item $F_s$ satisfies \X{\al}{9} or \X{\al}{10}.
			
			\begin{claim}\label{claim:9_10}
				$\om_s<s_1<s_2$, $\om_s\ge2p$, $s_1\ge\om_s+p+1$, $s_1\sim s_2$, $s_1\nsim\om_s$ and $s_2\nsim\om_s$.
			\end{claim}
			
			\begin{proof}[Proof of \Cref{claim:9_10}]
				Since $F_s$ satisfies \X{\al}{9} or \X{\al}{10}, $\om_s<s_1<s_2$ and $s_1\ge\om_s+p+1$. By \Cref{proposition:M_al conditions}\ref{ob_x9} and \ref{ob_x10}, $\om_s\ge2p$ and $s_1\sim s_2$. Therefore, $F_s\in\md$ implies $s_1\nsim\om_s$ and $s_2\nsim\om_s$. 
			\end{proof}
			
			We consider three cases: (I) $r_1<\om_s$, (II) $\om_s\le r_1<r_2\le s_2$, (III) $\om_s\le r_1$ and $s_2<r_2$.
			\begin{enumerate}[label=(\Roman*)]  
				\item Let $r_1<\om_s$. Then $r_1<\om_s<s_1<s_2$ implies $r_1\in F_s\setminus F_r$. By \Cref{proposition:9 or 10 u<om_i}\ref{u<om_i}, $\Frs{r_1}{s_1}\in\md$ and $\Frs{r_1}{s_1}\prec F_s$.
				
				\item Let $\om_s\le r_1<r_2\le s_2$. We consider the following subcases based on the value of $\om_r$:
				\begin{enumerate}[label=(\alph*)]
					\item $\om_r<\om_s$.
					
					Since $\om_r<\om_s<s_1<s_2$, we have $\om_r\in F_s\setminus F_r$. By \Cref{proposition:9 or 10 u<om_i}\ref{u<om_i}, $\Frs{\om_r}{s_1}\in\md$ and $\Frs{\om_r}{s_1}\prec F_s$.
					
					\item $\om_r=\om_s$.
					
					In this case, $\{r_1,r_2\}\cap\{s_1,s_2\}=\emptyset$ by our assumption. Hence $r_1\notin\{s_1,s_2\}$, and $r_2\le s_2$ implies $r_2<s_2$. Moreover, since $\om_s=\om_r\os r_1$ and $\om_s\le r_1$, we have $\om_s<r_1$. Therefore, $\om_s<r_1<r_2<s_2$ and $r_1\in F_s\setminus F_r$. Observe that $(\Frs{r_1}{s_1})^c=\{\om_s\}\sqcup\{r_1,s_2\}$. By \Cref{claim:9_10}, $s_2\nsim\om_s$.
					If $\om_s\sim r_1$ and $r_1\sim s_2$, then $s_2\nsim\om_s$ implies $r_1\sim r_2$. This contradicts $F_r\in\md$ (as $\om_r=\om_s$). Hence $\om_s\nsim r_1$ or $r_1\nsim s_2$. Using $s_2\nsim\om_s$, we get $\Frs{r_1}{s_1}\in\md$.
					
					Suppose $r_1>s_1$. Since $s_1\ge\om_s+p+1$ by \Cref{claim:9_10}, $r_1>\om_s+p+1=\om_r+p+1$.  We have $r_1<r_2<\om_r$. 
					If $F_s$ satisfies \X{\al}{9}, then $\om_r=\om_s\le n-2p-2$ and $r_2<s_2=\om_s+2p-\al+1=\om_r+2p-\al+1$. By \Cref{proposition:in M_bt+gm}\ref{satisfies x9}\ref{satisfies x9 ii}, $F_s\prec F_r$, a contradiction.
					If $F_s$ satisfies \X{\al}{10}, then $\om_r=\om_s>n-2p-2$ and $r_2<s_2=n-\al-1$. By \Cref{proposition:in M_bt+gm}\ref{satisfies x10}\ref{satisfies x10 ii}, $F_s\prec F_r$, again a contradiction. Hence $r_1<s_1$. By \Cref{proposition:9 or 10 u<om_i}\ref{om_i<u<i_1}, $\Frs{r_1}{s_1}\prec F_s$.
					
					\item $\om_r>\om_s$.
					
					We deal with three subcases: 
					
					\begin{enumerate}[label=(\roman*)]
						\item Let $\om_s<\mb{c}$ and $\om_r<2\mb{c}-\om_s$. Then $\om_r\os\om_s$ by \Cref{remark:order in S}\ref{part:iii}. Since $\om_s\os s_1,s_2$, we get $\om_r\os s_1,s_2$. Hence $\om_r\in F_s\setminus F_r$. Observe that $(\Frs{\om_r}{s_1})^c=\{\om_r\}\sqcup\{\om_s,s_2\}$. 
						By \Cref{remark:order in S}\ref{part:i}, $\om_s\os s_2$ and $s_2>\om_s$ imply $s_2\ge2\mb{c}-\om_s$. Thus $\om_s<\om_r<s_2$. We have $\om_s\le r_1<r_2\le s_2$. Since $s_2\nsim\om_s$ by \Cref{claim:9_10},  \Cref{proposition:al_i j_1 j_2 is a facet} yields $\Frs{\om_r}{s_1}\in\md$. 
						If $F_s$ satisfies \X{\al}{10}, then \Cref{proposition:M_al conditions}\ref{ob_x10} contradicts $\om_s<\mb{c}$. Hence $F_s$ satisfies \X{\al}{9}, which implies $s_2=\om_s+2p-\al+1$. Suppose $\Frs{\om_r}{s_1}\in\mcl{M}_{\al'}$ for some $\al<\al'\le p-1$. Since $\om_s<\om_r<s_2$, $\Frs{\om_r}{s_1}$ satisfies one of the conditions from \X{\al'}{3} to \X{\al'}{8}.
						It follows that $s_2\le\om_s+2p-\al'+1<\om_s+2p-\al+1$, a contradiction. Therefore, $\Frs{\om_r}{s_1}\in\mcl{M}_{\al''}$ for some $0\le\al''\le\al$. By \Cref{proposition:M_al conditions}\ref{ob_x10}, $\om_r\os\om_s$ implies $\Frs{\om_r}{s_1}\ll F_s$. Therefore $\Frs{\om_r}{s_1}\prec F_s$ by \Cref{def:prec}.
						
						\item Let $\om_s<\mb{c}$ and $\om_r\ge2\mb{c}-\om_s$. Then $2\mb{c}-\om_s>\mb{c}$. Clearly, $\om_r>\mb{c}$ and $\om_s\ge2\mb{c}-\om_r$.
						We have $2\mb{c}-\om_r\le\om_s\le r_1<r_2\le s_2$. Since $\om_r\os r_1$, $\om_r>\mb{c}$ and $r_1\ge2\mb{c}-\om_r$,  \Cref{remark:order in S}\ref{part:ii} yields $r_1>\om_r$.
						By \Cref{claim:9_10}, $s_1\sim s_2$ and $s_2\nsim\om_s$. Therefore, $\om_s<s_1<s_2$ implies $s_2\le s_1+p$. If $\om_r\ge s_1$, then $\om_s<s_1\le\om_r<r_1<r_2\le s_2$, which contradicts $\cn{F_r}$ is disconnected. Hence $\om_r<s_1$. This means that $\om_s<\om_r<s_1<s_2$, and thus $\om_r\in F_s\setminus F_r$.
						By \Cref{remark:order in S}\ref{part:iv}, $\om_r>\mb{c}$ and $2\mb{c}-\om_r\le\om_s<\om_r$ imply $\om_s\os\om_r$. Therefore, $(\Frs{\om_r}{s_1})^c=\{\om_s\}\sqcup\{\om_r,s_2\}$. 
						We have $s_2\nsim\om_s$, $\om_s<\om_r<s_2$ and  $\om_s\le r_1<r_2\le s_2$. By \Cref{proposition:al_i j_1 j_2 is a facet}, $\Frs{\om_r}{s_1}\in\md$. Hence $\Frs{\om_r}{s_1}\prec F_s$ by \Cref{proposition:9 or 10 u<om_i}\ref{om_i<u<i_1}.
						
						\item Let $\om_s\ge\mb{c}$. Then $\om_r>\om_s\ge\mb{c}$ implies $\om_r>\mb{c}$. Moreover, $\om_s\ge\mb{c}>2\mb{c}-\om_r$, which gives $\om_s\ge 2\mb{c}-\om_r$. Using the inequalities $\om_r>\mb{c}$ and $\om_s\ge 2\mb{c}-\om_r$, the remainder of the argument proceeds exactly as in case (ii). This gives $\om_r\in F_s\setminus F_r$, $\Frs{\om_r}{s_1}\in\md$ and $\Frs{\om_r}{s_1}\prec F_s$.
					\end{enumerate}
				\end{enumerate}
				
				\item Let $\om_s\le r_1$ and $s_2<r_2$. Then $\om_s<s_1<s_2<r_2$ implies $r_2\in F_s\setminus F_r$. Since $\om_s<s_2$, $\om_s\os r_2$ by \Cref{proposition: al_i precedes lambda}. Thus $(\Frs{r_2}{s_1})^c=\{\om_s\}\sqcup\{s_2,r_2\}$. Using $p<2p\le\om_s<s_2<r_2$ and $s_2\nsim\om_s$, we get $r_2\nsim\om_s$. Hence $\Frs{r_2}{s_1}\in\md$. 
				
				Suppose $\Frs{r_2}{s_1}\in\mcl{M}_{\al'}$ for some $\al\le\al'\le p-1$. We have $\om_s<s_2<r_2$. If $\om_s\le n-2p-2$, then $\Frs{r_2}{s_1}$ satisfies \X{\al'}{9} and $F_s$ satisfies \X{\al}{9}. This implies $r_2=\om_s+2p-\al'+1\le\om_s+2p-\al+1=s_2$, a contradiction. If $\om_s>n-2p-2$, then $\Frs{r_2}{s_1}$ satisfies \X{\al'}{10} and $F_s$ satisfies \X{\al}{10}. This implies $r_2=n-\al'-1\le n-\al-1=s_2$, again a contradiction. Therefore $\Frs{r_2}{s_1}\in\mcl{M}_{\al''}$ for some $0\le\al''<\al$. Hence $\Frs{r_2}{s_1}\prec F_s$ by \Cref{def:prec}\ref{def 2}.
			\end{enumerate} 
		\end{enumerate}	
		\vspace{-0.2 cm}
	\end{proof}
	
	Using \Cref{lemma: case M_0 a,lemma: case M_0 b,lemma: case M_al}, we conclude that the order $\prec$ provides a shelling order for the facets of $\Delta_3(C_n^p)$. 
	
	\subsection{Spanning Facets}\label{subsection:spanning facets}
	
	Having proved that $\prec$ provides a shelling order for the facets of $\Delta_3(C_n^p)$, we now characterize and count the spanning facets. We continue to use the notation and results established prior to \Cref{subsection:shelling order}. 
	
	We first define several subsets of $V(C_n^p)$. 
	Let $\mcl{U}_1=\{p+1, p+2, \dots, 2p-1\}$, $\mcl{U}_2=\{2p, 2p+1, \dots, \mb{c}-p\}$ and $\mcl{U}_3=\{\mb{c}-p+1, \mb{c}-p+2, \dots, n-p-1\}$. Clearly, $\mcl{U}_1$, $\mcl{U}_2$ and $\mcl{U}_3$ are pairwise disjoint.
	For each $u\in\mcl{U}_1\sqcup\mcl{U}_2\sqcup\mcl{U}_3$, define $\mcl{V}^{u}$ as follows:
	\begin{enumerate}
		\item If $u\in\mcl{U}_1$, then $\mcl{V}^{u} =\{u-2p \ (\text{mod} \ n), u-2p+1 \ (\text{mod} \ n), \dots, n-1\} \sqcup \{0, 1, \dots, 2\mb{c}-u-1\}$.
		
		\item If $u\in\mcl{U}_2$, then $\mcl{V}^{u}=\{u-2p, u-2p+1, \dots, u-2\} \sqcup \{u, u+1, \dots, 2\mb{c}-u-1\} \sqcup \{n-1\}$.
		
		\item If $u\in\mcl{U}_3$, then $\mcl{V}^{u}=\{\mu, \mu+1, \dots, u, u+1, \dots, u+p\} \sqcup \{u+t\ (\text{mod $n$})|\ p+1\le t\le 2p-1\} \sqcup \{\nu\}$, 
		where $\mu=\min\{2\mb{c}-u, u-2p\}$ and $\nu=\begin{cases} 
			n-1 & \text{if } u<n-2p-1, \\ 
			u+2p \ (\text{mod} \ n) & \text{if } u \ge n-2p-1. 
		\end{cases}$
	\end{enumerate}
	
	Recall that the set of facets  $M(\Delta_3(C_p^n))\subset\bigsqcup_{i=1}^{n-2}\mcl{A}_i$, where the sets $\mcl{A}_i$ are defined in \Cref{equation:Ai}. Let $\Sigma_1, \Sigma_2$ and  $\Sigma_3$ be subsets of $M(\Delta_3(C_p^n))$ defined as  follows: for $m\in[3]$, a facet $F \in \Sg_m$ if and only if $F\in\mcl{A}_i$ for some $i\in[n-2]$ such that $F^c=\{\om_i\} \sqcup \{i_1, n-1\}$, $\omega_i \in \mcl{U}_m,$ and $i_1 \in V(C_p^n) \setminus \mcl{V}^{\om_i}$. 
	
	Since $\mcl{U}_1$, $\mcl{U}_2$ and $\mcl{U}_3$ are pairwise disjoint, it follows that $\Sg_1$, $\Sg_2$ and $\Sg_3$ are pairwise disjoint. Let 
	$$\Sigma :=\Sigma_1 \sqcup \Sigma_2 \sqcup \Sigma_3.$$
	
	Our aim is to show that a facet $F$ is a spanning facet if and only if $F \in \Sg$. For this, we begin with some important results. 
	
	\begin{proposition}\label{proposition:structure in M_al}
		Let $F\in\md$. If  $F\in\mcl{M}_{\al}$ for $\al\in[p-1]$, then $F^c=\{u,v,w\}$ for some $u,v,w\in V(C_n^p)$ with $u<v<w\le u+2p$.
	\end{proposition}
	
	\begin{proof}
		Let $F\in\mcl{A}_i$ and $F^c=\{\om_i\}\sqcup\{i_1,i_2\}$. Since $F\in\mcl{M}_{\al}$, $F$ satisfies one of the conditions from \X{\al}{1} to \X{\al}{10}. 
		First, suppose $F$ satisfies \X{\al}{1} or \X{\al}{2}. Then $i_1<i_2<\om_i$. If $F$ satisfies \X{\al}{1}, then $\om_i<\mb{c}+\frac{p}{2}$ and $i_1=2\mb{c}-\om_i-p+\al-1$. Using $\al\ge 1$, we get $i_1\ge 2\mb{c}-\om_i-p>\om_i-2p$.    
		If $F$ satisfies \X{\al}{2}, then $i_1=\om_i-2p+\al-1\ge \om_i-2p$. It follows that $i_1<i_2<\om_i\le i_1+2p$.
		Now, suppose $F$ satisfies any of the conditions from  \X{\al}{3} to \X{\al}{8}. Then $i_1<\om_i<i_2\le i_1+2p-\al+1\le i_1+2p$.
		Finally, suppose $F$ satisfies \X{\al}{9} or \X{\al}{10}. Then $\om_i<i_1<i_2$. If $F$ satisfies \X{\al}{9}, then $\om_i=i_2-2p+\al-1\ge i_2-2p$.
		If $F$ satisfies \X{\al}{10}, then $\om_i>n-2p-2$ and $i_2=n-\al-1$, which implies $i_2\le n-2<\om_i+2p$. We get $\om_i<i_1<i_2\le \om_i+2p$.
		
		Therefore $F^c=\{u,v,w\}$ for some $u,v,w\in V(C_n^p)$ with $u<v<w\le u+2p$. 
	\end{proof}
	
	This proposition implies the following corollary.
	
	\begin{corollary}\label{corollary:structure in M_al}
		Let $F\in\md$. If there exist $u,v\in F^c$ with $u<v$ and $v-u> 2p$, then $F\in\mcl{M}_0$.
	\end{corollary}
	
	For $F \in\md$, recall that  $\Fuv{u}{v}=(F\setminus\{u\})\sqcup\{v\}$ for some $u\in F$ and $v\in F^c$.
	
	\begin{proposition}\label{proposition:order ll}
		Let $F \in \md$ such that  $F \in \mcl{A}_i$ and $F^c=\{\om_i\}\sqcup\{i_1, n-1\}$. For each $u \in F$, it follows that $\Fuv{u}{n-1}\ll F$.
	\end{proposition}
	
	\begin{proof}
		We have $(\Fuv{u}{n-1})^c=(F^c\setminus\{n-1\})\sqcup\{u\}=\{\om_i, i_1, u\}$ and $u\ne\om_i$. 
		Hence, either $\om_i\os u$ or $u\os\om_i$. 
		
		First, let $\om_i\os u$. By \Cref{remark:Aj}, $\om_i\os i_1$. This implies $\Fuv{u}{n-1}\in\mcl{A}_i$. If $u<i_1$, then $(\Fuv{u}{n-1})^c=\{\om_i\}\sqcup\{u,i_1\}$, and $\Fuv{u}{n-1}\ll F$  by \Cref{definition:order_ll}\ref{om_i=om_j}. 
		If $u>i_1$, then $(\Fuv{u}{n-1})^c=\{\om_i\}\sqcup\{i_1,u\}$, and thus $u<n-1$ implies that $\Fuv{u}{n-1}\ll F$ by \Cref{definition:order_ll}\ref{om_i=om_j}.
		
		Now, let $u\os\om_i$. Then $u=\om_j$ for some $j<i$ by the definition of $\os$, and $\Fuv{u}{n-1}\in\mcl{A}_j$. Thus, \Cref{definition:order_ll}\ref{om_i<om_j} implies that $\Fuv{u}{n-1}\ll F$. 
	\end{proof}
	
	By the definition of spanning facets, a facet $F$ is spanning if and only if for each $u\in F$, there exists $v\in F^c$ such that
	\begin{equation}  \label{eqn:hash}
		\Fuv{u}{v}\in\md  \text{ and }  \Fuv{u}{v} \prec F. \hfill \tag{{$\#$}}
	\end{equation}
	
	The following results are used to prove that if a facet $F\in\Sigma$, then $F$ is spanning.
	
	\begin{proposition}\label{proposition: al_i-2p and al_i+2p}
		Let $F\in\md$ such that $F\in\mcl{A}_i$  and $F^c=\{\om_i\}\sqcup\{i_1, n-1\}$. Suppose the following conditions hold: $(a)$ $\om_i>p$, $(b)$ $i_1\in V(C_n^p)\setminus\{\om_i+t\ (\text{mod $n$})| -2p\le t\le 2p-1\}$, and $(c)$ either $i_1\ne\om_i+2p\ (\text{mod $n$})$, or if $i_1=\om_i+2p\ (\text{mod $n$})$, then $\om_i+2p\ (\text{mod $n$})=\om_i+2p$.
		Then $F$ is a spanning facet.
	\end{proposition}
	
	\begin{proof}
		Note that $i_1\notin\{\om_i+t\ (\text{mod $n$})| -2p\le t\le 2p-1\}$ and $i_1<n-1$. To prove that $F$ is a spanning facet, for each $u\in F$, we need to find $v\in F^c$ such that $\Fuv{u}{v}$ satisfies \eqref{eqn:hash}. 
		Let $u\in F$. We have $n-1\in F^c$ and $(\Fuv{u}{n-1})^c=\{\om_i, i_1, u\}$. By \Cref{proposition:order ll}, $\Fuv{u}{n-1}\ll F$. 
		
		First, suppose that $i_1\ne\om_i+2p\ (\text{mod $n$})$. Then $i_1\notin\{\om_i-2p \ (\text{mod $n$}), \om_i-2p+1\ (\text{mod $n$}), \ldots, \om_i+2p\ (\text{mod $n$})\}$. 
		This implies $i_1\nsim\om_i$, and $u\nsim\om_i$ or $u\nsim i_1$. Hence $\Fuv{u}{n-1}\in \md$. Since $\om_i\os i_1$ by \Cref{remark:Aj}, $i_1\ne\om_i$. Observe that if $i_1<\om_i$, then $i_1<\om_i-2p$; and if $i_1>\om_i$, then $i_1>\om_i+2p$. 
		In either case, we have $F,\Fuv{u}{n-1}\in\mcl{M}_0$ by \Cref{corollary:structure in M_al}. By \Cref{def:prec}\ref{def 1}, $\Fuv{u}{n-1}\ll F$ implies $\Fuv{u}{n-1}\prec F$. Hence \eqref{eqn:hash} is satisfied by $\Fuv{u}{n-1}$. Therefore $F$ is a spanning facet.
		
		Now suppose that $i_1=\om_i+2p\ (\text{mod $n$})$. Then $\om_i+2p\ (\text{mod $n$})=\om_i+2p$. Since $\om_i>p$, it follows that $p<\om_i<\om_i+2p=i_1<n-1$. Hence $\om_i\nsim i_1$ and $\om_i<i_1$. Clearly, $u\ne \om_i, i_1$. We consider three cases: (i) $u<\om_i<i_1$, (ii) $\om_i<u<i_1$, (iii) $\om_i<i_1<u$.
		\begin{enumerate}[label=(\roman*)]
			\item Let $u<\om_i<i_1$. Since $i_1\notin\{\om_i-2p\ (\text{mod $n$}),\om_i-2p+1\ (\text{mod $n$}),\ldots,\om_i-1\ (\text{mod $n$})\}$, we have $u\nsim\om_i$ or $u\nsim i_1$.  Hence $\om_i\nsim i_1$ implies $\Fuv{u}{n-1}\in \md$. 
			From \Cref{corollary:structure in M_al}, $i_1=\om_i+2p>u+2p$ implies $\Fuv{u}{n-1}\in\mcl{M}_0$. Since $\Fuv{u}{n-1}\ll F$, $\Fuv{u}{n-1}\prec F$ by \Cref{def:prec}. Therefore $\Fuv{u}{n-1}$ satisfies \eqref{eqn:hash}.
			
			\item Let $\om_i<u<i_1$. We have $i_1\in F^c$ and $(\Fuv{u}{i_1})^c=\{\om_i, u, n-1\}$. Using $p<\om_i<u<i_1=\om_i+2p<n-1$, we get $\om_i\nsim n-1$, and $u\nsim n-1$ or $u\nsim\om_i$. 
			Therefore $\Fuv{u}{i_1}\in \md$. Since $\om_i+2p<n-1$, $\Fuv{u}{i_1}\in\mcl{M}_0$ by \Cref{corollary:structure in M_al}. 
			If $\om_i\os u$, then $u<i_1$ implies  $\Fuv{u}{i_1}\ll F$ by \Cref{definition:order_ll}\ref{om_i=om_j}; and if $u\os\om_i$, then $\Fuv{u}{i_1}\ll F$ by \Cref{definition:order_ll}\ref{om_i<om_j}. By \Cref{def:prec}, $\Fuv{u}{i_1}\prec F$ (as $\Fuv{u}{i_1}\in\mcl{M}_0$). Hence \eqref{eqn:hash} is satisfied by $\Fuv{u}{i_1}$.
			
			\item Let $\om_i<i_1<u$. Then $p<\om_i<\om_i+2p=i_1<u<n-1$, which implies $u\nsim\om_i$. Hence $\om_i\nsim i_1$ implies $\Fuv{u}{n-1}\in \md$. Using \Cref{corollary:structure in M_al}, $\Fuv{u}{n-1}\in\mcl{M}_0$. 
			Since $\Fuv{u}{n-1}\ll F$, \Cref{def:prec} implies $\Fuv{u}{n-1}\prec F$. Hence $\Fuv{u}{n-1}$ satisfies \eqref{eqn:hash}.
		\end{enumerate}
		In each case, there exists $v\in F^c$ such that $\Fuv{u}{v}$ satisfies \eqref{eqn:hash}. Therefore $F$ is a spanning facet.
	\end{proof}
	
	\begin{corollary}\label{corollary:spannig y1y2}
		Let $F\in \md$ such that $F\in\mcl{A}_i$  and $F^c=\{\om_i\}\sqcup\{i_1, n-1\}$. Suppose $\om_i\in \{p+1, p+2, \ldots,\mb{c}-p\}$ and $i_1\in V(C_n^p)\setminus\bigl(\{\om_i-2p \ (\text{mod $n$}), \om_i-2p+1\ (\text{mod $n$}), \ldots, \om_i-p-2\ (\text{mod $n$})\}\sqcup\{\om_i-p-1,\om_i-p,\ldots, 2\mb{c}-\om_i-1\}\bigr)$. Then $F$ is a spanning facet.
	\end{corollary}
	
	\begin{proof}  
		Since $\om_i\le\mb{c}-p$, we have $\om_i+2p-1\le2\mb{c}-\om_i-1$. Using $p\ge 2$, $2\mb{c}\le n+1$ and $\om_i\ge p+1$, it follows that $\om_i-p-1\ge 0$ and $2\mb{c}-\om_i-1\le 2\mb{c}-p-2\le n-p-1\le n-3$. 
		Therefore $\{\om_i-p-1\ (\text{mod $n$}),\om_i-p\ (\text{mod $n$}),\ldots,\om_i+2p-1\ (\text{mod $n$}),\ldots,2\mb{c}-\om_i-1\ (\text{mod $n$})\}=\{\om_i-p-1,\om_i-p,\ldots,\om_i+2p-1,\ldots,2\mb{c}-\om_i-1\}$. 
		Observe that $i_1\in V(C_n^p)\setminus\{\om_i+t\ (\text{mod $n$})| -2p\le t\le 2p-1\}$. By \Cref{proposition: c and p relation}\ref{range of c}, $\mb{c}\le n-3p+2$. Since $p+1\le\om_i\le\mb{c}-p$, we get $0<3p+1\le\om_i+2p\le\mb{c}+p\le n-2p+2<n$. Hence $\om_i+2p\ (\text{mod $n$})=\om_i+2p$. Moreover, $\om_i>p$. By \Cref{proposition: al_i-2p and al_i+2p}, $F$ is a spanning facet.
	\end{proof}
	
	\begin{lemma}
		Let $F\in\md$ be such that $F\in\Sg$. Then $F$ is a spanning facet. \label{lemma: if F in Sigma}
	\end{lemma}
	
	\begin{proof}
		Since $F\in\Sg=\Sg_1\sqcup\Sg_2\sqcup\Sg_3$, there exists $i\in[n-2]$ such that $F\in\mcl{A}_i$ and $F^c=\{\om_i\}\sqcup\{i_1, n-1\}$. 
		Let $\mcl{W}:=V(C_n^p)\setminus\bigl(\{\om_i-2p \ (\text{mod $n$}), \om_i-2p+1\ (\text{mod $n$}), \ldots, \om_i-p-2\ (\text{mod $n$})\}\sqcup\{\om_i-p-1,\om_i-p,\ldots, 2\mb{c}-\om_i-1\}\bigr)$.
		We consider three cases: (1) $F\in\Sg_1$, (2) $F\in\Sg_2$, (3) $F\in\Sg_3$.
		\begin{enumerate}[label=(\arabic*)]    
			\item Let $F\in\Sg_1$. Then $\om_i\in\mcl{U}_1=\{p+1,p+2,\ldots, 2p-1\}$ and $i_1\in V(C_p^n) \setminus \mcl{V}^{\om_i}=V(C_n^p)\setminus\bigr(\{\om_i-2p\ (\text{mod $n$}), \om_i-2p+1\ (\text{mod $n$}),\ldots,n-1\}\sqcup\{0,1,\ldots,2\mb{c}-\om_i-1\}\bigl)$. 
			By \Cref{proposition: c and p relation}\ref{range of c}, $\mb{c}\ge 3p-1$, which implies $2p-1\le \mb{c}-p$. Hence $\om_i\in\{p+1,p+2,\ldots,\mb{c}-p\}$. Observe that $i_1\in\mcl{W}$. Therefore, $F$ is a spanning facet by \Cref{corollary:spannig y1y2}. 
			
			\item Let $F\in\Sg_2$. Then $\om_i\in\mcl{U}_2=\{2p,2p+1,\ldots,\mb{c}-p\}$ and $i_1\in V(C_p^n) \setminus \mcl{V}^{\om_i}=V(C_n^p)\setminus(\{\om_i-2p, \om_i-2p+1,\ldots, \om_i-2\}\sqcup\{\om_i, \om_i+1, \ldots, 2\mb{c}-\om_i-1\}\sqcup\{n-1\})$.
			
			Since $p\ge 2$, we have $p+1<2p$. Hence $\om_i\in\{p+1,p+2,\ldots,\mb{c}-p\}$. Moreover, $2p\le \om_i\le\mb{c}-p$ and $\mb{c}\le n+1$ imply $0\le \om_i-2p\le \om_i-p-2\le\mb{c}-2p-2\le n-2p-1<n-1$. Thus $\{\om_i-2p\ (\text{mod $n$}),\om_i-2p+1\ (\text{mod $n$}),\ldots,\om_i-p-2\ (\text{mod $n$})\}=\{\om_i-2p,\om_i-2p+1,\ldots,\om_i-p-2\}$. This gives $\mcl{W}=V(C_n^p)\setminus\{\om_i-2p, \om_i-2p+1, \ldots, 2\mb{c}-\om_i-1\}$. 
			If $i_1\ne\om_i-1$, then $i_1\in\mcl{W}$, and by \Cref{corollary:spannig y1y2}, $F$ is a spanning facet.
			
			Now, assume $i_1=\om_i-1$. We prove that $F$ is a spanning facet. For this, it suffices to show that for each $u\in F$, there exists $v\in F^c$ such that $\Fuv{u}{v}$ satisfies \eqref{eqn:hash}. Let $u\in F$. We consider two subcases: (i) $u<2\mb{c}-\om_i$ and (ii) $u\ge 2\mb{c}-\om_i$.
			\begin{enumerate}[label=(\roman*)]
				\item $u<2\mb{c}-\om_i$
				
				We have $i_1\in F^c$ and $(\Fuv{u}{i_1})^c=\{\om_i, u, n-1\}$. By \Cref{proposition: c and p relation}\ref{range of c}, $\mb{c}\le n-3p+2$.  Since $\om_i\le \mb{c}-p$ and $p\ge 2$, we get $\om_i+2p\le\mb{c}+p\le n-2p+2\le n-2<n-1$. 
				Therefore, $n-1+2p\ (\text{mod $n$})=2p-1<\om_i<n-1-2p$. This implies $\om_i\nsim n-1$, and $u\nsim\om_i$ or $u\nsim n-1$. Hence $\Fuv{u}{i_1}\in\md$. 
				
				Since $\om_i<\mb{c}$, we have $\om_i<2\mb{c}-\om_i$. Clearly, $u\ne\om_i$. This means that either $u<\om_i<2\mb{c}-\om_i$ or $\om_i<u<2\mb{c}-\om_i$. 
				First, let $u<\om_i<2\mb{c}-\om_i$. Then $u<\om_i<\mb{c}$ implies $\om_i\os u$, and hence $(\Fuv{u}{i_1})^c=\{\om_i\}\sqcup\{u,n-1\}$.
				Using $u\ne i_1$ and $i_1=\om_i-1$, we get $u<i_1$. Hence $\Fuv{u}{i_1}\ll F$  by \Cref{definition:order_ll}\ref{om_i=om_j}. 
				Now, let $\om_i<u<2\mb{c}-\om_i$.  Since $\om_i<\mb{c}$, \Cref{remark:order in S}\ref{part:iii} implies $u\os\om_i$. By \Cref{definition:order_ll}\ref{om_i<om_j}, $\Fuv{u}{i_1}\ll F$. In both cases, $\Fuv{u}{i_1}\ll F$.
				By \Cref{corollary:structure in M_al}, $\Fuv{u}{i_1}\in\mcl{M}_0$ (as $\om_i+2p<n-1$). Therefore, $\Fuv{u}{i_1}\prec F$ by \Cref{def:prec}. It follows that $\Fuv{u}{i_1}$ satisfies \eqref{eqn:hash}.
				
				\item $u\ge2\mb{c}-\om_i$. 
				
				We have $n-1\in F^c$ and $(\Fuv{u}{n-1})^c=\{\om_i, i_1, u\}$. Since $2p\le\om_i\le \mb{c}-p$, we get $2p-1\le\om_i-1=i_1<\om_i<\om_i+2p\le 2\mb{c}-\om_i\le u<n-1$. This implies $u\nsim\om_i$ and $u\nsim i_1$. Hence $\Fuv{u}{n-1}\in\md$.
				
				Since $\om_i<\mb{c}$, $\om_i\os u$ by \Cref{remark:order in S}\ref{part:i}. Using $i_1<u$, it follows that $(\Fuv{u}{n-1})^c=\{\om_i\}\sqcup\{i_1,u\}$. By \Cref{definition:order_ll}\ref{om_i=om_j}, $\Fuv{u}{n-1}\ll F$ (as $u<n-1$).  
				Moreover, by \Cref{corollary:structure in M_al}, $u>i_1+2p$ implies $\Fuv{u}{n-1}\in\mcl{M}_0$. 
				Therefore $\Fuv{u}{n-1}\prec F$ by \Cref{def:prec}. Hence \eqref{eqn:hash} is satisfied by $\Fuv{u}{n-1}$.
			\end{enumerate}
			In either case, there exists $v\in F^c$ such that $\Fuv{u}{v}$ satisfies \eqref{eqn:hash}. Thus $F$ is a spanning facet.
			
			\item Let $F\in\Sg_3$.
			We have $\om_i\in\mcl{U}_3=\{\mb{c}-p+1,\mb{c}-p+2,\ldots, n-p-1\}$ and $i_1\in V(C_p^n) \setminus \mcl{V}^{\om_i}= V(C_n^p)\setminus(\{\mu, \mu+1,\ldots,\om_i,\om_i+1,\ldots,\om_i+p\}\sqcup\{\om_i+t\ (\text{mod $n$})|\ p+1\le t\le 2p-1\}\sqcup\{\nu\})$, where $\mu=\min\{2\mb{c}-\om_i, \om_i-2p\}$ and $\nu=\begin{cases}
				n-1 &\text{ if $\om_i<n-2p-1$},\\
				\om_i+2p\ (\text{mod $n$}) &\text{ if $\om_i\ge n-2p-1$}.
			\end{cases}$
			
			Since $\mu\le\om_i-2p$, we have $\{\om_i+t| -2p\le t\le p\}\subseteq\{\mu,\mu+1,\ldots,\om_i,\om_i+1,\ldots,\om_i+p\}$.  
			By \Cref{proposition: c and p relation}\ref{range of c}, $\mb{c}\ge 3p-1$. Using $\mb{c}-p+1\le\om_i\le n-p-1$, we obtain $0\le\mb{c}-3p+1\le\om_i-2p<\om_i+p\le  n-1$.
			This gives $\{\om_i+t\ (\text{mod $n$})| -2p\le t\le p\}=\{\om_i+t| -2p\le t\le p\}$. Hence $\{\om_i+t\ (\text{mod $n$})| -2p\le t\le 2p-1\}=\{\om_i+t| -2p\le t\le p\}\sqcup\{\om_i+t\ (\text{mod $n$})|\ p+1\le t\le 2p-1\}\subseteq\{\mu,\mu+1,\ldots,\om_i,\om_i+1,\ldots,\om_i+p\}\sqcup\{\om_i+t\ (\text{mod $n$})|\ p+1\le t\le 2p-1\}\sqcup\{\nu\}$.     
			Therefore, $i_1\in V(C_n^p)\setminus\{\om_i+t\ (\text{mod $n$})| -2p\le t\le 2p-1\}$.
			
			Suppose $i_1=\om_i+2p\ (\text{mod $n$})$. Since $\nu\ne i_1$, we have $\nu=n-1$, and hence $\om_i<n-2p-1$. This means that $\om_i+2p\ \text{(mod $n$)}=\om_i+2p$. 
			Moreover, $\mb{c}\ge 3p-1$ and $p\ge 2$ implies $\om_i\ge\mb{c}-p+1\ge 2p>p$.
			
			Hence $F$ is a spanning facet by \Cref{proposition: al_i-2p and al_i+2p}.
		\end{enumerate}
		Therefore, we conclude that if $F\in\Sg$, then $F$ is a spanning facet.
	\end{proof}
	
	We now proceed to prove that if $F\in \md$ is a spanning facet, then $F\in\Sg$. Before presenting the proof, we first establish some results.
	
	\begin{proposition}\label{proposition:al=u+p}
		Suppose $F\in \md$ such that $F^c=\{u,v,w\}$, where $u,v,w\in V(C_n^p)$ and $v,w\in\{u+1 \ (\text{mod $n$}),u+2 \ (\text{mod $n$}),\ldots, u+2p \ (\text{mod $n$})\}$. Then $F$ is not a spanning facet.
	\end{proposition}
	
	\begin{proof}
		Observe that $u,v,w\sim u+p$ (mod $n$). If $u+p$ (mod $n$) $\in F^c$, then $\cn{F}$ is connected, which contradicts $F\in \md$. Hence $u+p$ (mod $n$) $\in F$. For any $x\in F^c$, $C_n^p[(\Fuv{u+p\text{ (mod $n$)}}{x})^c]$ is connected, and therefore $\Fuv{u+p\text{ (mod $n$)}}{x}\notin\md$. 
		Thus there does not exist $x\in F^c$ such that $\Fuv{u+p\text{ (mod $n$)}}{x}$ satisfies \eqref{eqn:hash}.  Therefore,  $F$ is not a spanning facet.
	\end{proof}
	
	\begin{corollary}\label{corollary:not in M_al}
		Let $F\in\md$ such that $F\in\mcl{A}_i$ and $F^c=\{\om_i\}\sqcup\{i_1,i_2\}$. Suppose $F\in\mcl{M}_{\al}$ for some $\al\in[p-1]$. Then $F$ is not a spanning facet.  
	\end{corollary}
	
	\begin{proof}
		By \Cref{proposition:structure in M_al}, $F^c=\{u,v,w\}$ for some $u,v,w\in V(C_n^p)$ with $u<v<w\le u+2p$. Hence $F$ is not a spanning facet by \Cref{proposition:al=u+p}.
	\end{proof}
	
	\begin{proposition}\label{proposition:al_i and i_2} 
		Let $F\in \md$ be a spanning facet such that $F\in\mcl{A}_i$ and $F^c=\{\om_i\}\sqcup\{i_1,i_2\}$. Then $\om_i\in\{p+1,p+2,\ldots,n-p-1\}$ and $i_2=n-1$.
	\end{proposition}
	
	\begin{proof}
		By \Cref{remark:Aj}, $\om_i\os i_1,i_2$. Hence $\om_i\neq i_1,i_2$.
		
		\begin{itemize}
			\item Suppose $\om_i\le p$. Since $p<\mb{c}$ by \Cref{proposition: c and p relation}\ref{range of c}, we have $\om_i<\mb{c}$. Moreover, $2\mb{c}-\om_i\ge 2\mb{c}-p\ge n-p$ (as $2\mb{c}\ge n$). 
			First, suppose $i_1<\om_i$. Then $\cn{F}$ is disconnected and $\om_i\le p$ imply $i_2>\om_i$. Using $\om_i\os i_2$, \Cref{remark:order in S}\ref{part:i} yields $i_2\ge 2\mb{c}-\om_i$. 
			This means that $n-p\le i_2\le n-1$. Since $0\le i_1<\om_i\le p$, we get $i_1,\om_i\in\{i_2+1 \ (\text{mod $n$}),i_2+2 \ (\text{mod $n$}),\ldots, i_2+2p \ (\text{mod $n$})\}$. From \Cref{proposition:al=u+p}, $F$ is not a spanning facet, a contradiction. Hence $i_1>\om_i$. By \Cref{remark:order in S}\ref{part:i}, $\om_i\os i_1$ implies $i_1\ge 2\mb{c}-\om_i$, which gives $n-p\le i_1<i_2\le n-1$. Using $0\le \om_i\le p$, we obtain $i_2,\om_i\in\{i_1+1 \ (\text{mod $n$}),i_1+2 \ (\text{mod $n$}),\ldots, i_1+2p \ (\text{mod $n$})\}$. By \Cref{proposition:al=u+p}, $F$ is not a spanning facet, again a contradiction. Therefore $\om_i> p$.
			
			\item Suppose $\om_i\ge n-p$. By \Cref{proposition: c and p relation}\ref{range of c}, $n-p>\mb{c}$, which implies $\om_i>\mb{c}$. Note that $2\mb{c}-\om_i\le 2\mb{c}-n+p\le p+1$ (as $2\mb{c}\le n+1$). 
			
			First, suppose $i_2>\om_i$. Then $n-p\le\om_i<i_2\le n-1$. Since $\cn{F}$ is disconnected, $i_1<\om_i$. By \Cref{remark:order in S}\ref{part:i}, $\om_i\os i_1$ implies $i_1<2\mb{c}-\om_i$. Hence $0\le i_1\le p$.
			
			Now, suppose $i_2<\om_i$. Then $\om_i\os i_2$ implies $i_2<2\mb{c}-\om_i$ by \Cref{remark:order in S}\ref{part:i}. Thus $0\le i_1<i_2\le p$. Moreover, we have $n-p\le \om_i\le n-1$.
			
			In either case, $i_1,i_2\in\{\om_i+1 \ (\text{mod $n$}),\om_i+2 \ (\text{mod $n$}),\ldots, \om_i+2p \ (\text{mod $n$})\}$. It follows that $F$ is not a spanning facet by \Cref{proposition:al=u+p}, a contradiction. 
			Hence $\om_i<n-p$.
		\end{itemize} 
		We conclude that $\om_i\in\{p+1,p+2,\ldots,n-p-1\}$.
		
		By the definition of $\Omega$, we have $\om_i\os n-1$. 
		Suppose $n-1\in F$. Since $F$ is a spanning facet, there exists $v\in F^c$ such that $\Fuv{n-1}{v}$ satisfies \eqref{eqn:hash}. This means that $\Fuv{n-1}{v}\in\md$ and $\Fuv{n-1}{v}\prec F$. Since $F$ is a spanning facet, $F \in\mcl{M}_0$ by  \Cref{corollary:not in M_al}.
		Therefore, if $F \ll \Fuv{n-1}{v}$, then $F \prec \Fuv{n-1}{v}$ by \Cref{def:prec}, a contradiction. Hence $\Fuv{n-1}{v}\ll F$.
		
		If $v=\om_i$, then $(\Fuv{n-1}{v})^c=(\Fuv{n-1}{\om_i})^c=\{i_1,i_2,n-1\}$. Since $\om_i\os i_1,i_2,n-1$, \Cref{definition:order_ll}\ref{om_i<om_j} yields $F \ll \Fuv{n-1}{\om_i}$, a contradiction. Thus $v\in\{i_1,i_2\}$. Using $\om_i\os i_1,i_2,n-1$ and $i_1<i_2<n-1$, we have either $(\Fuv{n-1}{v})^c=(\Fuv{n-1}{i_1})^c=\{\om_i\}\sqcup\{i_2,n-1\}$ or $(\Fuv{n-1}{v})^c=(\Fuv{n-1}{i_2})^c=\{\om_i\}\sqcup\{i_1,n-1\}$. 
		By \Cref{definition:order_ll}\ref{om_i=om_j}, $i_1<i_2$ implies $F \ll \Fuv{n-1}{i_1}$, and $i_2<n-1$ implies $F \ll \Fuv{n-1}{i_2}$, again a contradiction. Hence $n-1\in F ^c$. 
		
		Clearly, $\om_i\os n-1$ and $i_1<i_2$ imply $i_2=n-1$.
	\end{proof}
	
	\begin{proposition}\label{proposition:F is not a spanning facet}
		Let $F \in \md$ such that  $F \in \mcl{A}_i$ and $F^c=\{\om_i\} \sqcup \{i_1, i_2\}$. Suppose there exists $u\in F$ with $\om_i\os u$ and $i_1<u$. If either $\Fuv{u}{i_2}\notin \md$, or  $\Fuv{u}{i_2}\in \md$ such that $\Fuv{u}{i_2}\in\mcl{M}_{\al}$ for some $\al\in[p-1]$, then $F$ is not a spanning facet.
	\end{proposition}
	
	\begin{proof}
		Suppose $F$ is a spanning facet. Since $u\in F$, there exists $v\in F^c$ such that $\Fuv{u}{v}$ satisfies \eqref{eqn:hash}. This means that $\Fuv{u}{v}\in\md$ and $\Fuv{u}{v}\prec F$. By \Cref{proposition:al_i and i_2}, $i_2=n-1$, and hence $F^c=\{\om_i\} \sqcup \{i_1, n-1\}$. By \Cref{remark:Aj}, $\om_i\os i_1,n-1$. Moreover, $F\in\mcl{M}_0$ by  \Cref{corollary:not in M_al}. Since $\Fuv{u}{v}\prec F$, \Cref{def:prec} implies $\Fuv{u}{v}\in\mcl{M}_0$ and $\Fuv{u}{v}\ll F$. 
		
		First, suppose $v=\om_i$. Then $(\Fuv{u}{v})^c=(\Fuv{u}{\om_i})^c=\{i_1,n-1,u\}$. By \Cref{definition:order_ll}\ref{om_i<om_j}, $\om_i\os i_1,n-1,u$ implies $F \ll \Fuv{u}{\om_i}$, a contradiction. Now, suppose $v=i_1$. Then $\om_i\os u,n-1$ and $u<n-1$ imply $(\Fuv{u}{v})^c=(\Fuv{u}{i_1})^c=\{\om_i\}\sqcup\{u,n-1\}$. 
		Since $i_1<u$, we get $F\ll \Fuv{u}{i_1}$ by \Cref{definition:order_ll}\ref{om_i=om_j}, again a contradiction. 
		Hence $v=i_2=n-1$. We get $\Fuv{u}{i_2}\in \md$. By assumption, $\Fuv{u}{i_2}\in\mcl{M}_{\al}$ for some $\al\in[p-1]$,  a contradiction to $\Fuv{u}{v}\in\mcl{M}_0$. 
		
		Hence, there does not exist any $v\in F^c$ such that $\Fuv{u}{v}$ satisfies \eqref{eqn:hash}. Therefore $F$ is not a spanning facet. 
	\end{proof}
	
	\begin{proposition}\label{proposition:F spanning facet then i_1 not in}
		Let $F \in \md$ such that  $F \in \mcl{A}_i$ and  $F^c=\{\om_i\} \sqcup \{i_1, i_2\}$. Suppose $F$ is a spanning facet. We have the following.
		\begin{enumerate}[label=(\roman*)]
			\item \label{om_i<=c} If $\om_i\le \mb{c}$, then $i_1\notin\{\om_i-2p,\om_i-2p+1,\ldots,\om_i-2\}$.
			
			\item \label{i_1 not in om_i+1 to om_i+2p-1} $i_1\notin\{\om_i+t\ |\ 1\le t\le p\}\sqcup\{\om_i+t\ (\text{mod $n$})\ |\ p+1\le t\le 2p-1\}$.
			
			\item \label{i_1 not=nu} Let $\nu:=\begin{cases}
				n-1 &\text{ if $\om_i<n-2p-1$},\\
				\om_i+2p\ (\text{mod $n$}) &\text{ if $\om_i\ge n-2p-1$}.
			\end{cases}$ Then $i_1\ne\nu$.
		\end{enumerate}    
	\end{proposition}
	
	\begin{proof}
		By \Cref{remark:Aj}, $\om_i\os i_1$. Since $F$ is a spanning facet, $i_2=n-1$ by \Cref{proposition:al_i and i_2}.
		\begin{enumerate}[label=(\roman*)]
			\item We have $\om_i\le \mb{c}$. By \Cref{proposition: c and p relation}\ref{range of c}, $\mb{c}<n-p$. Thus, $p\ge 2$ implies $\om_i\le\mb{c}<n-1$.  
			
			First, suppose $i_1\in\{\om_i-2p,\om_i-2p+1,\ldots,\om_i-p-1\}$. Then $i_1<i_1+p\le\om_i-1<\om_i<n-1=i_2$ and $i_1\ge\om_i-2p$. Since $i+p\ne \om_i,i_1, i_2$, we get $i_1+p\in F$.  Note that $(\Fuv{i_1+p}{i_2})^c=\{\om_i,i_1,i_1+p\}$. We have $i_1<i_1+p<\om_i\le i_1+2p$, which implies $i_1\sim (i_1+p)$ and $(i_1+p)\sim\om_i$. Thus $\Fuv{i_1+p}{i_2}\notin \md$. Moreover, $i_1+p<\om_i\leq\mb{c}$ implies $\om_i\os i_1+p$.
			From \Cref{proposition:F is not a spanning facet} (for $u=i_1+p$), $F$ is not a spanning facet, a contradiction. Hence $i_1\notin\{\om_i-2p,\om_i-2p+1,\ldots,\om_i-p-1\}$.
			
			Now, suppose $i_1\in\{\om_i-p,\om_i-p+1,\ldots,\om_i-2\}$. Then $i_1<i_1+1\le\om_i-1<\om_i<n-1=i_2$, which implies $i_1+1\in F$.  We have $(\Fuv{i_1+1}{i_2})^c=\{\om_i,i_1,i_1+1\}$. Since $\om_i-p\le i_1<i_1+1<\om_i$, it follows that $\Fuv{i_1+1}{i_2}\notin \md$.
			Moreover, $i_1+1<\om_i\leq\mb{c}$ implies $\om_i\os i_1+1$. Therefore $F$ is not a spanning facet by \Cref{proposition:F is not a spanning facet} (for $u=i_1+1$), again a contradiction. Hence $i_1\notin\{\om_i-p,\om_i-p+1,\ldots,\om_i-2\}$.
			
			Therefore $i_1\notin\{\om_i-2p,\om_i-2p+1,\ldots,\om_i-2\}$.
			
			\item Suppose $i_1\in\{\om_i+t\ |\ 1\le t\le p\}\sqcup\{\om_i+t\ (\text{mod $n$})\ |\ p+1\le t\le 2p-1\}$.    
			
			Since $F$ is a spanning facet, we have $p+1\le\om_i\le n-p-1$ by \Cref{proposition:al_i and i_2}.
			It follows that $\{\om_i+t\ (\text{mod $n$})\ |\ 1\le t\le p\}=\{\om_i+t\ |\ 1\le t\le p\}$. Hence $i_1\in\{\om_i+t\ (\text{mod $n$})\ |\ 1\le t\le 2p-1\}$.              
			Observe that if $\om_i+2p\ge n-1$, then $i_1<i_2=n-1$ implies $i_2\in\{\om_i+t\ (\text{mod $n$})\ |\ 2\le t\le 2p\}$. By \Cref{proposition:al=u+p}, $F$ is not a spanning facet, a contradiction. So, $\om_i+2p<n-1$.
			
			This implies $\{\om_i+t\ (\text{mod $n$})\ |\ p+1\le t\le 2p-1\}=\{\om_i+t\ |\ p+1\le t\le 2p-1\}$, and therefore $i_1\in\{\om_i+t\ |\ 1\le t\le 2p-1\}$. Thus $\om_i<i_1<i_1+1\le\om_i+2p<n-1$, which implies $i_1+1\in F$.
			
			If $\om_i<\mb{c}$, then $\om_i\os i_1$ and $i_1>\om_i$ imply $i_1\ge 2\mb{c}-\om_i$ by \Cref{remark:order in S}\ref{part:i}, and hence $\om_i\os i_1+1$. If $\om_i\ge\mb{c}$, then $i_1+1>\om_i$ implies $\om_i\os i_1+1$. 
			Since $F$ is a spanning facet, \Cref{proposition:F is not a spanning facet} implies $\Fuv{i_1+1}{i_2}\in\md$ and $\Fuv{i_1+1}{i_2}\in\mcl{M}_0$.
			
			We have $(\Fuv{i_1+1}{i_2})^c=\{\om_i\}\sqcup\{i_1,i_1+1\}$ with $\om_i<i_1<i_1+1$. Since $\cn{(\Fuv{i_1+1}{i_2})}$ is disconnected, $i_1\ge \om_i+p+1$. Hence $i_1\le \om_i+2p-1$ implies $i_1+1=\om_i+2p-\gm+1$ for some $ \gm\in[p-1]$.
			Since $\om_i\le n-2p-2$, \Cref{proposition:in M_bt+gm}\ref{satisfies x9}\ref{satisfies x9 i} yields $\Fuv{i_1+1}{i_2}\notin\mcl{M}_0$, a contradiction.
			
			Therefore $i_1\notin\{\om_i+1,\om_i+2,\ldots,\om_i+p\}\sqcup\{\om_i+p+1\ (\text{mod $n$}),\om_i+p+2\ (\text{mod $n$}),\ldots,\om_i+2p-1\ (\text{mod $n$})\}$. 
			
			\item Suppose $i_1=\nu$. Since $i_1<i_2=n-1$, we have $i_1\ne n-1$. Note that if $\om_i\le n-2p-1$, then $\nu=n-1$. Hence $\om_i\ge n-2p$. This implies $i_1=\nu=\om_i+2p\ (\text{mod $n$})$.
			Since $\om_i<n-1=i_2$ and $\om_i+2p\ge n>n-1=i_2$, it follows that $i_2\in\{\om_i+1\ (\text{mod $n$}),\om_i+2\ (\text{mod $n$}),\ldots,\om_i+2p-1\ (\text{mod $n$})\}$. By \Cref{proposition:al=u+p}, $F$ is not a spanning facet, a contradiction. Thus $i_1\ne\nu$. 
		\end{enumerate}
	\end{proof}
	
	\begin{lemma}\label{lemma: then F in Sigma}
		Suppose $F\in \md$ is a spanning facet. Then $F\in\Sg$. 
	\end{lemma}
	
	\begin{proof}
		We have $F\in \md$. Suppose $F\in\mcl{A}_j$ such that $F^c=\{\om_j\}\sqcup\{j_1,j_2\}$. Since $F$ is a spanning facet, \Cref{proposition:al_i and i_2} implies $\om_j\in\{p+1,p+2,\ldots,n-p-1\}$ and $j_2=n-1$. 
		
		In this proof, for each $\omega_j$, we first identify all possible values of $j_1$ for which $F$ is a spanning facet. Then, we show that for all such values of $j_1$ corresponding to a given $\omega_j$, $F\in\Sg_m$ for some $m\in[3]$, thereby implying that $F \in \Sigma$.
		Before proceeding further, recall that $\om_j\os j_1$ by \Cref{remark:Aj}. 
		
		Let $$\nu:=\begin{cases}
			n-1 &\text{ if $\om_j<n-2p-1$},\\
			\om_j+2p\ (\text{mod $n$}) &\text{ if $\om_j\ge n-2p-1$}.
		\end{cases}$$ 
		Based on the values of $\om_j$, we consider the following four cases.
		\begin{enumerate}[label=(\roman*)]
			\item $\om_j\in\{p+1,p+2,\ldots,2p-1\}$.
			
			By \Cref{proposition: c and p relation}\ref{range of c}, $\mb{c}\ge3p-1$. We have $\om_j\le 2p-1<3p-1\le\mb{c}$ and $\om_j\os j_1$. By \Cref{remark:order in S}\ref{part:i}, either $j_1<\om_j$ or $j_1\ge 2\mb{c}-\om_j$. 
			Suppose $j_1<\om_j$. Then $0\le j_1<\om_j<n-1=j_2$. Since $\cn{F}$ is disconnected, $j_2\nsim j_1$ or $j_1\nsim \om_j$. It follows that $\om_j>j_2+2p\ (\text{mod $n$})=2p-1\ge\om_j$, a contradiction. 
			Hence $j_1\ge 2\mb{c}-\om_j$, and therefore $j_1\notin\{0,1,\ldots,2\mb{c}-\om_j-1\}$.
			
			Since $\cn{F}$ is disconnected, $j_2\nsim\om_j$ or $\om_j\nsim j_1$. Using $p+1\le \om_j<2\mb{c}-\om_j\le j_1<n-1=j_2$, we obtain $j_1<\om_j-2p\ (\text{mod $n$})\le n-1$. This implies $j_1\notin\{\om_j-2p\ (\text{mod $n$}), \om_j-2p+1\ (\text{mod $n$}),\ldots,n-1\}$.
			
			Hence $j_1\in V(C_n^p)\setminus\bigl(\{\om_j-2p\ (\text{mod $n$}), \om_j-2p+1\ (\text{mod $n$}),\ldots,n-1\}\sqcup\{0,1,\ldots,2\mb{c}-\om_j-1\}\bigr)$. Note that $\om_j\in\mcl{U}_1$ and $j_1\in V(C_n^p)\setminus\mcl{V}^{\om_j}$. Thus $F\in\Sg_1$. 
			
			\item $\om_j\in\{2p,2p+1,\ldots,\mb{c}-p\}$.
			
			Since $\om_j\le\mb{c}-p<\mb{c}$  and $\om_j\os j_1$, either $j_1<\om_j$ or $j_1\ge 2\mb{c}-\om_j$ by \Cref{remark:order in S}\ref{part:i}. Hence $j_1\notin\{\om_j,\om_j+1,\ldots,2\mb{c}-\om_j-1\}$. Moreover, $j_1<j_2=n-1$ implies $j_1\ne n-1$.
			Since $F$ is a spanning facet and $\om_j<\mb{c}$, by \Cref{proposition:F spanning facet then i_1 not in}\ref{om_i<=c}, $j_1\notin\{\om_j-2p,\om_j-2p+1,\ldots,\om_j-2\}$.
			
			Therefore, $j_1\in V(C_n^p)\setminus\bigl(\{\om_j-2p, \om_j-2p+1,\ldots, \om_j-2\}\sqcup\{\om_j,\om_j+1,\ldots,2\mb{c}-\om_j-1\}\sqcup\{n-1\}\bigr)$. Here, $\om_j\in\mcl{U}_2$ and $j_1\in V(C_n^p)\setminus\mcl{V}^{\om_j}$, which implies that $F\in\Sg_2$.
			
			\item  $\om_j\in\{\mb{c}-p+1,\mb{c}-p+2,\ldots,\mb{c}\}$.
			
			Since $F$ is a spanning facet and $\om_j\le\mb{c}$, $j_1\notin\{\om_j-2p,\om_j-2p+1,\ldots,\om_j-2\}$ by \Cref{proposition:F spanning facet then i_1 not in}\ref{om_i<=c}.
			
			We now prove that $j_1\ne\om_j-1$. Suppose $j_1=\om_j-1$. By \Cref{proposition: c and p relation}\ref{range of c}, $\mb{c}\le n-3p+2$. Hence $j_1=\om_j-1<\om_j\le\mb{c}<\mb{c}+p\le n-2p+2<n-1=j_2 \implies \mb{c}+p\in F$. 
			We have $(\Fuv{\mb{c}+p}{j_2})^c=\{\om_j,\om_j-1,\mb{c}+p\}$. If $\om_j=\mb{c}$, then $\om_j\os\mb{c}+p$ and $\Fuv{\mb{c}+p}{j_2}\notin\md$. By \Cref{proposition:F is not a spanning facet} (for $u=\mb{c}+p$), $F$ is not a spanning facet, a contradiction. Hence $\om_j<\mb{c}$.
			Since $\om_j>\mb{c}-p$, $2\mb{c}-\om_j<\mb{c}+p$. So $\om_j\os\mb{c}+p$ by \Cref{remark:order in S}\ref{part:i}. 
			Using the fact that $F$ is a spanning facet,  \Cref{proposition:F is not a spanning facet} yields $\Fuv{\mb{c}+p}{j_2}\in\md$ with $\Fuv{\mb{c}+p}{j_2}\in\mcl{M}_0$.      
			
			Note that $(\Fuv{\mb{c}+p}{j_2})^c=\{\om_j\}\sqcup\{\om_j-1,\mb{c}+p\}$ with $j_1=\om_j-1<\om_j<\mb{c}+p$.
			Since $\mb{c}-p+1\le\om_j\le\mb{c}-1$,  $j_1=\om_j-1$ implies $j_1+p+2\le\mb{c}+p\le j_1+2p$. This means that $\mb{c}+p=j_1+2p-\gm+1$ for some $\gm\in[p-1]$. 
			First, suppose $\om_j<\mb{c}-\frac{p}{2}$. Then $\Fuv{\mb{c}+p}{j_2}\notin\mcl{M}_0$ by \Cref{proposition:in M_bt+gm}\ref{satisfies x3}\ref{satisfies x3 i}, a contradiction.
			Now, suppose $\om_j\ge\mb{c}-\frac{p}{2}$. We have $\om_j-p<\om_j-1=j_1<\om_j$. Since $\mb{c}+p=j_1+2p-\gm+1$ with $\gm\in[p-1]$, if $j_1\le 2\mb{c}-\om_j-p-1$, then $\Fuv{\mb{c}+p}{j_2}\notin\mcl{M}_0$ by \Cref{proposition:in M_bt+gm}\ref{satisfies x4}\ref{satisfies x4 i}, a contradiction. 
			This implies $j_1\ge 2\mb{c}-\om_j-p$. Since $\mb{c}-p+1\le\om_j<\mb{c}$, we have $\om_j=\mb{c}-\gm$ for some $\gm\in[p-1]$. Moreover, $\mb{c}+p=2\mb{c}-(\mb{c}-\gm)+p-\gm=2\mb{c}-\om_j+p-\gm$. 
			By \Cref{proposition:in M_bt+gm}\ref{satisfies x5}\ref{satisfies x5 i}, $\Fuv{\mb{c}+p}{j_2}\notin\mcl{M}_0$, again a contradiction. Hence $j_1\ne\om_j-1$.  
			
			Since $\om_j\os j_1$, we have $j_1\ne\om_j$. By \Cref{proposition:F spanning facet then i_1 not in}\ref{i_1 not in om_i+1 to om_i+2p-1} and \ref{i_1 not=nu}, $j_1\notin\{\om_j+1,\om_j+2,\ldots,\om_j+p\}\sqcup\{\om_j+p+1\ (\text{mod \ $n$}),\om_j+p+2\ (\text{mod \ $n$}),\ldots,\om_j+2p-1\ (\text{mod \ $n$})\}\sqcup\{\nu\}$.
			
			Therefore, $j_1\notin\{\om_j-2p,\om_j-2p+1,\ldots,\om_j,\om_j+1,\ldots,\om_j+p\}\sqcup\{\om_j+p+1\ (\text{mod $n$}),\om_j+p+2\ (\text{mod $n$}),\ldots,\om_j+2p-1\ (\text{mod $n$})\}\sqcup\{\nu\}$. 
			Since $\om_j-2p\le\mb{c}-2p<\mb{c}\le 2\mb{c}-\om_j$ implies $\min\{2\mb{c}-\om_j,\om_j-2p\}=\om_j-2p$ and $\om_j\in\mcl{U}_3$, we have $j_1\in V(C_n^p)\setminus\mcl{V}^{\om_j}$. Hence $F\in\Sg_3$.
			
			\item $\om_j\in\{\mb{c}+1,\mb{c}+2,\ldots,n-p-1\}$.  
			
			In this case, $\om_j\in\mcl{U}_3$. Since $\om_j>\mb{c}$ and $\om_j\os j_1$, by \Cref{remark:order in S}\ref{part:ii},  $j_1\notin\{2\mb{c}-\om_j,2\mb{c}-\om_j+1,\ldots,\om_j\}$. 
			Moreover, by \Cref{proposition:F spanning facet then i_1 not in}\ref{i_1 not in om_i+1 to om_i+2p-1} and \ref{i_1 not=nu}, $j_1\notin\{\om_j+1,\om_j+2,\ldots,\om_j+p\}\sqcup\{\om_j+p+1\ (\text{mod \ $n$}),\om_j+p+2\ (\text{mod \ $n$}),\ldots,\om_j+2p-1\ (\text{mod \ $n$})\}\sqcup\{\nu\}$.
			Therefore, $j_1\in V(C_n^p)\setminus(\{2\mb{c}-\om_j,2\mb{c}-\om_j+1,\ldots,\om_j,\om_j+1,\ldots,\om_j+p\}\sqcup\{\om_j+p+1\ (\text{mod $n$}),\om_j+p+2\ (\text{mod $n$}),\ldots,\om_j+2p-1\ (\text{mod $n$})\}\sqcup\{\nu\})$.
			
			First, assume $\mb{c}+p\le\om_j\le n-p-1$. Then $2\mb{c}-\om_j\le\mb{c}-p\le\om_j-2p$, which implies $\min\{2\mb{c}-\om_j,\om_j-2p\}=2\mb{c}-\om_j$. 
			Since $\om_j\in\mcl{U}_3$, we have $j_1\in V(C_n^p)\setminus\mcl{V}^{\om_j}$. Hence $F\in\Sg_3$. 
			
			Now, assume $\mb{c}+1\le\om_j\le\mb{c}+p-1$. 
			Then $\om_j-2p\le\mb{c}-p-1<\mb{c}-p+1\le 2\mb{c}-\om_j$. Hence $\min\{2\mb{c}-\om_j,\om_j-2p\}=\om_j-2p$. Since $\om_j\in\mcl{U}_3$, we aim to prove that $F\in\Sg_3$. For this we need to show that $j_1\in V(C_n^p)\setminus\mcl{V}^{\om_j}$. Observe that it suffices to show that $j_1\notin\{\om_j-2p,\om_j-2p+1,\ldots,2\mb{c}-\om_j-1\}$. 
			
			\begin{itemize}
				\item Suppose $j_1\in\{\om_j-2p,\om_j-2p+1,\ldots,2\mb{c}-\om_j-2\}$. Since $\mb{c}\le n-3p+2$ by \Cref{proposition: c and p relation}\ref{range of c} and $\om_j>\mb{c}$, we get $j_1<j_1+1\le 2\mb{c}-\om_j-1<\mb{c}<\om_j\le\mb{c}+p-1\le n-2p+1<n-1=j_2$. This gives $j_1+1\in F$.
				By \Cref{remark:order in S}\ref{part:ii}, $\om_j>\mb{c}$ and $j_1+1<2\mb{c}-\om_j$ imply $\om_j\os j_1+1$.  
				Since $F$ is a spanning facet,  \Cref{proposition:F is not a spanning facet} (for $u=j_1+1$) yields $\Fuv{j_1+1}{j_2}\in\md$ and $\Fuv{j_1+1}{j_2}\in\mcl{M}_0$. We have $(\Fuv{j_1+1}{j_2})^c=\{\om_j\}\sqcup\{j_1,j_1+1\}$ with $j_1<j_1+1<\om_j$.
				Since $\Fuv{j_1+1}{j_2}\in\md$, it follows that $j_1\le \om_j-p-2$. Then $j_1\ge \om_j-2p$ implies $j_1=\om_j-2p+\gm-1$ for some $\gm\in[p-1]$. Therefore, if $\om_j\ge\mb{c}+\frac{p}{2}$, then $\Fuv{j_1+1}{j_2}\notin\mcl{M}_0$ by \Cref{proposition:in M_bt+gm}\ref{satisfies x2}\ref{satisfies x2 i}, a contradiction. So, assume $\om_j<\mb{c}+\frac{p}{2}$.
				
				Suppose $j_1\ge 2\mb{c}-\om_j-p$. Then $j_1\le 2\mb{c}-\om_j-2$ implies $j_1=2\mb{c}-\om_j-p+\gm-1$ for some $1\le\gm\le p-1$. 
				Therefore, \Cref{proposition:in M_bt+gm}\ref{satisfies x1}\ref{satisfies x1 i} yields $\Fuv{j_1+1}{j_2}\notin\mcl{M}_0$, a contradiction. Hence $j_1<2\mb{c}-\om_j-p$. 
				
				It follows that $j_1<j_1+p<2\mb{c}-\om_j<\om_j<n-1=j_2$, which gives $j_1+p\in F$. 
				We have $(\Fuv{j_1+p}{j_2})^c=\{\om_j,j_1,j_1+p\}$. Using $j_1<j_1+p<\om_j\le j_1+2p$ (as $j_1\ge\om_j-2p$), we obtain $j_1\sim (j_1+p)$ and $(j_1+p)\sim\om_j$. Hence $\Fuv{j_1+p}{j_2}\notin \md$. Moreover, $\om_j>\mb{c}$ implies $\om_j\os j_1+p$ by \Cref{remark:order in S}\ref{part:ii}. 
				By \Cref{proposition:F is not a spanning facet} (for $u=j_1+p$), $F$ is not a spanning facet, a contradiction. Therefore $j_1\notin\{\om_j-2p,\om_j-2p+1,\ldots,2\mb{c}-\om_j-2\}$.
				
				\item Suppose $j_1=2\mb{c}-\om_j-1$. By \Cref{proposition: c and p relation}\ref{range of c}, $\mb{c}\le n-3p+2$. It follows that $j_1=2\mb{c}-\om_j-1<\mb{c}<\om_j<\mb{c}+p\le n-2p+2<n-1=j_2$. Hence $\mb{c}+p\in F$. 
				Since $F$ is a spanning facet, $\Fuv{\mb{c}+p}{j_2}\in\md$ and $\Fuv{\mb{c}+p}{j_2}\in\mcl{M}_0$ by \Cref{proposition:F is not a spanning facet} (for $u=\mb{c}+p$).  
				Using $\mb{c}<\om_j<\mb{c}+p$, we get $\om_j\os\mb{c}+p$. This implies $(\Fuv{\mb{c}+p}{j_2})^c=\{\om_j\}\sqcup\{j_1,\mb{c}+p\}$ with $j_1<\om_j<\mb{c}+p$.
				Since $\mb{c}+1\le\om_j\le\mb{c}+p-1$ and $j_1=2\mb{c}-\om_j-1$, we have $\mb{c}-p\le j_1\le\mb{c}-2$, which implies $j_1+p+2\le\mb{c}+p\le j_1+2p$.
				This means that $j_1=(\mb{c}+p)-2p+\gm-1$, where $\gm\in[p-1]$. Therefore, if $\om_j\ge\mb{c}+\frac{p}{2}$, then $\Fuv{\mb{c}+p}{j_2}\notin\mcl{M}_0$ by \Cref{proposition:in M_bt+gm}\ref{satisfies x8}\ref{satisfies x8 i}, a contradiction.
				So, let $\om_j<\mb{c}+\frac{p}{2}$.
				Since $\om_j>\mb{c}$, we get $2\mb{c}-\om_j+p<\mb{c}+p<\om_j+p$. We have $\om_j<\mb{c}+p$, and $j_1=(\mb{c}+p)-2p+\gm-1$, where $\gm\in[p-1]$.  
				By \Cref{proposition:in M_bt+gm}\ref{satisfies x7}\ref{satisfies x7 i}, $\Fuv{\mb{c}+p}{j_2}\notin\mcl{M}_0$, again a contradiction. Hence $j_1\ne2\mb{c}-\om_j-1$.
			\end{itemize}
		\end{enumerate}
		In all four cases, $F\in\Sg_m$ for some $m\in[3]$. Hence $F\in\Sg$.
	\end{proof}
	
	\subsubsection{Counting of Spanning Facets}\label{subsubsection:counting}
	From \Cref{lemma: if F in Sigma,lemma: then F in Sigma}, a facet $F\in \md$ is spanning for the shelling order $\prec$ on $\Delta_{3}(C_n^p)$ if and only if $F\in\Sg$. Since $\Sg=\Sg_1\sqcup\Sg_2\sqcup\Sg_3$, the total number of spanning facets is $|\Sg_1|+|\Sg_2|+|\Sg_3|$.
	
	Recall that for $m\in[3]$, a facet $F \in \Sg_m$ if and only if $F\in\mcl{A}_i$ for some $i\in[n-2]$ such that $F^c=\{\om_i\} \sqcup \{i_1, n-1\}$, $\om_i \in \mcl{U}_m$ and $i_1 \in V(C_p^n) \setminus \mcl{V}^{\om_i}$.
	\begin{enumerate}[label=(\roman*)]
		\item Let $F\in\Sg_1$. Then $\om_i \in\mcl{U}_1=\{p+1, p+2, \dots, 2p-1\}$ and  $\mcl{V}^{\om_i} =\{\om_i-2p \ (\text{mod} \ n), \om_i-2p+1 \ (\text{mod} \ n), \dots, n-1\} \sqcup \{0, 1, \dots, 2\mb{c}-\om_i-1\}$.
		Since $\om_i\le 2p-1$, we have $\{\om_i-2p\ (\text{mod $n$}), \om_i-2p+1\ (\text{mod $n$}),\ldots,n-1\}=\{n+\om_i-2p, n+\om_i-2p+1,\ldots,n-1\}$. It follows that $|\mcl{V}^{\om_i}|=(2p-\om_i)+(2\mb{c}-\om_i)=2\mb{c}+2p-2\om_i$ and hence $|V(C_p^n) \setminus \mcl{V}^{\om_i}|=n-(2\mb{c}+2p-2\om_i)$. Moreover, $|\mcl{U}_1|=p-1$. This implies
		\begin{align*}
			|\Sg_1|&=(p-1)(n-2\mb{c}-2p)+\left(2\textstyle\sum\limits_{\om_i=p+1}^{2p-1}\om_i\right)\\
			&=(p-1)(n-2\mb{c}-2p)+(2p-1)(2p)-p(p+1)=p^2+np-2\mb{c}p-p-n+2\mb{c}.
		\end{align*}
		
		\item Let $F\in\Sg_2$. Then $\om_i \in\mcl{U}_2=\{2p, 2p+1, \dots, \mb{c}-p\}$ and $\mcl{V}^{\om_i}=\{\om_i-2p, \om_i-2p+1, \dots, \om_i-2\} \sqcup \{\om_i, \om_i+1, \dots, 2\mb{c}-\om_i-1\} \sqcup \{n-1\}$.
		We have $|\mcl{U}_2|=\mb{c}-3p+1$ and $|V(C_p^n) \setminus \mcl{V}^{\om_i}|=n-(2\mb{c}+2p-2\om_i)$. Hence 
		\begin{align*}
			|\Sg_2|&=(\mb{c}-3p+1)(n-2\mb{c}-2p)+\left(2\textstyle\sum\limits_{\om_i=2p}^{\mb{c}-p}\om_i\right)\\
			&=(\mb{c}-3p+1)(n-2\mb{c}-2p)+(\mb{c}-p)(\mb{c}-p+1)-(2p-1)(2p)\\
			&=3p^2-\mb{c} ^2+\mb{c}n+2\mb{c}p-3np+n-p-\mb{c}.
		\end{align*}
		
		\item Let $F\in\Sg_3$. Then $\om_i\in\mcl{U}_3=\{\mb{c}-p+1, \mb{c}-p+2, \dots, n-p-1\}$ and $\mcl{V}^{\om_i}=\{\mu, \mu+1, \dots, \om_i, \om_i+1, \dots, \om_i+p\} \sqcup \{\om_i+t\ (\text{mod $n$})|\ p+1\le t\le 2p-1\} \sqcup \{\nu\}$, 
		where $\mu=\min\{2\mb{c}-\om_i, \om_i-2p\}$ and $\nu=\begin{cases} 
			n-1 & \text{if } \om_i<n-2p-1, \\ 
			\om_i+2p \ (\text{mod} \ n) & \text{if } \om_i \ge n-2p-1. 
		\end{cases}$
		
		If $\mb{c}-p+1\le\om_i\le\mb{c}+p-1$, then there are $2p-1$ possibilities for $\om_i$, and $\mu=\om_i-2p$. This implies $|V(C_p^n) \setminus \mcl{V}^{\om_i}|=n-(4p+1)$. 
		If $\mb{c}+p\le\om_i\le n-p-1$, then there are $n-\mb{c}-2p$ possibilities for $\om_i$, and $\mu=2\mb{c}-\om_i$, which implies $|V(C_p^n) \setminus \mcl{V}^{\om_i}|=n-(2\om_i-2\mb{c}+2p+1)$. Therefore 
		\begin{align*}
			|\Sg_3|&=(2p-1)(n-4p-1)+(n-\mb{c}-2p)(n+2\mb{c}-2p-1)-\left(2\textstyle\sum\limits_{\om_i=\mb{c}+p}^{n-p-1}\om_i\right)\\
			&=(2p-1)(n-4p-1)+(n-\mb{c}-2p)(n+2\mb{c}-2p-1)-(n-p-1)(n-p)\\
			&\quad+(\mb{c}+p-1)(\mb{c}+p)\\
			&=\mb{c}n-\mb{c} ^2-4p^2-n+2p+1.
		\end{align*}
	\end{enumerate}
	
	Total number of spanning facets for the shelling order $\prec$
	\begin{align*}
		&=|\Sg_1|+|\Sg_2|+|\Sg_3|\\    
		&=(p^2+np-2\mb{c}p-p-n+2\mb{c})+(3p^2-\mb{c} ^2+\mb{c}n+2\mb{c}p-3np+n-p-\mb{c} )\\&\quad+ (\mb{c}n-\mb{c} ^2-4p^2-n+2p+1)\\
		&=2\mb{c} n-2\mb{c}^2-2np+ \mb{c}-n+1=\frac{n^2-4np-n+2}{2}
		=\binom{n-2p}{2}-(2p^2+p-1).
	\end{align*}
	
	\begin{proof}[Proof of \Cref{theorem:main}]
		In \Cref{subsection:shelling order}, we established that $\prec$ is a shelling order on $\Delta_3(C_n^p)$ for $p\ge2$ and $n\ge 6p-3$. Therefore, $\Delta_{3}(C_n^p)$ is shellable. 
		In \Cref{subsubsection:counting}, we determined that $\Delta_3(C_n^p)$ has exactly $\binom{n-2p}{2}-(2p^2+p-1)$ spanning facets with respect to the shelling order $\prec$. 
		
		By \Cref{theorem:wedge}, $\Delta_3(C_n^p) \simeq \bigvee\limits_{\binom{n-2p}{2}-(2p^2+p-1)} \mathbb{S}^{n-4}.$
	\end{proof}
	
	\section{Conclusion and Future Directions}\label{section:future_directions}
	
	Bravo \cite{Andres_total_cut_dual} proved that the complex $\Delta_2(C_n^p)$ is not shellable. In \cite[Conjecture 4.2]{Shellability2025}, the authors conjectured that $\Delta_3(C_n^p)$ is  not shellable for $2p+3 \leq n \leq 4p$ (for $n \leq 2p+2$, $\Delta_3(C_n^p)$ is void), and shellable for $n \geq 4p+1$.
	In this article, we proved that for $p\ge 2$ and $n\ge6p-3$, $\Delta_3(C_n^p)$ is shellable. 
	
	Furthermore, based on SageMath computations for $k=4,5,6$, the authors in \cite{Shellability2025} raised the question of whether the complexes $\Delta_k(C_n^p)$, if nonvoid, are not shellable for $k\ge 4$ and $p\ge 3$.
	
	Since $C_n^p$ is a circulant graph on the generating set $\{1,2,\ldots,p\}$ and hence a Cayley graph of $\mathbb{Z}_n$, the following questions naturally arise.
	\begin{question}
		What can be said about the shellability of $3$-cut complexes of general circulant graphs? 
	\end{question}
	
	\begin{question}
		For which classes of Cayley graphs are the cut complexes shellable? 
	\end{question}
	
	\section*{Acknowledgement}
	The first author is supported by HTRA fellowship by IIT Mandi, India. The second author is supported by the seed grant project IITM/SG/SMS/95 by IIT Mandi, India.
	
	\bibliographystyle{abbrv}
	\bibliography{ref_New}
	
	\addcontentsline{toc}{section}{References}
	
\end{document}